\documentclass[12pt,a4paper]{amsart}
\makeatletter
\renewcommand\normalsize{%
    \@setfontsize\normalsize{11.7}{14pt plus .3pt minus .3pt}%
    \abovedisplayskip 10\p@ \@plus4\p@ \@minus4\p@
    \abovedisplayshortskip 6\p@ \@plus2\p@
    \belowdisplayshortskip 6\p@ \@plus2\p@
    \belowdisplayskip \abovedisplayskip}
\renewcommand\small{%
    \@setfontsize\small{9.5}{12\p@ plus .2\p@ minus .2\p@}%
    \abovedisplayskip 8.5\p@ \@plus4\p@ \@minus1\p@
    \belowdisplayskip \abovedisplayskip
    \abovedisplayshortskip \abovedisplayskip
    \belowdisplayshortskip \abovedisplayskip}
\renewcommand\footnotesize{%
    \@setfontsize\footnotesize{8.5}{9.25\p@ plus .1pt minus .1pt}
    \abovedisplayskip 6\p@ \@plus4\p@ \@minus1\p@
    \belowdisplayskip \abovedisplayskip
    \abovedisplayshortskip \abovedisplayskip
    \belowdisplayshortskip \abovedisplayskip}
\ifdefined\pdfpagewidth
\else
\fi
\calclayout
\makeatother

\usepackage{amsmath,amssymb}
\usepackage{mathrsfs}
\usepackage{comment}

\usepackage[english]{babel}
\usepackage[utf8x]{inputenc}
\usepackage[T1]{fontenc}
\usepackage{tikz-cd}

\usepackage[top=3cm,bottom=2cm,left=3cm,right=3cm,marginparwidth=1.75cm]{geometry}
\usepackage{graphicx}
\usepackage[colorinlistoftodos]{todonotes}
\usepackage[colorlinks=true, allcolors=blue]{hyperref}
\usepackage{cite}
\usepackage{enumitem}
\setlist[1]{itemsep=3pt}
\usepackage[all,cmtip]{xy}

\usepackage[capitalise]{cleveref}
\usepackage{autonum}
\usepackage{setspace}

\newcommand{\beq}{\begin{equation}}
\newcommand{\eeq}{\end{equation}}

\usepackage[colorinlistoftodos]{todonotes}

\theoremstyle{plain}
\newtheorem{teorema}{teorema}[section]
\newtheorem{theorem}[teorema]{Theorem}

\newtheorem{lemma}[teorema]{Lemma}
\newtheorem{proposition}[teorema]{Proposition}
\newtheorem{corollary}[teorema]{Corollary}

\theoremstyle{definition}
\newtheorem{definition}[teorema]{Definition}
\newtheorem{remark}[teorema]{Remark}

\numberwithin{equation}{section}

\let\oldtocsection=\tocsection
\let\oldtocsubsection=\tocsubsection
\let\oldtocsubsubsection=\tocsubsubsection
\renewcommand{\tocsection}[2]{\hspace{-1.2em}\oldtocsection{#1}{#2}}
\renewcommand{\tocsubsection}[2]{\hspace{-.2em}\oldtocsubsection{#1}{#2}}
\renewcommand{\tocsubsubsection}[2]{\hspace{0.8em}\oldtocsubsubsection{#1}{#2}}

\DeclareRobustCommand{\gobblefive}[5]{}

\makeatletter
\renewcommand\subsubsection{\@startsection{subsubsection}{3}%
  \z@{.5\linespacing\@plus.7\linespacing}{-.5em}%
  {\normalfont\bfseries}}
\makeatother

\newcommand{\R}{\mathbb{R}}
\newcommand{\N}{\mathbb{N}}

\newcommand{\supp}{\text{\rm supp}}

\newcommand{\C}{\mathbb{C}}

\newcommand{\dd}{\mathrm{d}}
\newcommand{\clos}{\text{\rm clos}}

\newcommand{\ve}{\varepsilon}

\renewcommand{\L}{\mathcal{L}}

\newcommand{\Geo}{{\rm Geo}}
\newcommand{\Dir}{{\rm Dir}}

\newcommand{\OptGeo}{{\rm OptGeo}}

\newcommand{\Tan}{{\rm Tan}}
\newcommand{\sfd}{\mathsf d}

\newcommand{\ee}{{\rm e}}
\newcommand{\ddiv}{{\rm div}}
\newcommand{\en}{{\mathcal{E}}}

\newcommand{\h}{{\mathsf{h}}}

\newcommand{\CP}{\mathbb{C}\mathrm{P}}
\newcommand{\rest}[1]{\big\rvert_{#1}} 

\title[]{Optimal transport between algebraic hypersurfaces}
\author{Paolo Antonini}
\address{Dipartimento di Matematica e Fisica "E. De Giorgi", Universit\`a del Salento, Lecce (Italy)
}
\email{paolo.antonini@unisalento.it}
\author{Fabio Cavalletti} \address{Mathematics Area, SISSA, Trieste (Italy)}
\email{cavallet@sissa.it}
\author{Antonio Lerario}
\address{Mathematics Area, SISSA, Trieste (Italy)}
\email{lerario@sissa.it}

\begin{document}
\maketitle
\begin{abstract}What is the optimal way to deform a projective hypersurface into another one? In this paper we will answer this question adopting the point of view of measure theory, introducing the optimal transport problem between complex algebraic projective hypersurfaces. 

First, a natural topological embedding of the space of hypersurfaces of a given degree into the space of measures on the projective space is constructed. Then, the optimal transport problem between hypersurfaces is defined through a constrained dynamical formulation, minimizing the energy of absolutely continuous curves which lie on the image of this embedding. In this way an inner Wasserstein distance on the projective space of homogeneous polynomials is introduced. This distance is finer than the Fubini--Study one.

The innner Wasserstein distance is  complete and geodesic: geodesics corresponds to optimal deformations of one algebraic hypersurface into another one. Outside the discriminant this distance is induced by a smooth Riemannian metric, which is the real part of an explicit Hermitian structure. Moreover, this hermitian structure is K\"ahler and the corresponding metric is of Weil--Petersson type. 

To prove these results we develop new techniques, which combine complex and symplectic geometry with optimal transport, and which we expect to be relevant on their own.

We discuss applications on the regularity of the zeroes of a family of multivariate polynomials and on the condition number of polynomial systems solving.

\end{abstract}

\setcounter{tocdepth}{1}

\tableofcontents
\smallskip
\section{Introduction}\label{sec:intro}

What is the optimal way to deform an algebraic hypersurface in the $n$--dimensional projective space into another one? In this paper we will answer this question adopting the point of view of measure theory, associating to an algebraic hypersurface a probability measure on the projective space and formulating the problem with the language of \emph{optimal transport}. 

Optimal Transport  is nowadays a pillar of modern mathematics. In a nutshell, it is a general tool to measure how expensive the transportation between two given probability measures $\mu_{0},\mu_{1}$
over a fixed  metric space $(M,\sfd)$ is. This is done by a minimization procedure:
for $q\geq 1$, the $q$--optimal transport problem between $\mu_0$ and $\mu_1$ is defined as
\beq
\label{eq:Wintrodef}W_q(\mu_0, \mu_1):=
\inf_{\xi \in \Pi(\mu_{0},\mu_{1})} \left(\int_{M\times M} \sfd(x,y)^q\,\xi(\dd x \dd y)\right)^{\frac{1}{q}},
\eeq
where $\Pi(\mu_{0},\mu_{1})$ consists of the set of probability measures $\xi$ over $M\times M$ with marginals $\mu_0$ and $\mu_1$ (see \cref{sec:wasserstein}). The elements of $\Pi(\mu_0, \mu_1)$ are called \emph{transport plans}, capturing the idea of transporting $\mu_0$ to $\mu_1$. The quantity in \eqref{eq:Wintrodef} defines indeed a distance, the \emph{$q$--Wasserstein distance}, on the set of Borel probability measures on $M$. The resulting metric space is denoted by $\mathscr{P}_q(M)$ and called the \emph{$q$--Wasserstein space}. When $(M, \sfd)$ is compact and geodesic, $\mathscr{P}_q(M)$ is also compact and geodesic: in this case the optimal transport problem \eqref{eq:Wintrodef} has a natural interpretation as a geodesic problem. 

Thanks to the generality of this formulation, during the last 30 years, Optimal Transport and the Wasserstein distance have played a central role in a large variety of different areas of mathematics.
To list few notable examples we mention:  PDEs \cite{JKO}, geometric inequalities \cite{FMP},  Random matrices \cite{FigalliGuionnet} and  
Quantum mechanics \cite{CarlenMaas}. 
For a comprehensive introduction to the subject we refer to the monographs \cite{villani:topics,villani:oldandnew}. 			
In differential geometry, recently Optimal Transport has been at the core of the synthetic treatment of lower Ricci curvature bounds,  culminated with the successful theory of J. Lott, K.-T. Sturm and C. Villani \cite{lottvillani,SturmI, SturmII} completing Gromov's program of
describing  lower bounds on the Ricci curvature in an intrinsic manner for metric spaces.

In this paper we will take $(M, \sfd)=(\CP^n, \sfd_\mathrm{FS})$, the $n$--dimensional complex projective space with the \emph{Fubini--Study} distance (see \cref{Appendix A}). We will treat a hypersurface of degree $d$ as a measure on $\CP^n$: for a smooth hypersurface this measure is given by integrating continuous functions on it, but some care has to be done to extend this definition to singular hypersurfaces, see \cref{sec:intromeasure} below. In the main part of the paper we will study the intrinsic Wasserstein geometry of the space of hypersurfaces, building the tools for answering the question posed at the beginning of the paper, establishing connections between optimal transport, symplectic geometry and singularity theory. In the final part we illustrate how to apply our techniques and formalism: first we generalize to higher dimensions the result of A. Parusinski and A. Rainer on the sharp regularity of roots of families of polynomials. We use this result to prove the compactness of the space of hypersurfaces with the inner $q$--Wasserstein distance for $q\in [1, q_0]$ with $q_0>1$. Then we propose a solution to a problem posed by C. Beltr\'an, J.-P. Dedieu, G. Malajovich and M. Shub on the condition length of a certain class of curves. 

Finally, we expect our techniques to be applicable beyond the context of complex algebraic geometry, for example in the more flexible framework of symplectic geometry. 

\subsection{Optimal transport between algebraic hypersurfaces}\label{example:one}As a warm up we first illustrate some of the ideas in the case $n=1$.

Let $\C^d$ denote the space of coefficients of monic polynomials $p(z)$ of degree $d$ in one complex variable, $p(z)=z^d+c_{1}z^{d-1}+\cdots +c_{d-1}z+c_d$, $c_k\in \C$. To each polynomial $p\in \C^d$ we can associate in a natural way a probability measure $\mu(p)\in \mathscr{P}(\C)$ defined by
\beq \label{eq:deltaintro}\mu(p):=\frac{1}{d}\sum_{p(z)=0}m_z(p)\cdot\delta_z,\eeq
where $m_z(p)\in \N$ denotes the multiplicity of $z$ as a zero of $p$. This defines a map
\beq\label{eq:solmap} \mu:\C^d\to \mathscr{P}_2(\C).\eeq
We can think of $\mu$ as a ``solution map'': given a polynomial $p$, this map gives all the solutions to $p(z)=0$, condensed into a single object, the measure $\mu(p)$. The image of $\mu$ can be identified with $\mathrm{SP}^d(\C)$, the $d$--th symmetric product of $\C$, i.e. the quotient of $\C^d$ under the action of the symmetric group $S_d$ by permutations of the factors. The Fundamental Theorem of Algebra implies that $\mu$ is in fact a homeomorphism onto its image. It turns out that the restriction of the $W_2$--metric to $\mathrm{SP}^d(\C)\subset \mathscr{P}(\C)$  is given by
\beq 
\label{eq:W21}W_2([x], [y])^2 =\min_{\sigma\in S_d} \frac{1}{d}\sum_{j=1}^d|x_j-y_{\sigma(j)}|^2,
\eeq
where $x=(x_1, \ldots, x_d)$ and $y=(y_1, \ldots, y_d)$,
with a similar expression for $q \neq2$. 

For every $q\geq 1$, $\mathrm{SP}^d(\C)$ is geodesically convex in $\mathscr{P}_{q}(\C)$.
Geodesics of 
$(\mathrm{SP}^d(\C),W_2)$ 
have a natural description:
once we have the optimal matching between
$[x]$ and $[y]$ (obtained
minimizing the sum of the square of their mutual distances in $\C$) 
then we let evolve the point 
$x_i$ along the segment joining 
$x_i$ to the optimal $y_{j_{i}}$.
As we will see, these geodesics have some special behavior with respect to the \emph{discriminant}, the set of polynomials with multiple roots. For instance if both $p_0$ and $p_1$ have simple roots, along the $2$--Wasserstein geodesic joining $\mu(p_0)$ with $\mu(p_1)$ the roots remain distinct. Optimal Transport allows to make this statement quantitative (\cref{thm:P14}). We will come back to this result later, in the context of condition numbers of polynomial system solving.

\subsubsection{The measure associated to an algebraic hypersurface in the projective space} \label{sec:intromeasure}
In higher dimension the situation is much more intricate.
Given natural numbers $n, d\in \N$, we denote by $H_{n,d}$ the space of complex, \emph{homogeneous} polynomials $p(z)$ in $n+1$ variables and of degree $d$,
$$
p(z_0, \ldots, z_n)=\sum_{|\alpha|=d}c_\alpha z_0^{\alpha_0} \cdots z_n^{\alpha_n}, \quad c_\alpha\in \C.
$$
Every nonzero polynomial $p\in H_{n,d}$, being homogeneous, defines a (possibly very singular) complex hypersurface $Z(p)$ in the complex projective space $\CP^n$,
$$
Z(p):=\{[z]\in \CP^n\,|\, p(z)=0\}\subset \CP^n.
$$

Since nonzero multiples of the same polynomial define the same hypersurface, it is more convenient to work with the projectivization $P_{n,d}$ of the space  $H_{n,d}$. This is itself a complex projective space
$$
P_{n,d}:=\mathrm{P}(H_{n,d})\simeq \CP^N, \quad N={n+d\choose d}-1.
$$
Elements of $P_{n,d}$ are therefore polynomials up to nonzero multiples but, in order to simplify the exposition, we will often call them ``polynomials'' and use the notation ``$p\in P_{n,d}$'' (it would be more precise to write ``$[p]\in P_{n,d}$'').

With this notation, we denote by $\Delta_{n,d}\subset P_{n,d}$ the discriminant, i.e. the set of polynomials $p\in P_{n,d}$ such that the equation $p=0$   \emph{is not} regular on $\CP^n$:
$$
\Delta_{n,d}:=\left\{p\in P_{n,d}\,\bigg|\, \exists z\in \C^{n+1}\setminus \{0\}, \, p(z)=\frac{\partial p}{\partial z_0}(z)=\cdots=\frac{\partial p}{\partial z_n}(z)=0\right\},
$$
If $p\in P_{n,d}\setminus \Delta_{n,d}$, then $Z(p)$ is a smooth complex submanifold of $\CP^n$, of real dimension $2n-2$. 

When $n=1$, this construction coincides with the one from \cref{example:one},  except that we view  our zeroes in the compact space $\CP^1$, rather than in $\C$. When $n>1$, algebraic hypersurfaces in $\C^n$ are never compact and it is more natural to work with $\CP^n$ as ambient space. 

As we have seen above, in the one--dimensional case we can treat the zero set of a nonzero polynomial $p$ as a sum of Dirac's delta measures, as in \eqref{eq:deltaintro}. The multi--dimensional analogue of this construction is the following.

Let us first look at the case $p\in P_{n,d}\setminus \Delta_{n,d}$. Restricting the Fubini--Study Riemannian metric of $\CP^n$ to $Z(p)$ turns it into an oriented Riemannian manifold, with volume form $\mathrm{vol}_{Z(p)}$. 
We define the associated measure $\mu(p)\in \mathscr{P}(\CP^n)$ to be the volume measure of  the smooth submanifold $Z(p)$, normalized to total mass $1$:
\beq
\label{eq:measureintro}\mu(p):=\frac{1}{\mathrm{vol}(Z(p))}\mathrm{vol}_{Z(p)}.
\eeq
Here we are tacitly agreeing that 
a Borel probability measure on $Z(p)$ 
is also a Borel probability measure on $\CP^n$.

Notice that all smooth hypersurfaces of degree $d$ have the same volume $\mathrm{vol}(Z)=d\cdot \mathrm{vol}(\CP^{n-1})$ (when $n=1$, this reduces to $d\cdot\mathrm{vol}(\CP^{0})=d$), 
so that the renormalization factor does not depend on $p$; this is a standard fact in complex algebraic geometry, see \cref{lemma:dis}.

Hence the correspondence $p\mapsto \mu(p)$ defines a ``family'' of probability measures on $\CP^n$, depending on the ``parameter'' $p\in P_{n,d}\setminus\Delta_{n,d}$. It is not difficult to show that $\mu(p)$ depends continuously on $p\in P_{n,d}\setminus\Delta_{n,d} $. Less trivial is the fact that this map can be continuously extended to the whole $P_{n,d}$, i.e. also on the discriminant. For instance, the Hausdorff measure of $Z(p)$ fails to be continuous even for $n=1$.
Moreover it is in general not true that similarly defined maps can be continuously extended also on discriminants, even in the simplest cases (see \cref{remark:notC}). 

For the continuous extension we will use the notion of multiplicity, together with an integral geometric argument.
We denote by $\mathbb{G}(1,n)$ the Grassmannian of lines in $\CP^n$, endowed with the uniform probability measure. For every $\ell\in \mathbb{G}(1,n)$, the restriction $p|_\ell$ is a polynomial in one variable, and one can talk about the multiplicity $m_z(p|_\ell)$ of its zeroes $z\in \ell\cap Z(p)$.

\begin{theorem}[The measure associated to an algebraic hypersurface]\label{thm:extintro}
For every $q\geq 1$, the map  $\mu:P_{n,d} \to \mathscr{P}_q(\CP^n)$ defined for $f\in \mathscr{C}^0(\CP^n)$ by
\beq
 \int_{\CP^n}f \dd \mu(p):=\frac{1}{d}\int_{\mathbb{G}(1,n)}\left(\sum_{z\in Z(p)\cap \ell}m_z(p|_{\ell}) f(z)\right)\mathrm{vol}_{\mathbb{G}(1,n)}(\dd \ell)
\eeq
is continuous and injective; restricted to $P_{n,d}\setminus \Delta_{n,d}$, it coincides with \eqref{eq:measureintro}. 
\end{theorem}
\subsubsection{The Wasserstein distance between algebraic hypersurfaces} 
At this point, using the map $\mu$ from \cref{thm:extintro}, we have associated to each element of  $P_{n,d}$ a unique element of $\mathscr{P}(\CP^n)$, getting in this way the $n$--dimensional analogue of the ``solution map'' \eqref{eq:solmap}, 
 $$
 \{\mathrm{coefficients}\}=P_{n,d}\stackrel{\mu}{\longrightarrow} \mu(P_{n,d})=\{\mathrm{zeroes}\}\subset \mathscr{P}(\CP^n).
 $$

Our goal now is to minimize the cost  \eqref{eq:Wintrodef}, when moving one hypersurface to another one \emph{within} the (finite--dimensional) space of hypersurfaces of a given degree. However, unless $n=1$, the image of $\mu$ \emph{is not} geodesically convex for the $W_q$--metric and, given $p_0, p_1\in P_{n,d}$, the geodesic joining $\mu(p_0)$ to $\mu(p_1)$ in $\mathscr{P}_q(\CP^n)$ will in general not stay on $\mu(P_{n,d})$. To overcome this issue and to formulate the optimal transport problem between complex hypersurfaces, we adopt a ``constrained'' version of the dynamic formulation of 
J.-D. Benamou and Y. Brenier \cite{BenamouBrenier}. 

To this end, recall first that a curve $\mu_t:[0,1]\to \mathscr{P}_q(\CP^n)$ is said to be in $AC^q(I, \mathscr{P}_q(\CP^n))$ if for almost every $t\in I$
$$\exists \lim_{h\to 0}\frac{W_q(\mu_{t+h}, \mu_t)}{|h|}=:|\dot\mu_t|\quad\textrm{and}\quad \int_{I}|\dot\mu_t|^q\dd t<\infty.$$
The function $|\dot\mu_t|$ is called the \emph{metric speed} of $\mu_t$ (notice that this notion makes sense for curves with values in general metric spaces, see \cref{Ss:metricgeometry}). 

For the rest of this section we focus on the case $q=2$; similar definitions can be given for $q\geq 1, q\neq2$. With the above notation, we define
the set of \emph{admissible curves}
\beq 
\Omega_{n,d}:=\Big{\{}\gamma:I\to P_{n,d}\,|\, \mu\circ \gamma\in \mathrm{AC}^2(I, \mathscr{P}_2(\CP^n))\Big{\}}.
\eeq 

If $\gamma\in \Omega_{n,d}$ is an admissible curve, we write $ \mu_t:=\mu(\gamma(t))$ and define the \emph{Energy} of $\gamma$ by
\beq \label{eq:mini}\en(\gamma):=\int_{I}|\dot\mu_t|^2\dd t.\eeq
Given $p_0, p_1\in P_{n,d}$ we denote by $\Omega_{n,d}(p_0, p_1)\subset \Omega_{n,d}$ the set of admissible curves $\gamma:I\to P_{n,d}$ such that $\gamma(0)=p_0, \gamma(1)=p_1$ and we 
define 
\beq\label{eq:bb}
W_2^\textrm{in}(p_{0},p_{1}): =   \inf_{\gamma\in \Omega_{n,d}(p_0, p_1)}\en(\gamma)^{\frac{1}{2}}.
\eeq
The superscript in \eqref{eq:bb} refers to the fact that $W_2^\mathrm{in}$ is the \emph{inner} metric inherited on $P_{n,d}$ from its embedding in $\mathscr{P}_2(\CP^n)$ via \cref{thm:extintro}, not to be confused with the restriction of the ambient $W_2$--metric, called the \emph{outer} metric.

The definition of inner metrics via the minimization of the Energy of  admissible curves, mimicking the familiar construction from Riemannian geometry,  is quite natural in metric geometry, where it is used to build a geodesic distance on non--geodesically convex subsets \cite{Burago}. To obtain therefore an insight on $W_2^\textrm{in}$ one needs to study the metric speed of admissible curves, a problem which reduces to studying the infinitesimal behaviour of the $W_{2}$ distance between two algebraic hypersurfaces.  For this  we will need a novel combination of techniques from complex, symplectic and metric geometry, which we explain in the next section. 

We note that, though being quite natural to address, the general optimal transport problem between measures supported on submanifolds
in an arbitrary Riemannian manifold seems to have a short bibliography. 
Gangbo--McCann \cite{GangboMcCann:shape} considered optimal transport for measures supported on hypersurfaces in the Euclidean space;
McCann--Sosio 
\cite{McCannSosio} and Kitagawa--Warren \cite{KitWarren} gave more refined results about optimal transport between measures 
supported on a codimension one sphere in Euclidean spaces.
Castillon \cite{Castillon} considered optimal transport between a measure supported on a submanifold of Euclidean spaces and a measure supported on a linear subspace.
Lott \cite{Lott17} characterized the tangent cone (in the $W_{2}$--metric) to a probability measure supported on a smooth submanifold of a Riemannian manifold.
Finally Ketterer--Mondino \cite{KettererMondino} linked the convexity properties of entropy functionals along $W_{2}$--geodesics concentrated on rectifiable submanifold to bounds on the partial traces of the Riemann tensor. None of these contributions constrain the paths along which the measures are allowed to move.

\subsubsection{The metric space $(P_{n,d},W^{\textrm{in}}_{2} )$} 
To see that the infimum in \eqref{eq:bb} is finite for every pair of points $p_0, p_1\in P_{n,d}$, and in particular that it defines a distance, 
requires a certain amount of work. We overview the main ideas in this section. 

We will use the following notation: whenever a curve $\gamma: I \to P_{n,d}$ will be given, we write $p_t : = \gamma(t)$. With a slight abuse of notation, we sometimes write $p_t:I\to P_{n,d}$ instead of $\gamma$.

  First assume $p_0, p_1\notin \Delta_{n,d}$. Since the discriminant has real codimension $2$, we can find a smooth curve $\gamma:I\to P_{n,d}\setminus \Delta_{n,d}$ joining $p_0$ and $p_1$. For this curve, all the hypersurfaces $Z(p_t)$ are embedded in $\CP^n$ in the same way (they are ``isotopic'' in the language of differential topology). This a consequence of the  \emph{Thom's Isotopy Lemma}: under the assumption $p_t\in P_{n,d}\setminus \Delta_{n,d}$, there exists a smooth family of diffeomorphisms $\varphi_t:\CP^n\to \CP^n$ such that $\varphi_t(Z(p_0))=Z(p_t)$ for every $t\in I$. There is also a Hamiltonian version of Thom's Isotopy Lemma, which in our case takes the following form (variations of this result can be found in \cite{zbMATH05286973,zbMATH01123717,sieberttian}, for a self--contained detailed proof see \cite{moser}).

 \begin{theorem}[Hamiltonian Thom's Isotopy Lemma]\label{thm:thomintro}Let $\gamma:I\to P_{n,d}\setminus \Delta_{n,d}$ be a smooth curve. Then there exists a Hamiltonian flow $\varphi_t:\CP^n\to \CP^n$ such that $\varphi_t(Z(p_0))=Z(p_t)$ for every $t\in I$.
 \end{theorem}
 
Using \cref{thm:thomintro}, we can write the measures $\mu_t:=\mu(p_t)$ as $\mu_t=(\varphi_t)_{\#}\mu_0$. In general, if $v_{t} \in L^{2}(\mu_{t};TM)$ is a time--dependent vector field solving (in the sense of distributions) the continuity equation 
\beq \label{eq:continuity}
\partial_{t} \mu_{t} + \ddiv(v_{t}\mu_{t}) = 0, 
\eeq
then, for almost every $t\in I$, the metric speed of $\mu_t$ equals:
\begin{equation}\label{speednormalintro}
|\dot{\mu_{t}} |=\| \mathsf{P}(v_{t}) \|_{L^{2}_{\mu_{t}}},\end{equation}
where $\mathsf{P}:L^2(\mu; TM)\to T_\mu(\mathscr{P}_2(M))$ denotes the projection on the  ``space of gradients'' (see \cref{sec:spaceofgrad}). 
In our case the crucial point is that, since $v_t$ is Hamiltonian,  for every $z\in Z(p_t)$ we have $\mathsf{P}(v_t)(z)=N(v_t(z))$, where $N$ denotes the orthogonal projection of the field on the normal bundle of $Z(p_t)$, see \cref{propo:ka}.

Summing up: if $\gamma:I\to P_{n,d}\setminus \Delta_{n,d}$ is a smooth curve and $v_t$ is the time--dependent vector field generating the Hamiltonian isotopy from \cref{thm:thomintro}, the Energy of the curve can be written as
\beq \en(\gamma)=\int_{I}|\dot \mu_t|^2=\int_{I}\int_{\CP^n}\|N(v_t)\|^2\dd \mu_t\dd t.\eeq
There are plenty of Hamiltonian fields $v_t$ solving the continuity equation \eqref{eq:continuity}, but the condition $p_t\in P_{n,d}\setminus \Delta_{n,d}$ determines the normal component of such fields and allows to prove the following theorem.

\begin{theorem}[The energy of an admissible curve avoiding the discriminant]\label{thm:metricspeedsmoothintro}Let $\gamma:I\to P_{n,d}\setminus \Delta_{n,d}$ be a $\mathscr{C}^1$ curve. Then $\gamma$ is admissible and \beq\label{eq:speedintro} \en(\gamma)=\frac{1}{d\cdot \mathrm{vol}(\CP^{n-1})}\int_{I}\int_{Z(p_t)}\frac{|\dot{p}_t(z)|^2}{\|\nabla^{\C}p_t(z)\|^2}\mathrm{vol}_{Z(p_t)}(\dd z)\dd t.\eeq
\end{theorem}

We observe two interesting consequences of the previous result, which are peculiar of the case $q=2$. First, the innermost integral in \eqref{eq:speedintro} defines a quadratic form on the tangent spaces to $P_{n,d}\setminus \Delta_{n,d}$. This is induced by the real part of a Hermitian form $\h$, which we call the \emph{Wasserstein--Hermitian structure}:
\beq 
\label{eq:herintro}\h_p(q_1, q_2):= \int_{Z(p)}\frac{q_1(z)\overline{q_2(z)}}{\|\nabla^\C p(z)\|^2}\mathrm{vol}_{Z(p)}(\dd z), \quad q_1, q_2\in T_{p}(P_{n,d}\setminus \Delta_{n,d}).
\eeq

In particular, the Energy of a $\mathscr{C}^1$ curve $\gamma:I\to P_{n,d}\setminus \Delta_{n,d}$ can be written in the usual Riemannian way as:
$$\en(\gamma)=\frac{1}{d\cdot \mathrm{vol}(\CP^{n-1})}\int_{I}\h_{p_t}(\dot{p}_t, \dot{p}_t)\, \dd t.$$
Notice that $p\in P_{n,d}$ denotes a class of polynomials in $H_{n,d}$ up to nonzero multiples, but the above definition does not depend on the choice of representatives, see \cref{lemma:ker} and \cref{sec:whs} for more details. Second, this Hermitian structure is not defined on $\Delta_{n,d}$ (this is because the integrand in \eqref{eq:herintro} is not defined when $p\in \Delta_{n,d}$), but its explicit form allows us to perform a semialgebraic reparametrization argument for curves with endpoints on $\Delta_{n,d}$ (\cref{thm:rep}), and to prove that any two points $q_0, q_1$, regardless being on the discriminant or not, can be joined by an admissible curve with \emph{finite} Energy (\cref{thm:finite}). In particular, the minimization problem \eqref{eq:mini} always admits a solution:
\beq \label{eq:W2in}
W_2^\mathrm{in}(p_0, p_1) =\min_{\gamma\in \Omega_{p_{0}, p_1}}\en(\gamma)^{\frac{1}{2}}.
\eeq

\begin{theorem}[The Wasserstein space of hypersurfaces]\label{thm:propertywintro}
The function $W_2^{\mathrm{in}}$ defines a distance on $P_{n,d}$ and the metric space 
$(P_{n,d}, W_2^{\mathrm{in}})$ is  complete and geodesic, i.e. for any couple $p_0,p_1$, there exists an admissible curve $\gamma$ from $p_0$ to 
$p_1$ such that 
$W_2^{\mathrm{in}}(\gamma(s),\gamma(t)) = |t-s| W_2^{\mathrm{in}}(p_0,p_1)$.
\end{theorem}

Notice that $P_{n,d}$ is compact for the ambient $W_2$--metric, by \cref{thm:extintro}. On the other hand, denoting by $\sfd_{\mathrm{FS}}$ the metric induced by the Fubini--Study structure on $P_{n,d}\simeq \CP^N$, the compactness of $(P_{n,d}, W_2^\mathrm{in})$ is equivalent to the continuity of 
$$\mathrm{id}:(P_{n,d}, \sfd_{\mathrm{FS}})\longrightarrow (P_{n,d}, W_2^\mathrm{in}).$$
In metric geometry, it is typically very hard to predict compactness, for the inner metric, of a compact space in the ambient topology (and in fact, in general, not even true for very simple examples, see \cite[Figure 2.2]{Burago}).  However, in a fashion similar to the one--dimensional case, we prove that the above map is continuous (in fact Lipschitz) ``away'' from the discriminant, i.e. where we have the  Riemannian structure \eqref{eq:herintro}. More precisely, setting
$$P_{n,d}(\epsilon):=\{p\in P_{n,d}\,|\, \sfd_{\mathrm{FS}}(p,\Delta_{n,d})\geq \epsilon\},$$
we prove the following result.

\begin{theorem}[Wasserstein and Fubini--Study]\label{thm:Lipschitzintro}For every $\epsilon>0$ the identity map
$$\mathrm{id}:(P_{n,d}(\epsilon), \sfd_{\mathrm{FS}})\longrightarrow (P_{n,d}(\epsilon), W_2^\mathrm{in})$$
is Lipschitz. In particular $P_{n,d}(\epsilon)$ is compact in the $W_2^\mathrm{in}$--topology.
\end{theorem}
It turns out that, for $\delta(n)>0$ small enough and for $1\leq q<1+\delta(n)$, the space $(P_{n,d}, W_q^\mathrm{in})$, is compact (see \cref{thm:compactqintro} below). The proof of this result requires some additional work and for it we will need higher dimensional versions of results on the regularity of roots of polynomials. We will discuss this in \cref{sec:regularityintro}. 
\subsection{The K\"ahler structure and a Weil--Petersson type metric}
Quite remarkably on $P_{n,d}\setminus \Delta_{n,d}$ the Wasserstein--Hermitian structure from \eqref{eq:herintro} is K\"ahler and the corresponding metric is a Weil--Petersson type metric. More precisely, consider the solution variety $V:=\{(p, z)\,|\, p(z)=0\}\subset P_{n,d}\times \CP^n$ together with the restrictions 
of the projections on the two factors:
$$ 
\begin{tikzcd}
           & V \arrow[ld, "\pi_1"'] \arrow[rd, "\pi_2"] &       \\
{P_{n, d}} &                                            & \CP^n
\end{tikzcd}$$
 From this diagram we have two induced maps at the level of $k$--forms: the usual pull--back $\pi_2^*:\Omega^k(\CP^n)\to \Omega^k(V)$  and the integration along the fibers 
 $$(\pi_1)_*:\Omega^{k}(V\setminus \mathrm{crit}(\pi_1))\to \Omega^{k-2n-2}(P_{n,d}\setminus \Delta_{n,d}).$$ On $\CP^n$ we have the Fubini--Study form $\omega_{\mathrm{FS}}\in \Omega^{1,1}(\CP^n)\subset \Omega^{2}(\CP^n)$, which we can pull--back to $V$ by $\pi_2$, and the integral along the fibers of $\pi_2^*\omega_{\mathrm{FS}}^n/n!$ (which is the volume form of $\CP^n$) turns out to be the imaginary part of $\h$. 
\begin{theorem}[The Wasserstein--K\"ahler structure]Let $\h$ be the Wasserstein--Hermitian structure \eqref{eq:herintro} and let $\sigma:=\mathrm{Im}(\h)\in \Omega^{2}(P_{n,d}\setminus \Delta_{n,d}).$
Then
\begin{equation}\label{eq:sigma}\sigma=(\pi_1)_*\pi_2^*\left(\frac{\omega_{FS}^n}{n!}\right)\in \Omega^{1,1}(P_{n,d}\setminus \Delta_{n,d}).\end{equation}
In particular $(P_{n,d}\setminus\Delta_{n,d}, \h)$ is a K\"ahler manifold.
\end{theorem}
 
The K\"ahler metrics on the parameter space of a family of algebraic varieties which can be 
obtained as pushforwards of objects on the total space of the family (as in \eqref{eq:sigma}) are recurring objects in complex algebraic geometry, and they are usually called \emph{Weil--Petersson type metrics}. For instance, the Weil–Petersson metric on the moduli space of curves can be obtained in this
way, and a version for a family of Calabi–Yau varieties appears for instance in \cite{Songtian, Tosatti}.

\subsection{The regularity of a curve of measures and a compactness result}\label{sec:regularityintro} Let us go back to the metric space $(P_{n,d}, W_q^\mathrm{in})$.
In order to prove its compactness for $1\leq q\leq 1+\delta(n)$ we will use a higher dimensional generalization of the results from  \cite{Sobolev}. We first discuss this generalization, which is interesting in its own, as it addresses and solves a natural question on the regularity of zeroes of families of \emph{multivariate} polynomials. We then explain how it can be used to prove the compactness of the $q$--Wasserstein inner  metric. 

Given a curve of polynomials $p_t:I\to P_{n,d}$, it is natural to ask what regularity we can expect for the curve of measures $\mu(p_t)$. 
We go back to the case $n = 1$. By the 
homeomorphism 
$$
\{\mathrm{coefficients}\}=\C^d\stackrel{\mu}{\longrightarrow} \mathrm{SP}^d(\C)=\{\mathrm{zeroes}\},
$$
it follows that the roots of a polynomial depend continuously on its coefficients but, in general, no more regularity can be assumed.
We can explain this phenomenon by looking at the previous homeomorphism from a metric point of view, endowing $\mathrm{SP}^d(\C)$ with the Wasserstein distance, as in \eqref{eq:W21}.
The loss of regularity  of the roots of a family of polynomials is due to the fact that the above map 
$\mu$, seen as a map between metric spaces,  \emph{is not} Lipschitz. This failure of regularity happens precisely on the discriminant,
 where we have polynomials with multiple roots.
For instance, the zeroes $e^\frac{i2\pi k}{d}t^{\frac{1}{d}}$ of the polynomial $z^d-t$ involve a $d$--th square root and cannot be Lipschitz  functions of $t$ near $t=0$. 
(Notice that, in this case, $\mu^{-1}:\mathrm{SP}^d(\C)\to \C^d$ is instead Lipschitz, as the coefficients of a polynomial depends smoothly on its zeroes.)

Still,  A. Parusinski and A. Rainer in a series of works culminated with \cite{Sobolev} have recently obtained the sharp result that the zeroes of a curve of polynomials (for $n=1$) enjoy a Sobolev regularity: 
in particular any continuous choice of a root of $p_{t}$ has first derivative in $L^{q}$ for $1\leq q<\frac{d}{d-1}$, provided 
$p_t:I\to \C^d$ is a $\mathscr{C}^{d-1,1}$ curve.

In our formalism  the results from \cite{Sobolev} are equivalently stated by saying that if $p_t:I\to \C^d$ is a $\mathscr{C}^{d-1,1}$ curve, then $\mu_t:=\mu(p_t)$ is in $AC^q(I, \mathscr{P}_q(\C))$ 
for every $1\leq q<\frac{d}{d-1}.$ It is then natural to adopt this formalism to address the same question in the more general setting of $n \geq 2$; to the best of our knowledge no other approach is at disposal to address this question.

Quite remarkably, in higher dimensions, the existence of rational curves on the zero set of some of the polynomials in the curve $p_t$ (a very classical problem in algebraic geometry!) 
seems to be an obstruction for the regularity of $\mu(p_t)$. In this direction we prove two results, the first being a direct generalization of \cite[Theorem 1]{Sobolev}. It requires the same regularity as \cite[Theorem 1]{Sobolev} on the curve $p_t$, but at the same time it needs the curve to miss (what we call) the \emph{Fano discriminant} (see \cref{definition:Fano}). It can tought as the set $F_{n,d}\subset P_{n,d}$ consisting of polynomials $p$ for which the set of lines on $Z(p)$ is not defined by regular equations on the Grassmannian $\mathbb{G}(1,n)$.

The Fano discriminant is contained in a proper algebraic set of $P_{n,d}$, see \cite[Theorem 6.34]{3264}.
When $n=1$, it is empty: in this case the only line is $\CP^1$ itself, where no nonzero polynomial can vanish identically. In general, if $d>2n-3$ and $p\notin  {F_{n,d}}$, then no line can be found on $Z(p)$ and one can use  integral geometry  to perform a coupling between measures ``uniformly'' on the set of lines. Notice that, when $n=1$, the conditions $d>2n-3$ and $p_t\notin F_{n,d}$ are empty, and the statement below reduces to \cite[Theorem 1]{Sobolev}.

\begin{theorem}[Regularity of curves of measures I]\label{T:Parusinskiintro}
Let $d>2n-3$. Let $p_t \in \mathscr{C}^{d-1,1}(I, P_{n,d})$ be a curve such that $p_t\notin {F_{n,d}}$ for every $t\in I$.
Then, for every $1 \leq q < d/(d-1)$ the curve
$\mu_t:=\mu(p_t) \in AC^{q}(I,\mathscr{P}_q(\CP^n)).$
\end{theorem}

The second regularity result that we prove allows to drop the restriction $p_t\notin F_{n,d}$, but requires to assume more regularity on the curve $p_t$, at the price of losing some regularity on the curve of measures $\mu_t$.

\begin{theorem}[Regularity of curves of measures II]\label{thm:rationalintro}For every $n\in \mathbb{N}$ there exists $e(n)>0$ such that for every curve $p_t \in \mathscr{C}^{k,1}(I, P_{n,d})$,  with $k\geq e(n) d-1$,
and for every $1 \leq q < e(n)d/(e(n)d-1)$ the curve
$\mu_t:=\mu(p_t) \in AC^{q}(I,\mathscr{P}_q(\CP^n)).$
\end{theorem}

The main idea for the proof of \cref{thm:rational} is to replace the use of lines, in the integral geometry approach, with the use of rational curves, see \cref{P:rationalcoupling}. The crucial point for our argument is to find at least one rational curve, possibly of very high degree, with the property that nor this curve nor any of its translates under the natural unitary group action lie on any of the hypersurfaces $Z(p_t)$, $t\in I$. The existence of such rational curve follows from a recent result of B. Lehmann, E. Riedl and S. Tanimoto, privately communicated to the authors (\cref{thm:riedl}). The  constant $e(n)$ in \cref{thm:rationalintro} is the degree of this rational curve, and is related to the so--called Fujita invariants.

Using \cref{thm:rationalintro}, we can finally prove the following compactness result. 
\begin{theorem}[Compactness]\label{thm:compactqintro}
For every $1 \leq q < 1+\frac{1}{e(n)d-1}$, 
the metric space $(P_{n,d}, W_q^\mathrm{in})$ is compact.
\end{theorem}

\subsection{A Wasserstein approach to condition numbers}In an influential sequence of papers, the ``Complexity of B\'ezout's Theorem'' series \cite{CB1, CB2, CB3, CB4, CB5}, M. Shub and S. Smale opened to a geometric investigation of the complexity of polynomial systems solving. This line of research has received a lot of attention, and culminated with the positive solution of Smale's 17th Problem \cite{Smaleproblems}: ``Can a zero of $n$ complex polynomial equations in $n$ unknowns be found approximately, on the average, in polynomial time with a uniform algorithm?''.  The solution to this problem follows from a combination of  works by C. Beltr\'an and L. M. Pardo \cite{BP}, M. Shub \cite{CB6}, P. B\"urgisser and F. Cucker \cite{BuCuannals} and finally by P. Lairez \cite{Lairez1}. 

Here we propose a conceptual application of our approach: to go back to the original question of Smale, 
 reinterpreting it in the language of Optimal Transport. We do not venture into this full program, which would be beyond the purposes of this paper. Nonetheless, we will show how our  point of view already gives a solution to \cite[Problem 14]{burgisser} in the case $n=1$ (see \cref{thm:P14intro} below). 

Let us explain the context behind \cite[Problem 14]{burgisser}. One of the main ideas underlying the algorithm solving Smale's problem is \emph{continuation along linear paths}: given a polynomial system $p_1=0$, which is assumed to be regular, one builds another regular polynomial system $p_0$ of which one solution $z_0$ is known, and then ``continues'' this solution along the linear homotopy $p_t:=(1-t)p_0+t p_1$. The fact that one can choose $p_0$ so that the linear path $p_t$ does not hit the discriminant follows from dimensional considerations (the discriminant has real codimension two); the complexity of the algorithm is related to the possibility of estimating the distance to the discriminant throughout the path $p_t$. In order to measure this distance, one introduces the notion of \emph{normalized condition number} of a system at one of its solutions. In the case of a single polynomial $p\in P_{1, d}$ at one of its zeroes $z\in \CP^1$, this is the quantity:
$$\nu_{\mathrm{norm}}(p, z):=\sfd_{\mathrm{BW}}(p, \Delta_{z})^{-1},$$
where $\Delta_{z}\subset \Delta_{1, d}$ denotes the set of polynomials with a multiple root at $z$ (see \cref{sec:condition} for more details), and $\sfd_{\mathrm{BW}}$ denotes the Bombieri--Weyl distance on the projective space of polynomials (see \cref{Appendix A}).

In order to address Smale's problem, prior to its solution, M. Shub \cite{CB6} suggested the possibility of changing the class of paths, switching from linear paths to special  homotopies $(p_t, z_t)$, with $p_t(z_t)\equiv 0$, called \emph{condition geodesics}. In order to describe them, in the case $n=1$, let us introduce the \emph{solution variety} $V:=\{(p, z)\in P_{1,d}\times \CP^1\,|\, p(z)=0\}$ together with the projection $\pi_1:V\to P_{1,d}$ on the first factor.  Denoting by $\Sigma'$ the set of critical points of $\pi_1$, the normalized condition number can be used to turn $V\setminus \Sigma'$ into a \emph{Lipschitz} Riemannian manifold as follows. First, we denote by $g_V$ the Riemannian structure on $V\subset P_{1,d}\times \CP^1$ induced by the product Riemannian structure. Then, for every absolutely continuous curve $\gamma_t=(p_t, z_t):I\to V\setminus \Sigma'$ we define its \emph{condition length} by
$$L_{\mathrm{cond}}(\gamma):=\int_{I}\nu_{\mathrm{norm}}(p_t, z_t)\cdot \|\dot\gamma_t\|_V\dd t.$$
Taking the infimum of the length over all absolutely continuous curves joining two points endows $V\setminus \Sigma'$ with a metric structure, called the \emph{condition metric}. A condition geodesic is a a curve $\gamma_t$ in $V\setminus \Sigma'$ minimizing the condition length between any two of its points.

A similar construction can be introduced in the case of polynomial systems, and in \cite{CB6} it is shown that $O ( d^{ 3 / 2} L_\mathrm{cond}(\gamma))$ Newton steps are sufficient to continue the zero $z_0$ from $p_0$ to $p_1$ along $\gamma_t$. Following this result, C. Beltrán, J.-P. Dedieu, G.Malajovich and M.Shub, in \cite{BDM1, BDM2}, asked whether the function $t\mapsto \log(\nu_{\mathrm{norm}}(\gamma_t))$ is convex along a condition geodesic. An affirmative answer to this question would imply that, for any condition geodesic $\gamma_t$, one has  $L_\mathrm{cond}(\gamma_t)\leq L_V(\gamma_t)\max\{\nu_{\mathrm{norm}}(p_0, z_0), \nu_{\mathrm{norm}}(p_1 , z_1)\},$ where $L_V(\gamma_t)$ is the length of $\gamma_t$ in the Riemannian metric of $V$ defined above. In particular, it would be possible to estimate the distance from the discriminant along the whole path $p_t$ by only knowing this distance at the starting and the ending system. For more details on this problem, we refer the reader to \cite[Problem 14]{burgisser}, from where this formulation is taken.

Here we propose a solution to this problem in the case $n=1$, changing the class of paths, and replacing \emph{condition geodesics} with lifts of \emph{Wasserstein geodesics}. We expect a similar result to be true in the case $n>1$, but we leave this question to future investigation (this will require to work with the $W_2^\mathrm{in}$ metric on the space of polynomial systems).

\begin{theorem}[The condition length of Wasserstein geodesics]\label{thm:P14intro}For every $d\in \N$ there exist $\beta_1, \beta_2, \beta_3, \beta_4>0$ such that, if $p_t:I\to P_{1,d}$ is a $W_2$--geodesic between $p_0, p_1\in P_{1,d}\setminus \Delta_{1,d}$ then the following is true:
\begin{enumerate}
\item for every $t\in I$ 
\beq\label{eq:MMs}
\mathrm{dist}(p_t, \Delta_{1,d})^2\geq \beta_1\min\{\mathrm{dist}(p_0, \Delta_{1,d}), \mathrm{dist}(p_1, \Delta_{1,d})\}^{\beta_2}.
\eeq
In particular, $P_{1,d}\setminus \Delta_{1,d}$ is $W_2$--geodesically convex;
\item if $\gamma:I\to V\setminus \Sigma'$ is a lift of $p_t$, i.e.  $\gamma(t)=(p_t, z_t)$, then:
$$
L_\mathrm{cond}(\gamma)\leq \beta_3 \cdot L_{V}(\gamma)\cdot \max\{\nu_{\mathrm{norm}}(p_0, z_0), \nu_{\mathrm{norm}}( p_1, z_1)\}^{\beta_4}.
$$
\end{enumerate}
\end{theorem}

The proof of this result is based on the structure of Wasserstein geodesics, see for instance the section on Monge--Mather shortening principle in \cite{villani:oldandnew}, which guarantees a quantitative  separation of roots along such curves (\cref{propo:quasiconc}). (The constants $\beta_i$ from the statement come from the use of Lojasiewicz's inequality, and we expect that actually $\beta_2=\beta_4=1$.)

We remark that recently P. Lairez \cite{Lairez2} has proposed to consider yet another class of paths, called \emph{rigid paths}, that have a quasi--optimal polynomial behavior in terms of the complexity of Smale's 17th Problem. We wonder what could be the behavior of an algorithm based on our Wasserstein approach. In this direction, we observe that, compared to the condition metric, which is only Lipschitz,  the $W_2^\mathrm{in}$ metric, away from the discriminant, comes from a smooth Riemannian metric (\cref{thm:smooth}); in particular Wasserstein geodesics can be found by means of a smooth Riemannian optimization problem.

\subsection{Structure of the paper} 
 \cref{S:OT} recalls basic facts about the optimal transport problem and Wasserstein distances on Riemannian manifolds reserving special care to the dynamical formulation \`a la Benamou--Brenier. The section also covers some needed tools from metric geometry.

 \cref{sec:Measures} is the heart of the construction of the correspondence $\mu$ (in all its incarnations) associating a measure to a polynomial. 
This is the map that appears in Theorem \ref{thm:extintro} before, proved in this section (this is precisely Theorem \ref{thm:ext}).
We begin with a careful study of the zero dimensional case (illustrated in the introduction) then, using integral geometry tools, we move to the higher dimensional case.

In  \cref{sec:Ham}, we use instruments from symplectic geometry to compute the metric speed of a curve of polynomials that stays away from the discriminant. In particular this is a consequence of  \cref{thm:thomintro}.

In \cref{S:OTalgebraic}, the optimal transport problem for hypersurfaces in the form of zero sets of polynomials is introduced and studied.
We define the notion of admissible curves, energy and inner metric. Here we prove Theorem \ref{thm:metricspeedsmoothintro} (corresponding to \cref{thm:metricspeedsmooth})
, we define the Wasserstein--Hermitian structure and  
prove \cref{thm:propertywintro} (\cref{thm:propertyw} in the body of the paper) and Theorem \ref{thm:Lipschitzintro} (corresponding to Theorem \ref{thm:Lipschitz}) of comparison between the Wasserstein and Fubini--Study metric. In \cref{sec:WK} we prove that the Wasserstein--Hermitian metric is K\"ahler (\cref{thm:WK} in the body of the paper).

In \cref{sec:abscont} we study the problem of the regularity of curves of solutions of polynomials in dimension greater than one. We prove here \cref{T:Parusinskiintro} (\cref{T:Parusinski} in the body) and \cref{thm:rationalintro} (\cref{thm:rational} in the body). \cref{thm:compactqintro} corresponds to \cref{thm:compactq} in the body of the paper.
In \cref{sec:condition} a Wasserstein approach to condition number is discussed; here we prove \cref{thm:P14intro} (\cref{thm:P14} in the body).


\subsection*{Acknowledgements}We thank Peter B\"urgisser and Jacopo Stoppa for insightful discussions. We also wish to thank Roberto Svaldi for pointing out the work of Eric Riedl and coauthors, and Eric Riedl for sharing with us their results, which helped proving a stronger versions of our \cref{thm:rational}, and an anonymous referee for pointing out the connection to Weil--Petersson type metrics.

\section{A quick overview on Optimal transport}
\label{S:OT}
In order to introduce an optimal transport distance over the space of homogeneous polynomials,
a careful analysis on the infinitesimal behaviour of the Wasserstein distance 
near a general probability measure is required. 
In this section we will therefore provide a quick introduction to Optimal Transport
and review the needed material for the rest of the paper. We refer to the monographs \cite{villani:topics, villani:oldandnew} and references therein
for a complete account on the theory.


\subsection{Wasserstein distance}\label{sec:wasserstein}

Let $(M,g)$ be a smooth compact Riemannian manifold. We will denote by $\mathscr{P}(M)$ the space of 
 Borel probability measures over $M$. 
On $\mathscr{P}(M)$ we will consider the following distances induced 
by Optimal Transport.

\begin{definition}[$q$--Wasserstein distance]Fix any $q\geq 1$. For any $\mu_{0},\mu_{1} \in \mathscr{P}(M)$,  the $q$-Wasserstein distance 
is defined by
$$
W_{q} (\mu_{0},\mu_{1}) : = \left(\inf_{\xi \in \Pi(\mu_{0},\mu_{1})}
\int_{M\times M} \sfd_{g}(x,y)^g \, \xi(dxdy)\right)^{\frac{1}{q}},  
$$
where $\sfd_{g}$ is the geodesic distance induced by the Riemannian metric $g$ (making $(M,{\sfd}_{g})$ a compact metric space) and 
$$
\Pi(\mu_{0},\mu_{1}) : = \Big{\{} \xi \in \mathscr{P}(M\times M) \colon 
(\pi_{1})_{\sharp} \xi = \mu_{0}, \ (\pi_{2})_{\sharp} \xi = \mu_{1} \Big{\}}
$$
is the (convex) set of \emph{transport plans} between $\mu_{0}$ and $\mu_{1}$; finally $\pi_{i} : M \times M \to M$ denotes the projection on the $i$-th component, for $i=1,2$. 
\end{definition}

Roughly, transport plans represent
the way in which $\mu_0$ is moved to $\mu_1$
while $\sfd_g^q$ is the chosen cost function for transporting an infinitesimal amount of mass.
The subset of those transport plans realising the  infimum in the definition of the $W_{q}$ distance 
are called optimal transport plans and 
are denoted by $\Pi_{opt}^q(\mu_{0},\mu_{1})$.

As we will always stick to the case of compact manifolds, $W_{q}$ will be well defined between any couple of elements of $\mathscr{P}(M)$ 
with a non--empty $\Pi_{opt}^q(\mu_{0},\mu_{1})$.
In particular 
$(\mathscr{P}(M),W_q)$ 
is a complete and separable metric space;
from the compactness of $M$ it 
follows moreover that 
it is compact as well.
For ease of notation we will sometimes write $\mathscr{P}_q(M)$ 
to intend the metric space  $(\mathscr{P}(M), W_q)$.

\subsection{
Length structures and inner distances}
\label{Ss:metricgeometry}

Here we follow \cite{Burago} 
to which we refer for the missing details.
If $(X,\sfd)$ is a metric space one can 
define a length structure in the following way. 
Firstly consider a class of admissible curves.

\begin{definition}[Absolutely continuous curves with values in a metric space]
Let $(X,\sfd)$ be a metric space and let $I$ be a real interval. 
We say that a curve $\gamma : I \to X$ belongs to $AC^{q}(I;X)$,  $q \geq 1$,
if there exists $m \in L^{q}(I)$ such that
\begin{equation}\label{E:ac}
\sfd(\gamma_s,\gamma_t) \leq \int_{s}^{t}m(\tau) \,d\tau, \qquad \forall \ s \leq t.
\end{equation}
When $q=1$, we simply denote $AC=AC^1.$
\end{definition}
\begin{theorem}[Theorem 1.1.2 of \cite{AGS:book}]
If $\gamma \in AC^{q}(I;X)$, $q\geq 1$,
then for $\L^{1}$-a.e. $t \in I$ there exists the limit
$$
\lim_{h\to 0} \frac{\sfd(\gamma_{t+h}, \gamma_t)}{|h|}.
$$
We denote the value of this limit by $|\dot \gamma|(t)$ and we call it metric derivative, or metric speed, of $\gamma$ at $t$. 
The function 
$t \mapsto |\dot \gamma|(t)$ belongs to $L^{q}(I)$ and 
$$
\sfd(u(s),u(t)) \leq \int_{s}^{t}|\dot \gamma|(\tau) \, d\tau, \qquad \forall \ s \leq t.
$$
Moreover $|\dot \gamma|(t) \leq m(t)$ for $\L^{1}$-a.e. $t\in I$, 
for each $m$ such that \eqref{E:ac} holds.
\end{theorem}

The space of absolutely continuous curves 
$AC^{q}(I;X)$ can be seen as a subset of 
the metric space $\mathscr{C}^{0}(I;X)$ 
(endowed with the natural sup distance); recall 
that $\mathscr{C}^{0}(I;X)$  is complete and separable provided the same is true for $X$. 

One then defines the length of absolutely continuous curves, that we can assume without loss of generality to be defined on $I=[0,1]$, by 
$$
L(\gamma) = \int_0^1|\dot \gamma|(t)\,dt.
$$
The length $L$ induces an ``intrinsic'' distance, called \emph{inner distance}, by 
minimizing the length 
functional among all the sets of admissible curves:
$$
\widehat {\sfd}(x,y) = \inf \Big{\{} L(\gamma)\colon \gamma \in AC([0,1];X), \gamma_0 = x,\, \gamma_1 = y \Big{\}}.
$$
If $\widehat{\sfd} = \sfd$, then $(X,\sfd)$ is called a length space. 
It is called a geodesic space 
if 
the infimum is attained by at least one admissible curve, hereafter called a geodesic.  

Our analysis will naturally 
bring us to study 
non-geodesically convex subsets
of a geodesic space $X$.  
For our scope it will be sufficient 
to assume $(X,\sfd)$ to be complete and separable and $Y \subset X$ 
to be a closed set so that $(Y,\sfd)$ is still complete and separable. 
Clearly if $Y$ fails to be  geodesically convex then $(Y,\sfd)$ will not be anymore a geodesic space. 
Nonetheless 
one can consider 
the length structure induced by $L$ over $Y$ 
and define the inner distance of $Y$ 
as follows:
\begin{equation}
\label{E:intrinsic}    
\sfd_Y^{\textrm{in}} (x,y): = 
\inf \Big{\{} L(\gamma) \colon  
\gamma \in AC([0,1];Y), \gamma_0 = x, \,\gamma_1 = y\Big{\}}.
\end{equation}
If for any couple $x,y \in Y$ there exists 
an admissible curve $\gamma \in 
AC([0,1];Y)$ of finite length $L(\gamma) < \infty$, then 
$\sfd_Y^{\textrm{in}}$ is a distance 
making $(Y,\sfd_Y^{\textrm{in}})$ a length space.

Through a standard reparametrization argument, see for instance \cite[Lemma 1.1.4]{AGS:book}, 
one notices that for any 
$q>1$
\begin{equation}
\label{E:p}
\sfd_Y^{\textrm{in}}(x,y)^q
= \inf \left\{ \int_0^1 |\dot \gamma|^q(t)\,dt 
\colon \gamma \in
AC^q([0,1];Y), \gamma_0 = x,\, \gamma_1 = y
\right\}.
\end{equation}
Finally we mention that, defining the energy functional 
$\mathcal{E} : \mathscr{C}^{0}(I;Y) \longrightarrow 
[0;\infty]$ 
by 
$$
\mathcal{E}(\gamma) : = 
\begin{cases}
\int_0^1 |\dot \gamma|^2(t)\, & u \in  AC^{2}(I;Y); \\
+ \infty, & {\textrm{ otherwise}},
\end{cases}
$$
then $\mathcal{E}$ is lower semicontinuous (see for instance \cite[Proposition 1.2.7]{Gigli_Pasqualetto_2020}). 

\smallskip
\subsection{Absolutely continuous curves, 
the continuity equation and the space of gradients}\label{sec:spaceofgrad}

We continue the short review of Optimal Transport  needed for this paper. 
Since $(M,\sfd_g)$ is a geodesic space, the $q$-Wasserstein space $\mathscr{P}_q(M)$
is a geodesic space as well (see \cite[Proposition 2.10]{SturmI}): for any 
$\mu_{0},\mu_{1} \in \mathscr{P}(M)$ there exists a curve 
$(\mu_t)_{t \in [0,1]} \subset \mathscr{P}(M)$ 
from $\mu_0$ to $\mu_1$ such that 
$$
W_{q}(\mu_{t},\mu_{s}) = |t-s|W_{q}(\mu_{1},\mu_{0}).
$$
Wasserstein geodesics are strongly linked to 
the geometry of the underlying space $M$. 
Denote by $\Geo(M)$ 
the set of minimal geodesics of $M$ seen as a closed 
subspace of $\mathscr{C}^{0}([0,1];M)$; 
then 
in the strictly convex case ($q>1$)
 $W_{q}$-geodesics have the following lifting property 
(see \cite[Proposition 2.10]{SturmI}):
for each $W_{q}$-geodesic $(\mu_t)_{t \in [0,1]}$
there exists a probability measure $\eta$ on $\Geo(M)$, 
such that
$$
(\ee_{t})_{\sharp} \eta = \mu_{t}, \qquad 
\forall \ s,t \in [0,1] \quad (\ee_{s},\ee_{t})_{\sharp} \eta \in \Pi_{opt}^q(\mu_{s},\mu_{t}), 
$$
where $\ee_t : \mathscr{C}^{0}([0,1];M) \to M$ is the evaluation map ($\ee_t(f) : = f(t)$).
The set of $\eta \in \mathscr{P}(\Geo(M))$ verifying the previous conditions is called the set of dynamical $q$-optimal plans 
and is denoted by $\OptGeo_q(\mu_{0},\mu_{1})$.

As discussed, our interest will be in studying the inner distance induced by $W_q$ on a distinguished closed 
subset of $\mathscr{P}(M)$ 
and to define its inner 
Wasserstein distance as in 
\cref{E:intrinsic}. 
It will be therefore necessary to discuss the metric speed 
of curves in $AC^q([0,1];\mathscr{P}_q(M))$.
As for $W_q$--geodesics, also 
$q$--absolutely continuous curves $\mathscr{P}_q(M)$ have a  geometric interpretation 
in terms of family of curves of $M$ firstly pioneered by 
J.-D. Benamou and Y. Brenier 
\cite{BenamouBrenier}.

If 
$\mu_{t} : I \to \mathscr{P}_{q}(M)$ is a $q$-absolutely continuous curve 
with $I\subset \R$ an interval,
there exists a Borel vector field $v : (x,t)\mapsto v_{t}(x) \in T_x M$ 
such that 
$\mu_t$ 
solves the continuity equation 
with vector field $v_t$:
\begin{equation}
\label{E:continuity1}
\partial_{t} \mu_{t} + \ddiv(v_{t}\mu_{t}) = 0, 
\end{equation}
in the sense of distributions (see \eqref{eq:sod} below). 
One can 
check that \eqref{E:continuity1}
implies that
$|\dot{\mu_{t}} | \leq \| v_{t} \|_{L^{q}_{\mu_{t}}}$ for $\mathcal{L}^{1}$-a.e. $t \in I$.

The vector field $v_t$ is highly non-unique and in order to restore the identity between the metric speed of $\mu_t$ and the norm of $v_t$
a subclass of special vector fields for \eqref{E:continuity1} has to be considered. 
This class is introduced by analogy with the infinitesimal behaviour of $q$-Wasserstein geodesics. 
As we will be mostly dealing with the case $q=2$, where an infinitesimal Riemannian structure appears naturally, we will 
continue this review 
only for the case $q=2$. 
We refer to \cite{AGS:book} and references therein for the other cases.

For every $\mu \in \mathscr{P}_2(M)$, the so-called ``space of gradients'' is defined by
\begin{equation}\label{E:Gradients}
T_{\mu}(\mathscr{P}_{2}(M)) : = 
\overline{\{\nabla f \colon f \in \mathscr{C}^{\infty}(M) \}}^{L^{2}_{\mu}}.
\end{equation}

This definition is motivated by the
classical results of Y. Brenier and R. McCann on the shape of 
optimal maps and on 2-cyclically monotone sets. 
For more details on all the possible notions of tangent space to $\mathscr{P}_{2}(M)$ 
we refer to \cite{Gigli11, Ohta09}.

We now state the main result 
concerning the Wasserstein metric speed
in $\mathscr{P}_q(M)$.
For the analogous and  earlier
result when $M$ is an Hilbert 
space
we refer to \cite[Theorem 8.3.1]{AGS:book} and \cite[Proposition 8.4.5]{AGS:book}, where also the case $q \neq 2$ is considered.

\begin{theorem}[Proposition 2.5 of \cite{Erbar_2010}]\label{T:AGS1}
Let $(\mu_{t})_t \in AC^2(I; \mathscr{P}_{2}(M))$ be a 2-absolutely continuous curve 
with $I\subset \R$ an interval.  
Then there exists a Borel vector field $v : (x,t)\mapsto v_{t}(x) \in T_x M$ such that
$$
v_{t} \in L^{2}(\mu_{t},TM), \qquad \| v_{t} \|_{L^{2}_{\mu_{t}}} = |\dot{\mu_{t}} |, \qquad 
\mathcal{L}^{1}-a.e. \ t \in I, 
$$
and the continuity equation 
\begin{equation}\label{E:continuity}
\partial_{t} \mu_{t} + \ddiv(v_{t}\mu_{t}) = 0, 
\end{equation}
holds in the sense of distributions on $I\times M$. Moreover for $\mathcal{L}^{1}$-a.e. $t \in I$, 
$v_{t} \in T_{\mu_{t}}(\mathscr{P}_2(M))$.\\
The vector field $v_t$ is uniquely determined in $L^2(\mu_t, TM)$ by
\eqref{E:continuity} and by being an element of $T_{\mu_{t}}(\mathscr{P}_2(M))$.

Conversely, if a narrowly continuous curve $\mu_{t} : I \to \mathscr{P}_{2}(M)$
satisfies the continuity equation for some Borel velocity field $v_{t}$
with $\|v_{t} \|_{L^{2}_{\mu_{t}}} \in L^{2}(I)$, then
$(\mu_{t})_t \in AC^2(I; \mathscr{P}_{2}(M))$ and 
$|\dot{\mu_{t}} | \leq \| v_{t} \|_{L^{2}_{\mu_{t}}}$ 
for $\mathcal{L}^{1}$-a.e. $t \in I$. 
\end{theorem}

Notice 
that 
\cite[Proposition 2.5]{Erbar_2010} and 
the discussion of \cref{Ss:metricgeometry}
permit to reobtain the 
so-called Benamou-Brenier formulation of the Optimal Transport problem:
for any $\mu_0,\mu_1 \in \mathscr{P}_2(M)$ 
it holds
\begin{equation}\label{E:BenamouBrenier}
W_2(\mu_0,\mu_1)^2 = \min \left\{ \int_0^1 \|v_t \|^2_{L^2_{\mu_t}}\, dt \right\},
\end{equation}
where the minimum is taken over all absolutely continuous curves $(\mu_t)_{t \in [0,1]}$ connecting $\mu_0$ to $\mu_1$ with tangent vector field $(v_t)_t$ in the sense of \eqref{E:continuity}.

By definition, $T_{\mu}(\mathscr{P}_{2}(M))$ is a closed subspace 
of $L^{2}(\mu;TM)$ 
and therefore has a uniquely defined projection
\beq
\mathsf{P}:L^2(\mu; TM)\to T_\mu(\mathscr{P}_2(M)).
\eeq
The orthogonal space to $T_{\mu}(\mathscr{P}_{2}(M))$ consists
of divergence-free vector fields, giving the following characterization for $\mathsf{P}$.

\begin{lemma}
[Lemma 2.4 of \cite{Erbar_2010}]
\label{L:AGSprojection}
A vector $v\in L^{2}(\mu;TM)$ belongs to $T_{\mu}(\mathscr{P}_{2}(M))$ if and only if 
$$
\| v  + w \|_{L^{2}_{\mu}} \geq \| v \|_{L^{2}_{\mu}}, \qquad \forall \ w \in L^{2}(\mu;TM)\ 
such \ that \ 
\ddiv(w\mu) = 0.
$$
In particular, for every $v\in L^{2}(\mu;TM)$ there exists a unique 
$\mathsf{P}(v) \in T_{\mu}(\mathscr{P}_{2}(M))$ 
in the equivalence class of $v$ modulo divergence--free vector fields: 
it is the element of minimal $L^{2}$-norm in this class.
\end{lemma}

In order to highlight 
the relevance of the projection $\mathsf{P}$ we 
reformulate \cref{T:AGS1}.

\begin{proposition}[Proposition 2.5 of \cite{Erbar_2010}]\label{P:AGS2}
Let $\mu_{t} : I \to \mathscr{P}_{2}(M)$ be an absolutely continuous curve 
and let $v_{t} \in L^{2}(\mu_{t};TM)$ be such that 
$$
\partial_{t} \mu_{t} + \ddiv(v_{t}\mu_{t}) = 0, 
$$
holds in the sense of distributions. Then $v_{t}$ satisfies 
$$
\| v_{t} \|_{L^{2}_{\mu_{t}}} = |\dot{\mu_{t}} |, \qquad 
\mathcal{L}^{1}-a.e., 
$$
if and only if $v_{t} = \mathsf{P}(v_{t}) \in T_{\mu_{t}}(\mathscr{P}_{2}(M))$ for 
$\mathcal{L}^{1}$-a.e. $t \in I$.
\end{proposition}

Relevant for this paper will be curves obtained via the pushforward of families of diffeomorphisms: 
let $\varphi_{t} : M \to M$ be a family of diffeomorphisms which is the flow of a time dependent vector field $v_t\in \Gamma(TM)$. If we
define $\mu_{t} : = (\varphi_{t})_{\sharp}\mu_{0}$,
then 
$$
\partial_{t} \mu_{t} + \ddiv (v_{t} \mu_{t}) = 0.
$$
Indeed by definition 
$\frac{d}{dt} \varphi_t(x) =v_t(\varphi_t(x))$ 
and for 
any smooth test function $\psi$:
\beq\label{eq:sod}
\frac{d}{dt} \int_M \psi \dd\mu_t = 
\frac{d}{dt} \int_M \psi(\varphi_t) \dd\mu_0 = \int g\Big{(}\nabla \psi(\varphi_t(x)), v_t(\varphi_t(x)) \Big{)}\dd\mu_0= \int g(\nabla \psi, v_t )  \dd\mu_t.
\eeq

\noindent
Moreover since $\ddiv ((v_{t} - \mathsf{P}(v_{t}))\mu_{t}) = 0$, 
$$
\partial_{t} \mu_{t} + \ddiv (\mathsf{P}(v_{t}) \mu_{t}) = 0.
$$
Since $\mathsf{P}(v_{t}) \in T_{\mu_{t}}(\mathscr{P}_{2}(M))$, by Proposition \ref{P:AGS2}
we have therefore obtained the following statement. 

\begin{corollary}\label{propo:metricspeedflow}
Let $\varphi_{t} : M \to M$ be a family of diffeomorphisms for $t \in I$,
with $\varphi_{t}$ be the flow map associated to a vector 
field $v_{t} \in \Gamma(TM)$ and $I$ any real interval.
Define for $t \in I$ the curve of probability measures 
$\mu_{t} : = (\varphi_{t})_{\sharp}\mu_{0}$.
Then, for $\mathcal{L}^{1}$-a.e. $t \in I$: 
$$
|\dot{\mu_{t}} | = \| \mathsf{P}(v_{t}) \|_{L^{2}_{\mu_{t}}}. 
$$
\end{corollary}

The description of the space of gradients 
is therefore a natural problem to address.
The following proposition describes the space of gradients in the case of a measure $\mu$ supported on a smooth submanifold $S\subseteq M$.
For other descriptions 
of the tangent cone to the $2$-Wasserstein space
and similar results, we refer to \cite{Lott17}.

\begin{proposition}\label{propo:decograd}
Let $S \subset M$ be a submanifold and $\mu \in \mathscr{P}(M)$ such 
that $\supp (\mu) \subset S$.
Then 
\begin{equation}\label{splittingsupported}
   \overline{\{\nabla f \colon f \in \mathscr{C}^{\infty}(M) \}}^{L^{2}_\mu}
= \overline{\{\nabla g \colon g \in \mathscr{C}^{\infty}(S) \}}^{L^{2}_\mu} 
\oplus
L^{2}(NS, \mu), 
\end{equation}
where $NS \subset TM$ is the normal bundle to $S$ in $M$. 
\end{proposition}

\begin{proof}
Since $\mu$ is supported on $S$ this arises by separating tangential and normal components at $S$ of the gradients of functions on $M$.
The unique non trivial inclusion to justify is: $L^2(NS,\mu) \subseteq  \overline{\{\nabla f \colon f \in \mathscr{C}^{\infty}(M) \}}^{L^{2}_\mu}$. For this
we just have to prove that any smooth section $\nu \in \Gamma(S,NS)$ is the normal component at $S$ of the gradient of a smooth function defined around $S$. This can be easily done using the tubular neighborhood theorem. Identify $NS \cong TS^{\bot}$; the map
$$\varphi:NS \longrightarrow M, \quad v\longmapsto \operatorname{exp} v,$$
maps diffeomorphically a neighborhood of $S$ inside $NS$ to a neighborhood $U$ of $S$ in $M$. 
Moreover $d\varphi \rest{S}=\operatorname{Id}$. Let $\pi:NS \to S$ be the projection and
define the function $g(x):=\langle \nu( \pi(\varphi^{-1}(x))   ), \varphi^{-1}(x) \rangle,$ for $x\in U$.
Denote with $\nabla^{\bot}g$ the normal component of the gradient over $S$. This is computed by 
differentiating $g$ along curves in the form $t \mapsto \operatorname{exp}(tw)=\varphi(tw)$ for $w\in NS$. It follows immediately: $\nabla^{\bot}g=\nu$. 
\end{proof}

\section{Measure theory and algebraic geometry}\label{sec:Measures}
In this section we will define the correspondence between homogeneous polynomials and probability measures on their zero sets. 
We will begin by studing families of measures depending on a finite dimensional parameter space.

The setting we will consider is, roughly speaking, the following. We are given a compact Riemannian manifold $(M, g)$, a manifold $P$,
and a smooth map 
$\pi:M\to P$. If $p\in P$ is a regular value of $\pi$, then $\pi^{-1}(p)$ is a smooth, compact submanifold of $M$, which inherits a Riemannian structure with its (normalized) volume measure $\mu(p)\in \mathscr{P}(M)$ that equals the Hausdorff measure of its dimension. Some care has to be put if one wants to extend the map $p\mapsto \mu(p)$ to the whole $P$ (i.e. also to the set of critical values of $\pi$) in a nice way. In fact, even if $\mu:\{\textrm{regular values of $\pi$}\}\to \mathscr{P}(M)$ is continuous, it is not always possible to extend it to a continuous map without imposing additional conditions on $\pi$ (see the example in \cref{remark:notC}).

In the case $M$ and $P$ are complex manifolds of the same dimension and $\pi$ is holomorphic, the so-called ``finiteness'' of $\pi$ guarantees a continuous extension of $\mu$ (we discuss this in \cref{sec:familieszero}). In fact, by \cref{remark:notC}, the finiteness condition is also necessary for the continuous extension. The key ingredient here is the notion of \emph{multiplicity} from complex algebraic geometry.

In the case $\dim(M)>\dim(P)$, in general, the situation is more involved. However, when $M=\CP^n$, using integral geometry arguments, one can provide the continuous extension for interesting classes of maps, including the one parametrizing complex hypersurfaces of a given degree (see \cref{sec:familiesgeneral}).

In this framework, the problem of the continuous extension can also be treated directly using a variation of the arguments from  \cite{Stoll}, but we prefer to keep the exposition self--contained as the description of the extension using integral geometry will become useful in \cref{sec:abscont}.

\subsection{Families of zero-dimensional measures}\label{sec:familieszero}
We say that a holomorphic map $\pi:V\to P$ between smooth, complex, compact and connected manifolds of the same dimension is \emph{finite} if for every $p\in P$ we have $\#\pi^{-1}(p)<\infty.$ We denote by $\Sigma(\pi)\subset V$ the set of critical points of $\pi$, and by $\Delta(\pi):=\pi(\Sigma(\pi))$ the set of its critical values (the ``discriminant''). The degree $\deg(\pi)$ is defined as $\deg(\pi):=\#\pi^{-1}(p)$ for $p\in P\setminus \Delta(\pi)$ (the set of critical points is an analytic set of complex codimension one, therefore $P\setminus \Delta(\pi)$ is connected). For every $v\in V$ we denote by $\mathrm{mult}_v(\pi)$ the multiplicity of $\pi$ at $v$. The multiplicity satisfies
\beq \deg(\pi)=\sum_{v\in \pi^{-1}(p)}\mathrm{mult}_v(\pi), \quad \forall p\in P.\eeq We refer to \cite[Chapter 5.2]{GH} for elementary properties of finite maps. 
\begin{lemma}\label{lemma:cont}Let $\pi:V\to P$ be a finite map between smooth, compact, complex manifolds of the same dimension. Then the map $\alpha:P\to \mathscr{P}(V)$, defined by
\begin{equation} \alpha(p):=\frac{1}{\deg(\pi)}\sum_{v\in \pi^{-1}(p)}\mathrm{mult}_{v}(\pi) \delta_v,\end{equation}
is continuous.
\end{lemma}
\begin{proof}The topology on $\mathscr{P}(V)$ is metrizable and it is enough to prove sequential continuity of $\alpha$: given $\{p_n\}_{n\in \mathbb{N}}\subset P$ converging to $p\in P$, the sequence $\{\alpha(p_n)\}_{n\in \mathbb{N}}$ converges to $\alpha(p)$. In other words, we need to prove that $\forall f\in \mathscr{C}^0(V)$
\beq \label{eq:lim1}\lim_{n\to \infty}\frac{1}{\deg(\pi)}\sum_{v\in \pi^{-1}(p_n)}\mathrm{mult}_{v}(\pi)f(v)=\frac{1}{\deg(\pi)}\sum_{v\in \pi^{-1}(p)}\mathrm{mult}_{v}(\pi) f(v)\eeq

Let us consider first the case $\{p_n\}_{n\in \N}\subset P\setminus \Delta(\pi).$ Let $p=\lim_n p_n$ and set $\pi^{-1}(p)=\{v_1, \ldots, v_r\}$, with $m_j:=\mathrm{mult}_{v_j}(\pi)$. For every $\epsilon>0$ there exists an open neighborhood $U$ of $p$ such that $\pi^{-1}(U)=V_1\sqcup\cdots \sqcup V_r$ and for every $j=1, \ldots, r$ and for every $v\in V_j$
\beq |f(v)-f(v_j)|\leq \epsilon.\eeq
For every $n\in \N$ we denote by 
$v_{n, j}^1, \ldots, v_{n, j}^{m_j}$ the (distinct!) preimages of $p_n$ in $V_j$. Notice in particular that, since $p_n\notin \Delta(\pi)$, we have $\mathrm{mult}_{v_{n,j}^k}(\pi)\equiv1$. Then
\begin{align}
\left| \sum_{v\in \pi^{-1}(p_n)}\mathrm{mult}_v(\pi)f(v)-\sum_{v\in \pi^{-1}(p)}\mathrm{mult}_v(\pi)f(v)\right|&=
{
\left|\sum_{j=1}^r\left[ \left(\sum_{k=1}^{m_j}f(v_{n,j}^k)\right)-m_j f(v_j)\right]\right|}\\
&=\left|\sum_{j=1}^r\sum_{k=1}^{m_j}\left(f(v_{n,j}^k)-f(v_j)\right)\right|\\
&\leq  \deg(\pi)\epsilon.
\end{align}
This proves \eqref{eq:lim1} in the case $\{p_n\}_{n\in \N}\subset P\setminus \Delta(\pi).$ For the general case we argue as follows. For every $n\in \N$ let $\{p_{n, m}\}_{m\in \N}\subset P\setminus \Delta(\pi)$ be a sequence such that $\lim_{m}p_{n,m}=p_n$. For every $\epsilon>0$ let $p_{n, m_\epsilon}=:\widetilde p_n$ such that $\lim_n \widetilde p_n=p$ and 
\beq \left|\sum_{v\in \pi^{-1}(\widetilde p_n)}f(v)-\sum_{v\in \pi^{-1}(p_n)}\mathrm{mult}_v(\pi)f(v)\right|\leq \epsilon.\eeq
Then:
\begin{align} \left| \sum_{v\in \pi^{-1}(p_n)}\mathrm{mult}_v(\pi)f(v)-\sum_{v\in \pi^{-1}(p)}\mathrm{mult}_v(\pi)f(v)\right|\leq& \bigg| \sum_{v\in \pi^{-1}(\widetilde p_n)}\mathrm{mult}_v(\pi)f(v)+\\
&-\sum_{v\in \pi^{-1}(p_n)}\mathrm{mult}_v(\pi)f(v)\bigg| +\\
&+\bigg| \sum_{v\in \pi^{-1}(p_n)}\mathrm{mult}_v(\pi)f(v)+\\
&-\sum_{v\in \pi^{-1}(p)}\mathrm{mult}_v(\pi)f(v)\bigg|.\end{align}
By the previous part of the proof, both summands go to zero as $n\to\infty$, and this concludes the proof. 
\end{proof}
 
 \subsection{Polynomials in one variable}We will now focus on the case of the family of measures given by the normalized sum of the delta measures centered at the zeroes of a polynomial. Our ``parameter'' space will be the (projectivization of the) complex vector space $\C[z_0, z_1]_{(d)}$ of homogeneous polynomials of degree $d$ in two variables.  
 
 Every nonzero homogeneous polynomial $p\in \C[z_0, z_1]_{(d)}$ can be written as
 \beq\label{eq:factorhomo} p(z_0, z_1)=\prod_{j=1}^r(\lambda_{j,1} z_0-\lambda_{j, 0}z_1)^{m_j},\eeq
 where $[\lambda_{j,0},\lambda_{j,1}]$, $j=1, \ldots, r$ are distinct points in $\CP^1$ and $m_1+\cdots +m_r=d$.
 
 With this notation, if $w=[w_0, w_1]\in \CP^1$, we denote by $m_w(p)$ the \emph{multiplicity} of $p$ at $w$. In other words, $m_w(p)=m_j$, if $w=[\lambda_{j,0}, \lambda_{j,1}]$ for some $j=1, \ldots, r$, and $m_w(p)=0$ otherwise. In particular, $m_w(p)\neq 0$ only if $w$ is a zero of $p$, in which case it is the multiplicity of it as a projective zero of $p$.
For a nonzero polynomial $p\in \C[z_0, z_1]_{(d)}$  we denote by $Z(p)\subset \CP^1$ the set of its zeroes, i.e., writing $p$ as in \eqref{eq:factorhomo}, we set 
\beq Z(p):=\{[\lambda_{j,0}, \lambda_{j,1}]\in \CP^1 , j=1, \ldots, r\}.\eeq

\begin{remark}[Homogeneous versus non-homogeneous polynomials]

The reader might prefer to work with polynomials in one variable of degree \emph{at most} $d$, whose zero set lives in the complex line $\C$. Every monic polynomial of the form
 \beq q(z)=z^d+a_{d-1}z^{d-1}+\cdots +a_1 z+a_0\eeq
 can be written (by the Fundamental Theorem of Algebra) as
 \beq q(z)=\prod_{j=1}^d(z-\lambda_j),\eeq
 for some numbers $\lambda_1, \ldots, \lambda_d\in \C$ (the zeroes of $q$). Homogenizing $q$ we get the polynomial
 \beq p(z_0, z_1)=z_1^d+a_{d-1}z_0z_1^{d-1}+\cdots +a_1 z_0^{d-1}z_1+a_0 z_0^d,\eeq
 whose zeroes in $\CP^1$ are $[1, \lambda_1], \ldots, [1, \lambda_d]$. Notice that $p$ can be written as $p(z_0, z_1)=\prod_{j=1}^d(z_1-\lambda_j z_0)$. The advantage of working with homogenous polynomials is that their zeroes live in a compact space (the projective line). In the case of homogeneous polynomials  in two variables it is still safe to think at their dehomogenization and look at their zeroes in $\C$ (notice, though, that if we start with a polynomial with $a_d=0$ some zeroes are lost in this procedure). However, for our purposes, when moving to more variables the reader should really think in terms of projective zeroes and pay attention to the fact that the set of zeroes of a complex polynomial in $\C^n$, $n\geq 2$, is not compact (see also \cref{sec:homosec}).
 \end{remark}

From \cref{lemma:cont} we deduce the following corollary (this also proved in \cite[Section 4.3.12]{federer}).

\begin{corollary}\label{corollary:CP1} Let $\CP^d$ be the projectivization of $\C[z_0, z_1]_{(d)}\simeq \C^{d+1}$. The map $\mu:\CP^d\to \mathscr{P}(\CP^1)$ defined by
\beq\label{eq:map1} \mu(p):=\frac{1}{d}\sum_{z\in Z(p)}m_{z}(p)\delta_z\eeq
is continuous.
\end{corollary}
\begin{proof}Let $P=\CP^d$ and consider the smooth algebraic set $V=\{(p,z)\in P\times \CP^1\,|\, p(z)=0\}$ and the diagram:
\beq\begin{tikzcd}
  & V \arrow[ld, "\pi_1"'] \arrow[rd, "\pi_2"] &       \\
P &                                            & \CP^1
\end{tikzcd}\eeq
where $\pi_1$ and $\pi_2$ are the restrictions of the projections on the two factors. The map $\pi_1:V\to P$ is a finite map between smooth, compact, complex manifolds and $\deg(\pi_1)=d.$ Moreover, using the notation of \cref{lemma:cont}, we see that $\mu(p)=(\pi_2)_{\#}\alpha(p)$. Since $\pi_2$ is continuous, the continuity of $\mu$ follows from \cref{lemma:cont}.
\end{proof}

\subsection{Families of hypersurfaces}\label{sec:familiesgeneral} We now generalize  the previous construction to homogeneous polynomials of more then two variables, whose zero sets are hypersurfaces in the projective space.

For a smooth Riemannian manifold $(M, g)$, we denote by $\mathrm{vol}_M$ the corresponding Riemannian volume measure (hence omitting the dependence on the Riemannian metric $g$ in the notation). In the case $M=\CP^n$, we denote by $g_{\mathrm{FS}}$ the Fubini--Study metric (see \cref{Appendix A}). 

\begin{definition}[The space of polynomials and the discriminant] For $n, d\in \N$, we denote by 
$$H_{n,d}:=\C[z_0, \ldots, z_n]_{(d)}$$
the space of homogeneous polynomials of degree $d$, a vector space of complex dimension $N+1={n+d\choose d}$. We also consider the \emph{discriminant} $D_{n,d}\subset H_{nd}$ defined by
$$D_{n,d}:=\left\{p\in H_{n,d}\,\bigg|\, \exists z\in \C^{n+1}\setminus \{0\},\, p(z)=\frac{\partial p}{\partial z_0}(z)=\cdots=\frac{\partial p}{\partial z_n}(z)=0\right\},$$
i.e. the set of polynomials $p$ such that the equation $p=0$ is not regular in $\CP^n$. 
Notice that the discriminant is an algebraic set in $H_{n,d}$ defined by a polynomial of degree $(n+1)(d-1)^{n-1}.$
We also denote by $P_{n,d}$  and $\Delta_{n,d}$ the projectivization of $H_{n,d}$ and $D_{n,d}$: 
$$P_{n,d}:=\mathrm{P}(H_{n,d})\simeq \CP^N\quad \textrm{and}\quad \Delta_{n,d}:=\mathrm{P}(D_{n,d}).$$
Notice that, given a point $[p]\in P_{n,d}$, its projective zero set
$$Z(p):=\{[x]\in \CP^n\,|\, p(x)=0\}$$ is well-defined  (for this reason we sometimes simply write $p$ for $[p]$).
\end{definition}

Let us make a couple of comments on the previous definitions. If $p\in P_{n,d}$, regardless of it being on the discriminant or not, the set $Z(p)\subset \CP^n$ is a (possibly singular) compact, complex hypersurface, which has Hausdorff dimension $2n-2$. If $p\in P_{n,d}\setminus \Delta_{n,d}$, then the set $Z(p)$ is a smooth complex manifold. If $p\in \Delta_{n,d}$ then, as a subset of $\CP^n$, $Z(p)$ can still be smooth. For example, if $p\in P_{n,d}\setminus \Delta_{n,d}$, then $p^2\in \Delta_{n,2d}$, but $Z(p^2)=Z(p)$, as a set, is a smooth complex submanifold. In order to emphasize this phenomenon, we will say that $Z(p)$ is \emph{regular} if $p\in P_{n,d}\setminus \Delta_{n,d}$ and \emph{singular} otherwise (thus being ``regular'' has to do with the equation defining a set, and not with the set itself). If $p\in \Delta_{n,d}$ we will denote by $Z(p)^{\mathrm{sing}}$ the set
$$Z(p)^{\mathrm{sing}}=\left\{[z]\in Z(p)\,\bigg |\, \frac{\partial p}{\partial z_0}(z)=\cdots=\frac{\partial p}{\partial z_n}(z)=0\right\}.$$
It can be that $Z(p)^{\mathrm{sing}}=Z(p)$ (for instance, this happens again for the polynomials of the form $p^2$). On the other hand, if $p\in P_{n,d}$ is irreducible, then $Z(p)^\mathrm{sing}$ is a proper algebraic subset of $Z(p)$, therefore of complex codimension one in it, and 
$$Z(p)^{\mathrm{reg}}:=Z(p)\setminus Z(p)^{\mathrm{sing}}$$
 is a smooth submanifold of $\CP^n$. 

In the rest we will often use \emph{semialgebraic} objects, whose definition we recall now. For more details the reader is referred to \cite{BCR}.

\begin{definition}[Semialgebraic sets and functions]\label{def:semialg} A semialgebraic set $S\subseteq \R^n$ is a set obtained by taking finite unions, finite intersections of sets of the form $\{p\leq 0\}$ and their complements,
where $p\in \R[x_1, \ldots, x_n]$ is a polynomial. A function $f:A\to B$, where $A\subset \R^n, B\subset \R^m$ are semialgebraic, is said to be semialgebraic if its graph $\Gamma(f)\subset A\times B\subset \R^{n+m}$ is a semialgebraic set.
\end{definition}

By \cite[Proposition 2.4.6]{BCR}, there exists a smooth, real algebraic embedding $\varphi:\CP^n\to \R^\ell$, for some $\ell\in \mathbb{N}$ sufficiently large. Therefore we can view $\CP^n$ as a semialgebraic set itself. Since complex algebraic set are in particular real algebraic sets (every complex equation corresponds to a pair of algebraic equations), every complex algebraic set $Z\subseteq \CP^n$ is therefore semialgebraic. Therefore we can stratify $Z$ as a disjoint union 
\beq \label{eq:strata}Z=\bigsqcup_{j=1}^s Z_j,\eeq
where each $Z_j\subset \CP^n$ is a smooth semialgebraic set (see \cite[Theorem 9.1.8]{BCR}).

\begin{definition}[Volume of a complex algebraic set]\label{def:volume} Let $Z\subset \CP^n$ be a complex algebraic set, stratified as \eqref{eq:strata}. Each stratum inherits a Riemannian metric $g_{\mathrm{FS}}|_{Z_j}$ and hence a volume measure $\mathrm{vol}_{Z_j}$.  We define the volume measure $\mathrm{vol}_Z\in \mathscr{P}(\CP^n)$ as follows. For a continuous function $f:Z\to \R$ we set
\beq \int_{Z}f(z)\mathrm{vol}_{Z}(\dd z):=\sum_{\dim(Z_j)= 2n-2}\int_{Z_j}  f(z)\mathrm{vol}_{Z_j}(\dd z).\eeq
Since any two semialgebraic stratifications can be commonly refined (see \cite[Theorem 9.1.8]{BCR}), this definition does not depend on the choice of the stratification from \eqref{eq:strata}.
We denote by $\mathrm{vol}(Z)$ the total mass of $\mathrm{vol}_Z$.
\end{definition}

In the case $p\in P_{n,d}\setminus \Delta_{n,d}$ (i.e. if the hypersurface $Z(p)\subset \CP^n$ is  \emph{regular}), then for every Borel set $A\subseteq Z(p)$, 
\beq \mathrm{vol}_{Z(p)}(A)=\frac{1}{(n-1)!}\int_{A}\omega_{\mathrm{FS}}^{n-1},\eeq
where $\omega_{\mathrm{FS}}$ denotes the K\"ahler form of $\CP^n$ (see \cref{Appendix A}).

We observe that \cref{def:volume} does not produce a continuous map in the Wasserstein topology, i.e. the correspondence $p\mapsto \frac{1}{\mathrm{vol}(Z(p))}\mathrm{vol}_{Z(p)}$ is not continuous on $\Delta_{n,d}$. 
As we have seen in \cref{corollary:CP1}, a continuous extension on $\Delta_{1,d}$ should take into account the multiplicities of zeroes. In the general case we proceed as follows. First, in \cref{lemma:dis}, we give an integral geometric formulation of the volume, for smooth submanifolds of $\CP^n$. Then, in \cref{thm:ext}, we use this formulation to extend the map $\mu$ to $\Delta_{n,d}$.

We denote by $\mathbb{G}(1,n)$ the Grassmannian of projective lines in $\CP^n$ (notice that $\mathbb{G}(1, n)$ is isomorphic to $G(2, n+1),$ the Grassmannian of complex $2$--planes in $\C^{n+1}$). There is a natural action of $U(n+1)$ on $\mathbb{G}(1,n)$, and we endow $\mathbb{G}(1,n)$ with a Riemannian metric which is invariant under this action, normalized such that $\mathrm{vol}_{\mathbb{G}(1,n)}(\mathbb{G}(1,n))=1.$  

\begin{lemma}\label{lemma:dis}Let $f\in \mathscr{C}^0(\CP^n)$ and $Z\subset \CP^n$ be a complex submanifold of complex dimension $n-1$. Then for every $U\subseteq Z$ open we have
\beq\label{eq:ig} \int_{U}f(z) \mathrm{vol}_{Z}(\dd z)=\mathrm{vol}(\CP^{n-1})\int_{\mathbb{G}(1, n)}\left(\sum_{z\in U\cap \ell}f(z)\right)\mathrm{vol}_{\mathbb{G}(1,n)}(\dd \ell),\eeq
the integrand being a.e. finite.
In particular, if $p\in P_{n,d}\setminus \Delta_{n,d}$, then $\mathrm{vol}_{Z(p)}(Z(p))=d\cdot \mathrm{vol}(\CP^{n-1}).$
\end{lemma}
\begin{proof}Let us recall first the kinematic formula in complex projective spaces from \cite[Example 3.12 (b)]{Howard}. Given two open subsets $A\subseteq Z, B\subseteq L$ of complex submanifolds $Z, L$ of complementary dimensions in $\CP^n$, $\dim_\C (Z)=a$, $\dim_\C(L)=b$, we have:
\beq \label{eq:igfc}\int_{U(n+1)}\#\left(A\cap gB\right) \mathrm{vol}_{U(n+1)}(\dd g)=\frac{\mathrm{vol}_Z(A)}{\mathrm{vol}_{\CP^a}(\CP^a)}\frac{\mathrm{vol}_L(B)}{\mathrm{vol}_{\CP^b}(\CP^b)},\eeq 
where ``$\dd g$'' denotes the integration with respect to the normalized Haar measure $\int_{U(n+1)} \dd g=1$ (which is obtained as the volume form of a left--invariant Riemannian metric on $U(n+1)$).
Applying this to the special case  $B=\CP^1$, 
we get
\beq \label{eq:igfc2}\int_{U(n+1)}\#\left(A\cap gB\right) \mathrm{vol}_{U(n+1)}(\dd g)=\frac{\mathrm{vol}_Z(A)}{\mathrm{vol}_{\CP^{n-1}}(\CP^{n-1})}.\eeq 
In the case $A = Z=Z(p)$ is regular,
the integrand is equal 
to $d$ for a.e. $g$ (see \cite{GH}) and
\beq \label{eq:igf3}d=\frac{\mathrm{vol}_Z(Z)}{\mathrm{vol}_{\CP^{n-1}}(\CP^{n-1})},\eeq 
proving the second part of the statement.

Observe now that the map $U(n+1)\to \mathbb{G}(1,n)$ given by $g\mapsto g( \ell_0)$, where $\ell_0\in \mathbb{G}(1, n)$ is a fixed line, is a Riemannian submersion. Therefore, in the case $A\subseteq Z$ is open and $B=\CP^1$, using the smooth coarea formula, \eqref{eq:igfc} can be written as
\beq\label{eq:igf2} \int_{\mathbb{G}(1,n)}\#(A\cap\ell)\mathrm{vol}_{\mathbb{G}(1,n)}(\dd \ell)=\frac{\mathrm{vol}_Z(A)}{\mathrm{vol}_{\CP^{n-1}}(\CP^{n-1})}.\eeq
Consider now the linear functional $\lambda:\mathscr{C}^0(Z)\to \R$ defined by:
\beq f\mapsto \mathrm{vol}_{\CP^{n-1}}(\CP^{n-1})\int_{\mathbb{G}(1, n)}\left(\sum_{z\in Z\cap \ell}f(z)\right)\mathrm{vol}_{\mathbb{G}(1,n)}(\dd \ell).\eeq
This is a continuous and positive linear functional on $\mathscr{C}^0(Z)$, therefore by Riesz's Representation Theorem \cite[Theorem 2.14]{Rudin}, there exists a unique nonnegative Radon measure $\nu$ such that for every $f:Z\to \R$ Borel measurable:
\beq \lambda(f)=\int_{Z}f(z)\nu(\dd z).\eeq
Let now $U$ be an open set of $Z$ and $f=\chi_U$ be its characteristic function. Approximating $f$ with continuous functions, we see that
\begin{align} \nu(U)&=\lambda(f)={\mathrm{vol}_{\mathbb{C}P^{n-1}}(\mathbb{C}P^{n-1})}\int_{\mathbb{G}(1, n)}\left(\sum_{z\in Z\cap \ell}f(z)\right)\mathrm{vol}_{\mathbb{G}(1,n)}(\dd \ell)\\
&={\mathrm{vol}_{\mathbb{C}P^{n-1}}(\mathbb{C}P^{n-1})}\int_{\mathbb{G}(1, n)}\#(U\cap \ell)\mathrm{vol}_{\mathbb{G}(1,n)}(\dd \ell)\\
&={\mathrm{vol}_Z(U)}.\quad \textrm{(by \eqref{eq:igf2})}\end{align}
This shows that, as measures, $\nu=\mathrm{vol}_{Z}$ and \eqref{eq:ig} follows.
\end{proof}

We have now a well defined map 
\beq \mu:P_{n,d}\setminus \Delta_{n,d}\longrightarrow \mathscr{P}(\CP^n)\eeq 
given by 
\beq \label{eq:tbe}\mu(p)(f):= \frac{1}{\mathrm{vol}(Z(p))}\int_{Z(p)}f(z)\mathrm{vol}_Z(\dd z), \quad f\in \mathscr{C}^0(\CP^n).\eeq
We prove now that this map can be continuously extended to the whole $P_{n,d}$, generalizing \cref{corollary:CP1}.
Notice that the normalization factor $\frac{1}{\mathrm{vol}(Z(p))}$ in \eqref{eq:tbe} does not actually depend on the hypersurface, because of \cref{lemma:dis}.

\begin{theorem}\label{thm:ext}
The map  $\mu:P_{n,d} \to \mathscr{P}(\CP^n)$ defined for $f\in \mathscr{C}^0(\CP^n)$ by
\beq
\label{eq:ext} \int_{\CP^n}f \dd \mu(p):=\frac{1}{d}\int_{\mathbb{G}(1,n)}\left(\sum_{z\in Z(p)\cap \ell}m_z(p|_{\ell}) f(z)\right)\mathrm{vol}_{\mathbb{G}(1,n)}(\dd \ell)
\eeq
is continuous and, when restricted to $P_{n,d}\setminus \Delta_{n,d}$ coincides with \eqref{eq:tbe}. 
\end{theorem}

\begin{proof}The fact that the restriction of $\mu$ to $P_{n,d}\setminus \Delta_{n,d}$ coincides with \eqref{eq:tbe} is an immediate consequence of \cref{lemma:dis}. For the continuity of $\mu$, as above, it is enough to prove sequential continuity. Let therefore $\{p_k\}_{k\in\N}\subset P_{n,d}$ be a sequence converging to $p$. For every $k\in \N$ denote by $F_k\subset \mathbb{G}(1,n)$ the set of lines contained in $Z(p_k)$ and by $F$ the set of lines contained in $Z(p)$.  Notice that these sets have measure zero in $\mathbb{G}(1,n)$, therefore the complement of their union 
\beq \widetilde{\mathbb{G}}:=\mathbb{G}(1, n)\setminus \left(F\cup\bigcup_{k\in \N}F_k\right)\eeq
has full measure in $\mathbb{G}(1,n)$. Given $f\in \mathscr{C}^0(\CP^n)$ and for every $k\in \N$ denote by $g_k, g:\mathbb{G}(1,n)\to \R$ the functions defined for $\ell\in \widetilde{\mathbb{G}}$ by:
\beq g_k(\ell):=\frac{1}{d}\sum_{Z\in Z(p_k)\cap \ell}m_z(p|_{\ell}) f(z)\quad \textrm{and}\quad g(\ell):=\frac{1}{d}\sum_{Z\in Z(p)\cap \ell}m_z(p|_{\ell}) f(z),\eeq
and set these function to zero on $\widetilde{\mathbb{G}}^c$. Then for every $k\in \N$ we have
\beq \int_{\CP^n}f\dd\mu(p_k)=\int_{\mathbb{G}(1,n)}g_k(\ell)\mathrm{vol}_{\mathbb{G}(1,n)}(\dd \ell)\quad \textrm{and}\quad \int_{\CP^n}f\dd\mu(p)=\int_{\mathbb{G}(1,n)}g(\ell)\mathrm{vol}_{\mathbb{G}(1,n)}(\dd \ell).\eeq 
The sequence of functions $\{g_k\}_{k\in \N}$ converges pointwise to $g$ by \cref{corollary:CP1}. Moreover the whole sequence is dominated by $|g_k|\leq \max_{z\in \CP^n}|f(z)|$, therefore the convergence $\mu(p_k)\to \mu(p)$ follows from the Dominated Convergence Theorem.
\end{proof}
We prove now that $\mu:P_{n,d}\to \mathscr{P}(\CP^n)$ is injective. We need a couple of lemmas first.

\begin{lemma}\label{lemma:irreducible}If $p\in \C[z_0, \ldots, z_n]_{(d)}$ is irreducible, then $\mu(p)=\frac{1}{d\cdot \mathrm{vol}(\CP^{n-1})}\mathrm{vol}_{Z(p)^{\textrm{reg}}}.$\end{lemma}
\begin{proof}If $p$ is irreducible, then $Z(p)^\textrm{sing}$ has complex codimension at least $2$ in $\CP^n$ (see \cite[Chapter 0.2]{GH}). In particular, almost every projective line $\ell\in \mathbb{G}(1, n)$ misses it; since the condition ``$\ell $ misses $Z(p)^\textrm{sing}$'' is semialgebraic, it follows that the set $D_1\subset \mathbb{G}(1,n)$ of lines missing $Z(p)^\textrm{sing}$ is a full-measure semialgebraic set. Moreover, by Bertini's Theorem (see \cite[Chapter 1.1]{GH}), for the generic line $\ell$ the equation $p|\ell=0$ is regular. Again, since the condition ``$p|_\ell=0$ is regular'' is semialgebraic, the set $D_2$ of lines $\ell$ such that $p|_{\ell}=0$ is regular is also full-measure and semialgebraic. (Notice that $p|_\ell=0$ is regular if and only if for every $z\in Z(p)\cap \ell$ we have $m_z(p|_\ell)=1$). 

In particular, given $f\in \mathscr{C}^0(\CP^n)$, we have
\begin{align}\int_{\CP^n}f \dd \mu(p)&=\frac{1}{d}\int_{\mathbb{G}(1,n)}\left(\sum_{z\in Z(p)\cap \ell}m_z(p|_{\ell}) f(z)\right)\mathrm{vol}_{\mathbb{G}(1,n)}(\dd \ell)\\
&=\frac{1}{d}\int_{D_1\cap D_2}\left(\sum_{z\in Z(p)\cap \ell}m_z(p|_{\ell}) f(z)\right)\mathrm{vol}_{\mathbb{G}(1,n)}(\dd \ell),&
\end{align}
since $D_1\cap D_2$ is full measure. This is also equal to
\begin{align}
&=\frac{1}{d}\int_{D_1\cap D_2}\left(\sum_{z\in Z(p)\cap \ell} f(z)\right)\mathrm{vol}_{\mathbb{G}(1,n)}(\dd \ell)&\textrm{($\ell \in D_1$)}\\
&=\frac{1}{d}\int_{D_1\cap D_2}\left(\sum_{z\in Z(p)^\mathrm{reg}\cap \ell}f(z)\right)\mathrm{vol}_{\mathbb{G}(1,n)}(\dd \ell)&\textrm{($\ell \in D_2$)}\\
&=\frac{1}{d}\int_{\mathbb{G}(1,n)}\left(\sum_{z\in Z(p)^\mathrm{reg}\cap \ell}f(z)\right)\mathrm{vol}_{\mathbb{G}(1,n)}(\dd \ell)&\textrm{($D_1\cap D_2$ is full measure)}\\
&=\frac{1}{d}\int_{Z(p)^{\textrm{reg}}}f (z)\mathrm{vol}_{Z(p)^{\textrm{reg}}}(\dd z)&\textrm{(by \cref{lemma:dis})}.
\end{align}
This concludes the proof.
\end{proof}

We need also the following lemma.

\begin{lemma}\label{lemma:factorize}Let $p=p_1^{m_1}\cdots p_k^{m_k}\in \C[z_0, \ldots, z_n]_{(d)}$, with $p_j\in \C[z_0, \ldots, z_n]_{(d_j)}$ irreducible and $m_j\in \mathbb{N}$. Then 
\beq \mu(p)=\frac{1}{d\cdot \mathrm{vol}(\CP^{n-1})}\sum_{j=1}^km_k \cdot \mathrm{vol}_{Z(p_j)^{\textrm{reg}}}.\eeq
\end{lemma}
\begin{proof}Using \eqref{eq:ext}, for a given $f\in \mathscr{C}^0(\CP^n)$, we have:
\begin{align} \int_{\CP^n}f \dd \mu(p)&=\frac{1}{d}\int_{\mathbb{G}(1,n)}\left(\sum_{z\in Z(p)\cap \ell}m_z(p|_{\ell}) f(z)\right)\mathrm{vol}_{\mathbb{G}(1,n)}(\dd \ell)\\
&=\frac{1}{d}\int_{\mathbb{G}(1,n)}\left(\sum_{j=1}^k \sum_{z\in Z(p_j)\cap \ell}m_k m_z(p_j|_{\ell}) f(z)\right)\mathrm{vol}_{\mathbb{G}(1,n)}(\dd \ell)\\
&=\sum_{j=1}^km_k\frac{1}{d}\int_{\mathbb{G}(1,n)}\left( \sum_{z\in Z(p_j)\cap \ell}m_z(p_j|_{\ell}) f(z)\right)\mathrm{vol}_{\mathbb{G}(1,n)}(\dd \ell)\\
&=\sum_{j=1}^k m_k \int_{Z(p_j)^{\mathrm{reg}}}f(z) \mathrm{vol}_{Z(p_j)^{\textrm{reg}}}(\dd z)&\textrm{(by \cref{lemma:irreducible})}.
\end{align}
Since almost every lines misses the common intersection $Z(p_j)\cap Z(p_j)$ (which have complex codimension at least $2$), this proves the claim. \end{proof}
\begin{proposition}\label{propo:compactness}The map $\mu:P_{n,d}\to \mathscr{P}(\CP^n)$ given by \eqref{eq:ext} is injective. In particular $\mu$ is a homeomorphism onto its image.
\end{proposition}
\begin{proof}Let $[p]\neq [q]$, so that $p\neq \lambda q$ for $\lambda \in \C\setminus \{0\}$. Let us write $p=p_1^{m_1}\cdots p_k^{m_k}$ and $q=q_1^{r_1}\cdots q_h^{r_h}$ with the $p_j$ and the $q_i$ irreducible and coprime.

Then the sets $\{Z(p_i)\}$ and $\{Z(q_j)\}$ are the irreducible components of $Z(p)$ and $Z(q)$ respectively. The irreducible components are uniquely determined, therefore, if the sets $\{[p_1], \ldots,[ p_k]\}$ and $\{[q_1], \ldots, [q_h]\}$ are different, then the supports of the measures $\mu(p)$ and $\mu(q)$ are different (as the two polynomials have different zero sets) and $\mu([p])\neq \mu([q]).$ 

If $\{[p_1], \ldots,[ p_k]\}=\{[q_1], \ldots, [q_h]\}$, then $h=k$ and, up to relabeling, $[p_j]=[q_j]$. Since $[p]\neq [q]$ there must be $j\in \{1, \ldots, k\}$ such that $m_j\neq r_j$ and the fact that $\mu([p])\neq \mu([q])$ follows from \cref{lemma:factorize}.

Since $\mu$ is continuous and injective and $P_{n,d}$ is compact, it is a homeomorphism onto its image.
\end{proof}
\begin{remark}\label{remark:notC}The construction for algebraic hypersurfaces is a special case of a more general framework. Consider the set $V_{n,d}:=\{([p], [z])\in  P_{n,d}\times \CP^n\,|\, p(z)=0\}.$ The projection $\pi_2:V_{n,d}\to \CP^n$, whose fibers consist of the set of polynomials vanishing at a given point (a projective hyperplane in the space of polynomials), turns $V$ into a fiber bundle over $\CP^n$; in particular $V_{n,d}\subset  P_{n, d}\times \CP^n$ is a smooth complex submanifold. The projection on the first factor $\pi_1:V_{n,d}\to P_{n,d}$ has fibers which are precisely the algebraic hypersurfaces of degree $d$. Endow $\CP^n$ and $P_{n,d}$ (which are both complex projective spaces) with the Fubini--Study Riemannian structure (see \cref{Appendix A}), and $V_{n,d}$ with the induced Riemannian structure $g$. 
 Then we can view the map $\mu:P_{n,d}\setminus \Delta_{n,d}\to \mathscr{P}(\CP^n)$ from \eqref{eq:tbe} as:
\beq \mu(p)=(\pi_2)_{\#}\lambda_{\pi_1^{-1}(p)},\eeq
were, in general, if $X\hookrightarrow V$ is a $m$--dimensional, compact and smooth submanifold of a Riemannian manifold $(V, g)$, we denote by $\lambda_{X}\in \mathscr{P}(V)$ the measure defined on $f\in \mathscr{C}^0(V)$ by 
\beq \lambda_X(f)=\frac{1}{\mathrm{vol}_m(X)}\int_{X}f \dd\mathrm{vol}_X.\eeq 
More generally, one can consider a smooth map $\pi:V\to P$, where $(V, g)$ is a Riemannian manifold, and define $\alpha:P\setminus\Delta(\pi)\to \mathscr{P}(V)$ by $\alpha(p):=\lambda_{\pi^{-1}(p)}.$
Even under the assumptions that $V$ and $P$ are complex of the same dimension and $\pi:V\to P$ is holomorphic, it is not possible in general to continuously extend $\alpha$ to the whole $P$. For instance, the conclusions of \cref{lemma:cont} are false if we drop the finiteness assumption on $\pi:V\to P$. To see this, take $P=\CP^2$  and $V=\mathrm{Bl}_{p_0}(\CP^2)$ (the blow--up at one point). Both $P$ and $V$ are compact complex manifolds of the same complex dimension. We have $\Delta(\pi)= \{p_0\}$ and the map $\pi:V\to P$ is finite (one--to--one) only on $V\setminus\pi^{-1}(p_0)$, since $\pi^{-1}(p_0)\simeq \CP^1.$ In this case, for $p\neq p_0$, we have $\alpha(p)=\delta_{\pi^{-1}(p)}$, but for $p_n\to p_0$ the limit $\alpha(p_n)$ depends on the sequence $\{p_{n}\}_{n\in \N}$.
This is because in general there is no notion of ``multiplicity'' for the fibers. In the case of \cref{thm:ext} the manifolds $V_{n,d}$ and $P_{n,d}$ do not have the same dimension (unless $n=1$) but we can still continuously extend $\mu=(\pi_2)_{\#}\circ \alpha$  using integral geometry.
\end{remark}

\section{Symplectic geometry and metric speed of Hamiltonian isotopies}\label{sec:Ham}

To define an optimal transport distance between algebraic hypersurfaces 
we will need  to
study curves inside the image of $\mu$. 
This will be done relying on 
techniques from symplectic geometry. 

Recall that a \emph{symplectic manifold} $(M, \omega)$ is a smooth manifold $M$ equipped with a closed, nondegenerate $2$--form.  The nondegeneracy assumption on $\omega$ implies that for every smooth function $H:M\to \R$ there is a unique vector field $V_H\in \Gamma(TM)$ such that
$$
\dd H= \omega(V_H,\cdot)=\iota_{V_H} \omega.
$$
A non--autonomous vector field $v_t$ on a symplectic manifold $(M, \omega)$ is said to be \emph{Hamiltonian} if there exists a smooth function $H:M\times I\to \R$ (the Hamiltonian) such that, denoting by $h_t:=H(\cdot, t):M\to \R$, we have  $v_t=V_{h_t}$, i.e. $\dd h_t=\iota_{V_{h_t}}\omega$. Notice that the flow $\varphi_t$ of a Hamiltonian vector field $V_{h_t}$ preserves the symplectic form: in fact, by Cartan's magic formula,
$$
\frac{d}{dt}\varphi_t^*\omega
=\varphi_t^* \mathcal{L}_{X_{h_t}}\omega
=\varphi_t^*(\iota_{V_{h_t}} \dd \omega + \dd (\iota_{V_{h_t}}\omega))=0,
$$
since $\omega$ is closed and $\iota_{V_{h_t}}\omega=\dd h_t$.

Recall also that a \emph{K\"ahler manifold} $(M, \omega, J)$ is a symplectic manifold $(M, \omega)$ equipped with an integrable almost--complex structure (i.e. and endomorphims $J:TM\to TM$ with $J^2=-\mathrm{id}_{TM}$  and admitting a holomorphic atlas) which is compatible with the symplectic form. This means that the bilinear form $g:TM\times TM\to \R$ defined by 
\beq g(u,v):=\omega (u, Jv)\eeq
is symmetric and positive definite. In particular $g$ defines a Riemannian metric on $M$, which we call the \emph{K\"ahler Riemannian metric}.

 \subsection{Hamiltonian isotopies}
 
 \begin{definition}[Smooth family of submanifolds]\label{def:smoothfamily} We say that a family $\{S_t\}_{t\in I}$ of submanifolds of a manifold $M$ is a \emph{smooth family} if there exist a smooth map $F:M\times I\to L$, where $I=[0,1]$ and $L$ is a smooth manifold, and a submanifold $A\hookrightarrow L$ such that for every $t\in I$ the map $f_t:=F(\cdot, t):M\to L$ is transverse\footnote{If $f:M\to L$ is trasnverse to $A$, we will write $f\pitchfork A$.} to $A$ and $S_t=f_t^{-1}(A).$
 \end{definition}
 
 Recall that $f_t : M \to L$ transverse to $A$ means that $\dd f_t(T_x M)+ T_{f(x)}A=T_{f_t(x)}L$ at every point $x \in f^{-1}_t(A)$. In this case $f^{-1}_t(A)$ is a submanifold of the same codimension as $A$.

 \begin{remark}\label{rem:thom}A classical result in differential topology (Thom's isotopy lemma) states that if $M$ is compact and $A$ is closed, given a smooth family $\{S_t\hookrightarrow M\}_{t\in I}$ there exists a (not unique) isotopy (i.e. a smooth family of diffeomorphisms with $\varphi_0=\mathrm{id}_M$) $\varphi_t:M\to M$ such that $S_t=\varphi_t(S_0)$. We recall now how the construction of the isotopy $\varphi_t$ is made, using the notion of flow of a \emph{non-autonomous} vector field on $M$. First observe that we can associate to a non-autonomous vector field $v_t$ on $M$, an autonomous vector field $\widehat{v}(x,t):=v_t(x)+\partial_t$ on $M\times I$. The isotopy $\varphi_t$ is obtained by restricting the time--$t$ flow $\widehat{\varphi}_t$ of $\widehat{v}$ to the slice $M\times \{0\}$ (which is sent to the slice $M\times \{t\}\simeq M$).  In order to understand what properties of  $v_t$ would guarantee that $\varphi_t(S_0)=S_t, $ denote by $\widehat{S}:=F^{-1}(A)\subseteq M\times I$. The transversality condition $f_t\pitchfork A$ for every $t\in I$ implies that $\widehat{S}$ is a smooth submanifold of $M\times I$. If the vector field $v_t$ is such that the corresponding vector field $\widehat{v}$  on $M\times I$ is tangent to $\widehat{S}$, then $\widehat{S}$ would be preserved by the flow $\widehat{\varphi}_t$ and $\varphi_t(S_0)=\widehat{\varphi}_t(\widehat{S}\cap(M\times \{0\}))\subseteq \widehat{S}\cap(M\times \{t\})\simeq S_t.$ (The reverse inclusion is obtained by considering the flow of $-\widehat{v}$.) In other words: in order to find an isotopy $\varphi_t:M\to M$ such that $\varphi_t(S_0)=S_t$, it is enough to find a non--autonomous vector field $v_t$ on $M$ such that $\widehat{v}$ is tangent to $\widehat{S}$. A crucial observation is that the condition ``$\widehat{v}$ is tangent to $\widehat{S}\,$'' behaves well under partition of unities: if we can find, for every $z=(x,t)$ a neighborhood $U_z$ of $z$ in $M\times I$ and a vector field $\widehat{v}_z$ on $U_z$ of the form $\widehat{v}_z(x,t)=v_{z,t}(x)+\partial_t$ such that ``$\widehat{v}_z$ is tangent to $\widehat{S}\cap U_z$'', then, if $\{\rho_z\}_{z\in M\times I}$ is a partition of unity subordinated to $\{U_z\}_{z\in M\times I}$, the field $\widehat{v}:=\sum_{z}\rho_z v_z$ is of the form $\widehat{v}(x,t)=\sum_z \rho_z(x,t) v_{z,t}(x)+\partial_t$ and ``$\widehat{v}$ is tangent to $\widehat{S}$''. 
\end{remark}
 
There is also an Hamiltonian version of Thom's isotopy lemma
(see \cite{zbMATH05286973,zbMATH01123717,sieberttian})
which is well known by the symplectic community. A complete proof can be found in \cite{moser}.

 \begin{theorem}\label{thm:thom}Let $(M, \omega)$ be a smooth, compact symplectic manifold and $\{S_t\}_{t\in I}$ be a smooth family of symplectic submanifolds. Then there exists a Hamiltonian flow $\varphi_t:M\to M$ such that $\varphi_t(S_0)=S_t$ for every $t\in I$.
 \end{theorem}
 
\begin{corollary}\label{coro:hamilthom}Let $(M, \omega)$ be a smooth, compact symplectic manifold and $\{S_t\}_{t\in I}$ be a smooth family of symplectic submanifolds. For every $t\in I$ denote by $\mu_t\in \mathscr{P}(M)$ the measure defined by setting for every  $f\in \mathscr{C}^{0}(M, \R)$:
\beq \label{eq:defflow}\int_M f \dd \mu_t:=\frac{\displaystyle \int_{S_t}f \omega^{n-1}}{\displaystyle \int_{S_t} \omega^{n-1}}.\eeq
Then there exists a Hamiltonian flow $\varphi_t:M\to M$ such that $\mu_t=(\varphi_t)_\#\mu_0$ for every $t\in I$.
\end{corollary}

\begin{proof}This follows immediately from \cref{thm:thom}, using the fact that $\varphi_t^*\omega=\omega.$
\end{proof}

\subsection{The metric speed of the isotopy}

We now proceed calculating the metric speed of the curve $(\mu_t)_{t\in [0,1]}$ from  
\cref{coro:hamilthom} using the 
results of Section \ref{S:OT}, under the additional assumption that $M$ is K\"ahler. The connection with Section \ref{S:OT} is given by \cref{propo:metricspeedflow}.

\begin{proposition}\label{propo:ka}Let $(M, \omega, J)$ be a smooth, compact, K\"ahler manifold and $\{S_t\}_{t\in I}$ be a smooth family of codimension--two submanifolds such that $(S_t, \omega|_{S_t}, J|_{S_t})$ is K\"ahler for every $t\in I$. Denote by $\mu_t\in \mathscr{P}(M)$ the measure defined by \eqref{eq:defflow} and by $v_t$ the time--dependent vector field whose flow generates the corresponding Hamiltonian isotopy given by \cref{coro:hamilthom}. For every $t\in I$ let also denote by $T(v_t)$ and $N(v_t)$ the orthogonal projections (with respect to the K\"ahler Riemannian structure) of $v_t$ onto $TS_t$ and $NS_t$ (the normal bundle of $S_t$ in $M$), respectively. Then, $t\mapsto \mu_t$ is absolutely continuous and 
\beq |\dot \mu_t|^2=\int_{S_t}\|N(v_t)\|^2\dd \mu_t.\eeq
\end{proposition}

\begin{proof}Using \cref{propo:metricspeedflow} we get that $|\dot \mu_t|^2=\|\mathsf{P}(v_t)\|_{L^2(\mu_t)}^2$ for almost every $t\in I$, where $\mathsf{P}(v_t)$ denotes the projection of $v_t$ onto
\beq \overline{\{\nabla f \colon f \in \mathscr{C}^{\infty}(M) \}}^{L^{2}(\mu_t)}
= \overline{\{\nabla g \colon g \in \mathscr{C}^{\infty}(S) \}}^{L^{2}(\mu_t)} 
\oplus
L^{2}(NS, \mu_{t}),\eeq
where this last decomposition follows from \cref{propo:decograd}. 

We will prove that $T(v_t)$ is $\mu_t$--divergence free, which implies that $\mathsf{P}(v_t)=N(v_t)$, from which the conclusion follows. To this end, denoting by $h_t:M\to M$ the time--dependent Hamiltonian corresponding to the Hamiltonian flow $\varphi_t$, it will be sufficient to show that $T(v_t)$ (a vector field on $S_t$) is the Hamiltonian vector field of $h_t|_{S_t}$. 

Let therefore $w\in TS_t$. We have
\begin{align}\dd h_t|_{S_t}(w)&=\dd h_t(w)=\omega (T(v_t)+N(v_t), w)\\
&=\omega(T(v_t)+N(v_t), -JJw)& \textrm{(using the fact that $J$ is a complex structure)}\\
&=-g(T(v_t)+N(v_t), Jw)& \textrm{(by definition of $g$)}\\
&=-g(T(v_t), Jw)&\textrm{(since $J|_{S_t}$  is a complex structure on $S_t$)}\\
&=\omega(T(v_t), w)\\
&=\omega|_{S_t}(T(v_t), w).
\end{align}
This  shows that $T(v_t)$ is the Hamiltonian vector field of $h_t|_{S_t}$ and concludes the proof.
\end{proof}

\begin{remark}\label{remnoncompact}
In the previous proposition the compactness of $M$ is used only to know that an Hamiltonian isotopy determining the family of submanifolds can be found. If we know, a priori, such an Hamiltonian isotopy (and the submanifolds of the family stay compact) the result remains of course true.
\end{remark}

\begin{lemma}\label{lemma:invert}
Let $\pi:L\to M$ be a rank--$k$ vector bundle on a Riemannian manifold $(M, g)$. Let also $F:M\times I\to L$ be a smooth map such that for every $t\in I$ the map $f_t:=F(\cdot, t):M\to L$ is a section of $L$ which is transversal to the zero section $A\hookrightarrow L$. (Using the language of \cref{def:smoothfamily} this implies that $\{S_t:=f_t^{-1}(A)\}_{t\in I}$ is a smooth family of submanifolds of $M$.) Let $v_t$ be a time--dependent vector field on $M$ such that $v_t+\partial_t$ is tangent to $\widehat{S}:=F^{-1}(A).$ For every $(x,t)\in \widehat{S}$, denoting by $N(v_t(x))$ the orthogonal projection of $v_t(x)$ on the normal space to $S_t$, we have:
\beq N(v_t(x))=-\big(\pi_2|_{f_t(x)}\circ d_xf_t|_{N_xS_t}\big)^{-1}\frac{\partial F}{\partial t}(x, t),\eeq
where, for every $a\in A$,  the linear map $\pi_2|_{a}:T_a L\simeq T_aA\oplus L_a\to L_a$ is the natural projection and $L_a:=\pi^{-1}(a)$.
\end{lemma}

\begin{proof}Assume $v_t+\partial_t$ is tangent to $\widehat{S}$ at $(x, t)$. Then $D_{(x,t)}F(v_t(x)+\partial_t)\in T_{f_t(x)}A\subset T_{f_t(x)}L$ and, in particular,
\beq \pi_2|_{f_t(x)}\left(D_{(x,t)}F(v_t(x)+\partial_t)\right)=0.\eeq
Since for every $t\in I$ the map $f_t:M\to L$ is a section, we have $\frac{\partial F}{\partial t}(x, t)\in L_{f(x_t)}$, which implies $\pi_2|_{f_t(x)}\left(\frac{\partial F}{\partial t}(x, t)\right)=\frac{\partial F}{\partial t}(x, t).$ Substituting this in the above equation gives:
\begin{align}
-\frac{\partial F}{\partial t}(x, t)&=\pi_2|_{f_t(x)}\left(D_{(x,t)}F(v_t(x))\right)\\
&=\pi_2|_{f_t(x)}\left(d_xf_t(v_t(x))\right)\\
&=\pi_2|_{f_t(x)}\left(d_xf_t(N(v_t(x))+d_xf_tT(v_t(x))\right)\\
&\label{eq:final}=\pi_2|_{f_t(x)}\left(d_xf_t(N(v_t(x))\right)&\textrm{(since $d_xf_t(T_x S_t)\subset T_{f_t(x)}A$)}.
\end{align}
The transversality assumption $f_t\pitchfork A$ implies that $\pi_2|_{f_t(x)}\circ d_xf_t|_{N_xS_t}$ is invertible and \eqref{eq:final} gives the statement.
\end{proof}

It will be useful for us to have a formulation of the previous result using vector bundles trivializations.
\begin{lemma}\label{lemma:Ncoord}Under the same assumptions and notations of \cref{lemma:invert}, let $z\in M$, $U$ be a neighborhood of $z$ and $\psi:L|_{U}\to U\times \R^k$ be a vector bundle trivialization. Denote by $\widetilde{F}:U\times I\to \R^k$ the map defined by
\beq\label{eq:Ftilde} \psi(F(x, t))=(x, \widetilde{F}(x,t))\quad \forall (x, t)\in U\times I.\eeq
For every $t\in I$ let also $\widetilde{f}_t:=\widetilde{F}(\cdot, t):U\to \R^k$.
Then, for every $(x,t)\in \widehat{S}\cap (U\times I)$ we have:
\beq N(v_t(x))=\left(d_x\widetilde{f}_t|_{N_xS_t}\right)^{-1}\frac{\partial \widetilde{F}}{\partial t}(x,t).\eeq  
\end{lemma}
\begin{proof}The proof follows easily by rewriting the previous proof in coordinates.
\end{proof}

\section{The optimal transport problem between algebraic hypersurfaces}\label{S:OTalgebraic}

We now introduce the 
 optimal transport problem  between two algebraic 
hypersurfaces adopting the dynamical formulation. 
This will produce a new metric on the space of homogeneous polynomials.

\begin{definition}[Admissible curves and their Energy]\label{def:admissible}Let $I=[0,1]$ and $\mu:P_{n,d}\to \mathscr{P}(\CP^n)$ be the map from \cref{thm:ext}. We define the set of \emph{admissible curves} to be
\beq 
\Omega:=\{\gamma:I\to P_{n,d}\,|\, \mu\circ \gamma\in \mathrm{AC}^2(I, \mathscr{P}_2(\CP^n))\}.
\eeq 
If $\gamma\in \Omega$ is an admissible curve, we write $t\mapsto \mu_t:=\mu(\gamma(t))$ and define the \emph{Energy} of $\gamma$ by
\beq \en(\gamma):=\int_{I}|\dot\mu_t|^2\dd t.\eeq
\end{definition}

It is worth stressing again that the metric speed 
$|\dot\mu_t|$, defined in \cref{S:OT}, is calculated with respect to the $W_2$-distance
as denoted by $\mathscr{P}_2(\CP^n)$.

\begin{definition}
[The inner Wasserstein distance on $P_{n,d}$]
\label{D:W2in}
We define the inner Wasserstein
distance 
$W_2^{\textrm{in}}$ on $P_{n,d}$ by
\beq 
W_2^{\textrm{in}}(p_0, p_1):=\mathrm{inf}\{\en(\gamma)^{1/2}\,|\, \gamma\in \Omega, \gamma(0)=p_0, \gamma(1)=p_1\}.
\eeq 
\end{definition}
We will prove below (\cref{thm:finite}) that the function $W_2^{\textrm{in}}$ is finite, i.e. that any two points in $P_{n,d}$ can be joined by an admissible curve, and that it is a distance. In order to do this we will first need to investigate the behaviour of the Energy on curves in $P_{n,d}\setminus \Delta_{n,d}$, as done in the next section.

Similarly to \cref{D:W2in}, for any other $q\geq  1$, over $P_{n,d}$ one can consider also the inner $W_q$ distance defined by  
$$
W_q^{\textrm{in}}(p_0,p_1)^q
: = \inf \int_I |\dot \mu_t|_q^q \, 
\dd t,
$$
where the infimum runs 
over the set of curves 
$\gamma \in AC^q(I;\mathscr{P}_q(\CP^n))$ 
such that $\gamma(0)=p_0$ 
and $\gamma(1)=p_1$.
Recall that if no exponent is specified, the metric
speed has to be considered in 
$\mathscr{P}_2(\CP^n)$.

\subsection{The Wasserstein--Hermitian structure on the set of regular hypersurfaces}\label{sec:whs}

\begin{theorem}\label{thm:metricspeedsmooth}Let $c:I\to \C[z_0, \ldots, z_n]_{(d)}\setminus D_{n,d}$ be a $\mathscr{C}^1$ curve. For every $t\in I$ denote by $\mu_t=\mu([c_t])\in \mathscr{P}(\CP^n)$ the probability measure defined by \eqref{eq:tbe} and by 
\beq \widehat{Z}(c_t):=\{z\in S^{2n+1}\,|\, c_t(z)=0\}\eeq
the zero set of $c_t$ on the sphere (this is a smooth manifold by assumption). Then
\beq\label{eq:speed} |\dot\mu_t|^2=\frac{1}{d\cdot \mathrm{vol}(S^{2n-1})}\int_{\widehat{Z}(c_t)}\frac{|\dot{c}_t(b)|^2}{\|\nabla^{\C}c_t(b)\|^2}\mathrm{vol}_{\widehat{Z}(c_t)}(\dd b),\eeq
where $\nabla^\C p(b):=(\frac{\partial p}{\partial z_0}(b), \ldots, \frac{\partial p}{\partial z_n}(b))$ denotes the ``complex gradient'' of a polynomial $p$.\end{theorem}

\begin{proof}For this proof we make extensive use of the notation from \cref{Appendix A}.

\smallskip
{\bf Step 1.}
Let us observe first that $\{Z_t:=Z(c_t)\}_{t\in I}$ is a smooth family of submanifolds in $\CP^n$, in the sense of \cref{def:smoothfamily}. To this end, denote by $\sigma_t:=\sigma_{c_t}\in \Gamma(L_{n,d}, \CP^n)$ the section defined by the polynomial $c_t$ (see \cref{sec:homosec}). We define the map $F:\CP^n\times I\to L_{n,d}$ by
\beq F([z], t):=\sigma_t([z]).\eeq
The condition $c_t\notin D_{n,d}$ implies that for every $t\in I$ we have $\sigma_t\pitchfork A$ and $Z_t=\sigma_t^{-1}(A)$ is a smooth submanifold. Moreover, for every $t\in I$ the manifold $Z_t$ is a smooth complex hypersurface, hence K\"ahler. We are therefore in the setting of \cref{propo:ka}, which gives:
\beq\label{eq:ms} |\dot \mu_t|^2=\int_{Z_t}\|N(v_t)\|^2\dd \mu_t=\frac{1}{\mathrm{vol}(Z_t)}\int_{Z_t}\|N(v_t)\|^2 \omega_{\mathrm{FS}}^{n-1},\eeq
where $v_t$ is the time--dependent vector field whose flow generates the Hamiltonian isotopy given by \cref{coro:hamilthom} and $N(v_t)$ is its orthogonal projection on $NZ_t$. 

\smallskip
{\bf Step 2.}
Without loss of generality, we prove \eqref{eq:speed} for $t=0$. We denote by $p_0:=c_0
$ and $p_1:=\dot c_0$ so that $c_t=p_0+tp_1+o(t)$. Since the map associating to a polynomial $p$ the corresponding section $ \sigma_{p}$ of $L_{n,d}$ is linear (see \cref{sec:homosec}) we have $\sigma_t=\sigma_{0}+t\sigma_1+o(t),$ where $\sigma_0:=\sigma_{p_0}$ and $\sigma_1:=\sigma_{p_1}$. 

Let $[b]\in Z_0$ with $b\in S^{2n+1}$. In order to compute $N(v_0([b]))$ we use  \cref{lemma:Ncoord}. Denote by $\psi_b:L|_{U_b}\to U_b\times \C$ the trivialization given by \eqref{eq:trivib}. Then, by \eqref{eq:sectionincoord},
\beq \widetilde{F}([z], t)=\left([z], \frac{\sigma_0([z])+t\sigma_1([z])+o(t)}{\sigma_b([z])}\right)=\left([z], \frac{p_0(z)+tp_1(z)+o(t)}{\langle z, b\rangle^d}\right),\eeq
where $\sigma_b$ is the section defined in \eqref{eq:secb}. By \cref{lemma:Ncoord}, we have
\beq N(v_0([b]))=-\left(D_{[b]}\widetilde{f}_0|_{N_{[b]}Z_0}\right)^{-1}\frac{\partial \widetilde{F}}{\partial t}([b], 0), \eeq
where, according to the definition \eqref{eq:Ftilde}, 
\beq\widetilde{f}_t([z])=\frac{p_0(z)+tp_1(z)+o(t)}{\langle z, b\rangle^d}.\eeq
In order to perform the computation of the norm of $N(v_0([b]))$, we use the parametrization $\phi_b:\C^n\to U_b$ introduced in \eqref{eq:parab}. In fact, by \eqref{eq:isob}, we have
\beq\label{eq:re} \|N(v_0([b]))\|_{(T_{[b]}\CP^n, \,h_{\mathrm{FS}, [b]})}=\|D_0\phi^{-1}_bN(v_0([b]))\|_{(\C^n, \,h_{\mathrm{std}})}.\eeq
Set now $\widetilde{g}_t(v):=\widetilde{f}_t(\phi_b(v))$ so that 
\beq\label{eq:expression} \widetilde{G}(v, t):=\widetilde{F}(\phi_b(v), t)=\widetilde{g}_t(v)=\frac{p_0\left(R\left(\begin{matrix}1\\v\end{matrix}\right)\right)+tp_1\left(R\left(\begin{matrix}1\\v\end{matrix}\right)\right)+o(t)}{\left\langle R\left(\begin{matrix}1\\v\end{matrix}\right), b\right\rangle^d},\eeq
where $R$ is a unitary matrix such that $Re_0=b$. Denote also by
\beq V:=D_0\phi_b^{-1}\left(N_{[b]}S_0\right))=\left(T_0\widetilde{g}^{-1}(0)\right)^{\perp_{\mathrm{std}}}.\eeq
With these notations, by \eqref{eq:re}, we have
\beq \|N(v_0([b])\|_{(T_{[b]}\CP^n, \,h_{\mathrm{FS}, [b]})}=\left\|\left(D_0\widetilde{g}_0|_{V}\right)^{-1}\frac{\partial \widetilde G}{\partial t}(0,0)\right\|_{(\C^n, \,h_{\mathrm{std}})}.\eeq
To start with
\beq \frac{\partial \widetilde G}{\partial t}(0,0)=\frac{p_1\left(R\left(\begin{matrix}1\\0\end{matrix}\right)\right)}{\left\langle R\left(\begin{matrix}1\\0\end{matrix}\right), b\right\rangle^d}=\frac{p_1(b)}{\|b\|^2}=p_1(b),\eeq
since $Re_0=b$ and $b\in S^{2n+1}$.

Notice now that, the complex--linear operator $D_0\widetilde{g}_0:\C^n\to \C$ acts as:
\beq w=(w_1, \ldots, w_n)\mapsto \sum_{j=1}^n\frac{\partial \widetilde{g}_0}{\partial v_j}(0)w_j.\eeq
Therefore, defining
\beq \nabla^\C\widetilde{g}_0(0):=\left(\frac{\partial \widetilde{g}_0}{\partial v_1}(0), \ldots, \frac{\partial \widetilde{g}_0}{\partial v_n}(0)\right),\eeq
using the standard Hermitian structure we can write
\beq D_0\widetilde{g}_0:w\mapsto\langle w, \overline{\nabla^\C\widetilde{g}_0(0)}\rangle_{\C^n}.\eeq Consequently, since $T_0\widetilde{g}_0^{-1}(0)=\mathrm{ker}(D_0\widetilde{g}_0)=\overline{\nabla^\C\widetilde{g}_0(0)}^{\perp}$, we see that $V=\left(T_0\widetilde{g}^{-1}(0)\right)^{\perp}$ can be identified with 
\beq V=\mathrm{span}_\C\{\overline{\nabla^\C\widetilde{g}_0(0)}\}.\eeq
In particular, $D_0\widetilde{g}_0|_{V}$ can be described as the linear map
\beq D_0\widetilde{g}_0|_{V}: \lambda\cdot \overline{\nabla^\C\widetilde{g}_0(0)}\mapsto\langle\lambda\cdot \overline{\nabla^\C\widetilde{g}_0(0)}, \overline{\nabla^\C\widetilde{g}_0(0)}\rangle_{\C^n}=\lambda \|\nabla^\C\widetilde{g}_0(0)\|^2. \eeq
This implies that
\beq \left(D_0\widetilde{g}_0|_{V}\right)^{-1}\frac{\partial \widetilde G}{\partial t}(0,0)=\frac{p_1(b)}{\|\nabla^\C\widetilde{g}_0(0)\|^2}\overline{\nabla^\C\widetilde{g}_0(0)},\eeq
and, consequently:
\beq\label{eq:interN} \|N(v_0([b]))\|^2=\frac{|p_1(b)|^2}{\|\nabla^\C\widetilde{g}_0(0)\|^2}.\eeq
Given $w\in \C^n$, using \eqref{eq:expression}, for every $w\in\C^{n}$ we have
\begin{align}\label{eq:step}D_0\widetilde{g}_0w&=\langle w, \overline{\nabla^\C\widetilde{g}_0(0)}\rangle_{\C^n}\\
&=D_0p_0R\left(\begin{smallmatrix}0\\w\end{smallmatrix}\right)\\
&=\langle R\left(\begin{smallmatrix}0\\w\end{smallmatrix}\right), \overline{\nabla^{\C}p_0(b)}\rangle_{\C^{n+1}}=(*).\end{align}
Observe now that, since $Re_0=b$ and $R$ is a unitary matrix, $R\left(\begin{smallmatrix}0\\w\end{smallmatrix}\right)\in b^\perp$ for every $w\in \C^n$ and, consequently,
\beq(*)=\langle R\left(\begin{smallmatrix}0\\w\end{smallmatrix}\right), \mathrm{proj}_{b^\perp}\overline{\nabla^{\C}p_0(b)}\rangle_{\C^{n+1}}
=\langle \left(\begin{smallmatrix}0\\w\end{smallmatrix}\right), R^*\mathrm{proj}_{b^\perp}\overline{\nabla^{\C}p_0(b)}\rangle_{\C^{n+1}}.\eeq
Comparing with \eqref{eq:step}, and using the fact that $R^*$ is a unitary matrix, we have:
\beq \|\nabla^\C\widetilde{g}_0(0)\|^2=\|\mathrm{proj}_{b^\perp}\overline{\nabla^{\C}p_0(b)}\|^2.\eeq
Observe now that, for a homogeneous polynomial $p$, Euler's identity $\sum_j z_j\partial_{j}p(z)=\mathrm{deg}(p)p(z)$ can be written as
\beq \mathrm{deg}(p)p(z)=\langle z, \overline{\nabla^\C p(z)}\rangle.\eeq In particular, if $z=b$ is a zero of $p$, we have $\langle b, \overline{\nabla^\C p(b)}\rangle=0$. Consequently
\begin{align} \overline{\nabla^\C p(b)}&=\langle \overline{\nabla^\C p(b)}, b\rangle \cdot b+\mathrm{proj}_{b^\perp}\overline{\nabla^\C p(b)}\\
&=\overline{\langle b, \overline{\nabla^\C p(b)}\rangle}\cdot b+\mathrm{proj}_{b^\perp}\overline{\nabla^\C p(b)}\\
&=\mathrm{proj}_{b^\perp}\overline{\nabla^\C p(b)}.
\end{align} 
Since $\|\overline{\nabla^\C p(b)}\|=\|\nabla^\C p(b)\|,$ together with \eqref{eq:interN}, the expression \eqref{eq:ms} gives:
\beq |\dot \mu_t|^2=\frac{1}{\mathrm{vol}(Z_t)}\int_{Z_t}\frac{|\dot{c}_t(b)|^2}{\|\nabla^{\C}c_t(b)\|^2}\mathrm{vol}_{Z_t}(\dd [b]),\eeq
where in the integrand, which is a function of $[b]$, we have chosen representatives $b\in S^{n+1}$. Using the fact that the quotient map $S^{2n+1}\to \CP^n$ is a Riemannian submersion with fibers which are isometric to circles and the fact that $\mathrm{vol}(Z_t)=d\cdot \mathrm{vol}(\CP^{n-1})$ (\cref{lemma:dis}), one can immediately rewrite the last integral as in \eqref{eq:speed}. This concludes the proof.
\end{proof}
Let now $p\in \C[z_0, \ldots, z_n]_{(d)}\setminus D_{n,d}$. Motivated by \eqref{eq:speed}, we define a Hermitian form $\widehat{h}_p$ on  $T_p\C[z_0, \ldots, z_n]_{(d)}\simeq\C[z_0, \ldots, z_n]_{(d)} $ by setting:
\beq\label{eq:inte}\widehat{h}_{p}(q_1, q_2):=\int_{\widehat{Z}(p)}\frac{q_1(b)\overline{q_2(b)}}{\|\nabla^{\C}p(b)\|^2}\mathrm{vol}_{\widehat{Z}(p)}(\dd b).
\eeq 
This Hermitian form \emph{is not} positive definite, since $\widehat{h}_{p}(p,p)=0$. In fact we have the following.

\begin{lemma}\label{lemma:ker}For every $p\in \C[z_0, \ldots, z_n]_{(d)}\setminus D_{n,d}$ we have $\mathrm{ker}(\widehat{h}_p)=\mathrm{span}_{\C}\{p\}.$
\end{lemma}
\begin{proof}Let $q\in \mathrm{ker}(\widehat{h}_p)$. Then for all $u\in \C[z_0, \ldots, z_n]_{(d)}$ we have $\widehat{h}_{p}(q, u)=0$. In particular this holds true for $u=q$. Observe that, since $p\notin D_{n,d}$, the function $b\mapsto  \|\nabla^{\C}p(b)\|^2$ is never zero on $\widehat{Z}(p)$, otherwise $p$ would have a singular projective zero. In order for the integral in \eqref{eq:inte} to vanish, we must therefore have $|q|^2|_{\widehat{Z}(p)}\equiv 0.$ This means that $Z(p)\subseteq Z(q)$. Writing $q=q_1\cdots q_\ell$ as a product of irreducible factors, we must have $Z(p)=Z(q_j)$ for some $q_j$ and, since $p$ is irreducible (because $Z(p)$ is smooth), $q_j=c_jp $ for some $c_j\in \C$ nonzero. But then, since $\deg(q)=\deg(p)$ all the other irreducible factors of $q$ must have zero degree, which implies that $q$ is a multiple of $p$.
\end{proof}
We define now a  Hermitian structure on $P_{n,d}\setminus \Delta_{n,d}$, as follows. Denote by 
\beq \pi:\C[z_0, \ldots, z_n]_{(d)}\setminus \{0\}\to P_{n,d}\eeq 
the quotient map. 
Given $[p]\in P_{n,d}\setminus \Delta_{n,d}$ and $v_1, v_2\in T_{[p]}(P_{n,d})$, let $p\in \C[z_0, \ldots, z_n]_{(d)}$ be a representative and $q_1, q_2\in T_{p}(\C[z_0, \ldots, z_n]_{(d)})$ such that $D_p\pi q_i=v_i.$ 
We define
\beq 
\label{eq:hW}(\h)_{[p]}(v_1,
v_2):=\widehat{h}_p(q_1, q_2).
\eeq
This definition does not depend on the choice of the representatives $p$ and $q_1, q_2$. In fact, let $\lambda p\in \C[z_0, \ldots, z_n]_{(d)}$ be another representative for $[p]$, where $\lambda \in \C\setminus\{0\}$. Then $\lambda q_1$ and $\lambda q_2$ are two vectors in $T_{\lambda p}(\C[z_0, \ldots, z_n]_{(d)})$ such that $D_{\lambda p}\pi\lambda q_i=v_i$ and
\begin{align} \widehat{h}_{\lambda p}(\lambda q_1, \lambda q_2)&=\int_{\widehat{Z}(\lambda p)}\frac{\lambda q_1 \overline{\lambda q_2}}{\|\nabla^\C (\lambda p)\|^2}\mathrm{vol}_{\widehat{Z}(\lambda p)}\\
&=\int_{\widehat{Z}( p)}\frac{|\lambda|^2 q_1 \overline{ q_2}}{\|\lambda \nabla^\C p\|^2}\mathrm{vol}_{\widehat{Z}( p)}\\
&=\widehat{h}_p(q_1, q_2).\end{align}
\cref{lemma:ker} implies now that the Hermitian form $\h$ is positive definite.

\begin{definition}[The Wasserstein--Hermitian structure on the set of nonsingular hypersurfaces]\label{def:WHS} We call the form $\h$ defined in \eqref{eq:hW} the \emph{Wasserstein--Hermitian structure} on $P_{n,d}\setminus \Delta_{n,d}$.
\end{definition}

We notice that the integrand in \eqref{eq:inte} is a well defined function on $Z(p)$, therefore, using the fact that the Riemannian submersion $S^{2n+1}\to \CP^n$ restricts to a Riemannian submersion $\widehat{Z}(p)\to Z(p)$ (always with fibers that are circles), we can rewrite it as:
\begin{align}\label{eq:inte2}\widehat{h}_{p}(q_1, q_2)&=\int_{\widehat{Z}(p)}\frac{q_1(b)\overline{q_2(b)}}{\|\nabla^{\C}p(b)\|^2}\mathrm{vol}_{\widehat{Z}(p)}(\dd b)\\
&=2\pi\int_{Z(p)}\frac{q_1(b)\overline{q_2(b)}}{\|\nabla^{\C}p(b)\|^2}\mathrm{vol}_{Z(p)}(\dd b)\\
\label{lasteq}&=\frac{d\cdot \mathrm{vol}(S^{2n-1})}{d\cdot \mathrm{vol}(\CP^{n-1})}\int_{Z(p)}\frac{q_1(b)\overline{q_2(b)}}{\|\nabla^{\C}p(b)\|^2}\mathrm{vol}_{Z(p)}(\dd b).
\end{align}
The last expression reconciles with the notation adopted in \cref{sec:intro}, but the expression in \eqref{eq:inte} is more practical to work with. 

\begin{proposition}\label{thm:smooth}
    The Hermitian form $\h$ on $P_{n,d}\setminus \Delta_{n,d}$ is smooth.
\end{proposition}
\begin{proof}Put $\Omega=\mathbb{C}[z_0,\dots,z_n]_{(d)}\setminus D_{n,d}$. Then 
    it is clear that it is sufficient to show that the form 
    $$\widehat{h}:\Omega \times \mathbb{C}[z_0,\dots z_n]_{(d)}^2 \longrightarrow \mathbb{C}, \quad (p,q_1,q_2) \mapsto \widehat{h}_p(q_1,q_2),$$ is smooth with respect to all its variables. Consider the solution manifold on the sphere:
    $$\widehat{V}=\Big{\{}(p,x)\in \Omega \times S^{2n+1}: x \in \widehat{Z}(p) \Big{\}}.$$ It is a smooth manifold equipped with the projection $\pi_1:\widehat{V}\to \Omega$ on the first factor. Now $\pi_1$ is proper (smooth with compact fibers) and by the Ehresmann Fibration Theorem is a locally trivial fibration. In other words, around any point  $p_0\in \Omega$ we can find a small open set $U$ and a diffeomorphism $\Phi: \pi_1^{-1}(U)\longrightarrow U \times \widehat{Z}(p_0)$ mapping the fibers of $\pi_1$ to the fibers of the projection on $U$.  In particular the integration domains vary smoothly together with the 
    volume forms $p\mapsto \textrm{vol}_{\widehat{Z}(p)}$. Indeed, in $U$, the maps  $\Psi_p:=\Phi^{-1}(p,\cdot):\widehat{Z}(p_0) \longrightarrow \widehat{Z}(p)$ give a smooth family of diffeomorphisms. Defining $\sigma_p:=\Psi_p^*\operatorname{vol}_{\widehat{Z}(p)}$ 
    we get a smooth family of volume forms on $\widehat{Z}(p_0)$
    such that
$$\widehat{h}_p(q_1,q_2)=\int_{\widehat{Z}(p_0)} \dfrac{q_1 \overline{q}_2}{\|\nabla^{\mathbb{C}}p \|^2 \circ \Psi_p} \sigma_p.$$
This implies the required smoothness.
\end{proof}

\begin{corollary}\label{coro:energy}Let $\gamma:I\to P_{n,d}\setminus \Delta_{n,d}$ be a $\mathscr{C}^1$ curve. Then $\gamma$ is admissible in the sense of \cref{def:admissible} and, denoting by $t\mapsto \mu_t:=\mu(\gamma(t))$, we have
\beq \en(\gamma) =\frac{1}{d\cdot \mathrm{vol}(\CP^{n-1})}\int_{0}^1(\h)_{\gamma(t)}(\dot \gamma(t), \dot \gamma(t))\dd t.\eeq
\end{corollary}

\begin{proof}Let $c:I\to \C[z_0, \ldots, z_n]_{(d)}$ be a smooth curve lifting $\gamma$ (i.e. such that $[c(t)]=\gamma(t)$). Then, by \cref{thm:metricspeedsmooth}, we have
\begin{align}\en(\gamma)&=\int_{0}^1|\dot\mu_t|^2\dd t\\
&=\frac{1}{d\cdot \mathrm{vol}(S^{2n-1})}\int_{0}^{1}\widehat{h}_{c(t)}(\dot c(t), \dot c(t))\dd t\\
&=\frac{1}{d\cdot \mathrm{vol}(S^{2n-1})}\int_{0}^1 \frac{d\cdot \mathrm{vol}(S^{2n-1})}{d\cdot \mathrm{vol}(\CP^{n-1})}(\h)_{\gamma(t)}(\dot\gamma(t), \dot \gamma(t))\dd t\quad \textrm{(by \eqref{lasteq})},\\
&=\frac{1}{d\cdot \mathrm{vol}(\CP^{n-1})}\int_{0}^1 (\h)_{\gamma(t)}(\dot\gamma(t), \dot \gamma(t))
\end{align}
where the last identity follows from the definition of the Wasserstein--Hermitian structure and the fact that $D_{c(t)}\pi\dot c(t)=\dot \gamma(t)$ for every $t\in I$. 
\end{proof}

\begin{remark}In the case $d=1$ the discriminant $\Delta_{n,1}$ is empty and the Wasserstein metric coincides with the Fubini--Study metric (up to multiples). To see this, we identify $\CP^n\simeq (\CP^n)^*$ and observe that the Fubini--Study metric is $U(n+1)$-invariant; moreover, since the action of $U(n)=\mathrm{Stab}_{[e_0]}$ on $T_{[e_0]}\CP^n\simeq U(n)$ is the standard action, this action is irreducible. In particular, by Schur's Lemma, the Fubini--Study metric is the unique $U(n+1)$--invariant Riemannian metric on $\CP^n$ (up to multiples). Since our Wasserstein--Hermitian structure is $U(n+1)$--invariant, it must therefore be a multiple of the Fubini--Study one.
\end{remark}

\subsection{The Wasserstein--K\"ahler structure}\label{sec:WK}In this section we prove that the imaginary part of the Wasserstein--Hermitian structure $\h$ from  \cref{def:WHS} is a closed $2$--form, and therefore  $(P_{n,d}\setminus\Delta_{n,d}, \h)$ is a K\"ahler manifold.

Before proceeding, let us recall the notion of \emph{fiber integral} of a differential form. Let $\pi:W\to U$ be a proper submersion with $f$--dimensional fibers and $\alpha\in \Omega^a(W)$ be a differential form. The fiber integral of $\alpha$, denoted by $\pi_*\alpha\in \Omega^{a-f}(U)$, is the differential form defined as follows.
Given $u\in U$ and $v_1, \ldots, v_{a-f}\in T_uU$, we first choose vector fields $\tilde{v}_1, \ldots, \tilde{v}_{a-f}$ on $W$ along $\pi^{-1}(u)$ such that $D\tilde{v}_j=v_j$ for every $j=1, \ldots, a-f.$  (These vector fields exist, since $\pi$ is a submersion; the construction will not depend on this choice.) Then we consider the interior product $\iota_{\tilde{v}_1}\cdots \iota_{\tilde{v}_{a-f}}\alpha$, which is a top form on $\pi^{-1}(u)$, and we set
$$(\pi_*\alpha)_u(v_1, \ldots, v_{a-f}):=\int_{\pi^{-1}(u)} \iota_{\tilde{v}_1}\cdots \iota_{\tilde{v}_{a-f}}\alpha.$$
This defines a smooth differential form on $U$.

Let now $V=\{(p, z)\,|\, p(z)=0\}\subset P_{n,d}\times \CP^n$ be the solution variety, together with the restrictions 
$$\pi_1:V\to P_{n,d}\quad \textrm{and}\quad \pi_{2}:V\to \CP^n$$
of the projections on the two factors. Set $W:=V\setminus \pi_{1}^{-1}(\Delta_{n,d})$ and $U=P_{n,d}\setminus \Delta_{n,d}.$ We will apply the previous construction to  $\pi_2^*\omega_{\mathrm{FS}}^n/n!\in \Omega^{n,n}(W)\subset \Omega^{2n}(W)$ (note that this is the pull--back of the volume form of $\CP^n$).

\begin{theorem}[The Wasserstein--K\"ahler structure]\label{thm:WK}Let $\h$ be the Wasserstein--Hermitian structure and let $\sigma:=\mathrm{Im}(\h)\in \Omega^{2}(P_{n,d}\setminus \Delta_{n,d}).$
Then
\begin{equation}\label{eq:sigma}\sigma=(\pi_1)_*\pi_2^*\left(\frac{\omega_{FS}^n}{n!}\right)\in \Omega^{1,1}(P_{n,d}\setminus \Delta_{n,d}).\end{equation}
In particular $(P_{n,d}\setminus\Delta_{n,d}, \h)$ is a K\"ahler manifold.
\end{theorem}

\begin{proof}Let $p\in P_{n,d}\setminus\Delta_{n,d}$ and $q_1, q_2\in T_pP_{n,d}.$ In order to compute $\sigma_p(q_1, q_2)$, we need to lift $q_1, q_2$ to vector fields on $V$ along $\pi_1^{-1}(p)\simeq Z(p)$. Working in homogenous coordinates, since $TV=\{(\dot p, \dot z)\,|\, \dot p(z)+D_zp\dot z=0\}$, it is immediate to see that the vector fields defined by
$$\tilde{q}_i(p, z):=\left(q_i, -q_i(z)\xi(z)\right), \quad \textrm{where}\quad \xi(z):=\frac{\overline{\nabla^\C p(z)}}{\|\nabla^\C p(z)\|},$$
are two such lifts.

Now, recall the following basic fact from differential geometry. Given a $2$--form $\beta$ and two vector fields $X, Y$, we have
\beq\label{eq:inner}\iota_X\iota_Y \beta^n=n\beta(X, Y)\beta^{n-1}+n(n-1)\iota_X\wedge\iota_Y\wedge \beta^{n-2}.\eeq
In the case $\beta=\pi_2^*\omega_{\mathrm{FS}}, $ we have 
$$\iota_{\tilde{q}_i}\beta=\omega_{\mathrm{FS}}(-q_i(z)\xi(z), \cdot).$$
Since $\xi(z)\perp T_zZ(p),$ it follows that $\omega_{\mathrm{FS}}(-q_i(z)\xi(z), \tau)=\mathrm{Im}(h_{\mathrm{FS}}(-q_i(z)\xi(z), \tau))=0$ for every $\tau\in T_{z}Z(p)$, i.e.
\beq\label{eq:vanish}\iota_{\tilde{q}_i}\beta|_{\pi^{-1}(p)}=0.\eeq
In particular, using \eqref{eq:inner} and \eqref{eq:vanish},
\begin{align}\sigma_{p}(q_1, q_2)&=\int_{\pi_2^{-1}(p)}\frac{1}{n!}\iota_{\tilde q_1}\iota_{\tilde q_2}\beta^n\\
&=\frac{n}{n!}\int_{Z(p)}\omega_{\mathrm{FS}}(-q_1(z)\xi(z), -q_2(z)\xi(z))\omega_{\mathrm{FS}}^{n-1}\\
&=\int_{Z(p)}\frac{\mathrm{Im}(q_1(z)\overline{q_2(z)})}{\|\nabla^\C p(z)\|^2}\frac{\omega_{\mathrm{FS}}^{n-1}}{(n-1)!}\\
&=\int_{Z(p)}\frac{\mathrm{Im}(q_1(z)\overline{q_2(z)})}{\|\nabla^\C p(z)\|^2}\mathrm{vol}_{Z(p)}(\dd z)\\
&=\mathrm{Im}(\h_p(q_1, q_2)),
\end{align}
which proves \eqref{eq:sigma}. The fact that $\sigma$ is closed follows now from the fact that integration along the fibers commutes with exterior derivative.
\end{proof}


\subsection{Properties of $W_2^{\textrm{in}}$}We are now in the position of proving that for any two points $p_{0},p_{1} \in P_{n,d}$, 
the function $W_2^{\textrm{in}}(p_{0},p_{1})$ is finite and it defines a distance. We prove first some auxiliary results, using tools from semialgebraic geometry (see \cref{def:semialg}).

\begin{lemma}\label{lemma:reparametrization}Let $f:[0, r]\to \R^\ell$ be a continuous semialgebraic curve. Then there exists $m\in \mathbb{N}$ and $0<r_0\leq r$ such that the curve $\widetilde{f}(s):=f(s^m)$ is $\mathscr{C}^1$ on $[0, r_0]$.
\end{lemma}
\begin{proof}Denote by $(f_1, \ldots, f_\ell)$ the components of $f$. Since each of them is continuous and semialgebraic, for every $i$ there exists $q_i\in \mathbb{N}$, $r_i\leq r$ and a $\mathscr{C}^1$ function $g_i:[0, r_i]\to \R$ such that:
\beq f_i(t)=f_i(0)+t^{\frac{1}{q_i}}g_i(t)\quad \textrm{for all}\quad 0\leq t\leq r_i.\eeq
Letting $m=\max_i q_i$ and $r_0=\min_i r_i$ we see that
\beq f_i(s^m)=f_i(0)+t^{\frac{m}{q_i}}g_i(s^m)\quad \textrm{forall}\quad 0\leq t\leq r_0.\eeq
Since $\frac{m}{q_i}\geq 1$ and $g_i$ is $\mathscr{C}^1$, this implies that $\widetilde{f}$ is $\mathscr{C}^1$ on $[0, r_0].$
\end{proof}

\begin{lemma}\label{lemma:selection}Let $\gamma:[0,1]\to P_{n,d}$ be a continuous semialgebraic curve such that $\gamma(0)\in \Delta_{n,d}$ and $\gamma(t)\in P_{n,d}\setminus \Delta_{n,d}$ for all $t\in (0, 1]$. Then there exist $0<r\leq 1$ and a continuous semialgebraic curve $b:[0, r]\to \CP^n$ such that $b|_{(0, r]}$ is $\mathscr{C}^1$ and \beq \label{eq:argmax}b(t)\in \mathrm{argmax}_{b\in Z(\gamma_t)} \frac{|\dot{\gamma}_t(b)|^2}{\|\nabla^{\C}\gamma_t(b)\|^2}.\eeq
Moreover, for every $0<t\leq r$ we have:
\beq \label{eq:ineqarg} \frac{|\dot{\gamma}_t(b(t))|^2}{\|\nabla^{\C}\gamma_t(b(t))\|^2}\leq |\dot b(t)|^2.\eeq
\end{lemma}
\begin{proof}Notice first that the function 
\beq t\mapsto \max_{\theta\in Z(\gamma_t)}\frac{|\dot{\gamma}_t(\theta)|^2}{\|\nabla^{\C}\gamma_t(\theta)\|^2}\eeq
is well defined on $(0,1)$, since $\gamma(0,1]\notin \Delta_{n,d}$, and semialgebraic (its graph is semialgebraic). In particular the following set $S\subset (0,1)\times \CP^n$ is semialgebraic:
\beq S:=\left\{(t, b)\,\bigg|\, b\in Z(\gamma_t)\quad \textrm{and}\quad \frac{|\dot{\gamma}_t(b)|^2}{\|\nabla^{\C}\gamma_t(b)\|^2}=\max_{\theta\in Z(\gamma_t)}\frac{|\dot{\gamma}_t(\theta)|^2}{\|\nabla^{\C}\gamma_t(\theta)\|^2}\right\}.\eeq
Let now $\pi:S\to (0,1)$ be the restriction of the projection on the first factor. By the Semialgebraic Triviality Theorem \cite[Theorem 9.3.2]{BCR} there exist $0<\delta<1$, a semialgebraic set $F\subset \CP^n$ and a semialgebraic homeomorphism $\psi:\pi^{-1}((0, \delta))\to (0, \delta)\times F$ such that, denoting by $p_1:(0, \delta)\times F\to (0, \delta)$ the projection on the first factor, we have $\pi=p_1\circ \psi.$

Pick $f_0\in F$ and define the curve $b:[0, \delta)\to \CP^n$ by:
\beq b(t):=\begin{cases} \psi^{-1}(t, f_0)&t\in (0, \delta)\\
\lim_{t\to 0}\psi^{-1}(t, f_0)&t=0.
\end{cases}
\eeq
The limit $\lim_{t\to 0}\psi^{-1}(t, f_0)$ exists since $\psi$ is semialgebraic with codomain $\CP^n$ (which is compact).
Such curve $b:[0, \delta]\to \CP^n$ is continuous and is semialgebraic, therefore there exists $0<r\leq \delta$ such that $b|_{(0, r)}$ is $\mathscr{C}^1$. Since $(t,b(t))\in S$, then $b(t)$ verifies \eqref{eq:argmax} for all $0<t\leq r$, this proves the first part of the statement.

For the second part of the statement, we use the fact that $b(t)\in Z(\gamma_t)$ for all $0<t\leq r$, hence $\gamma_t(b(t))\equiv 0$. Differentiating this last equation (viewed in $S^{2n+1}\subset \C^{n+1}$) with respect to $t$ gives $\dot{\gamma}_t(b(t))+\langle \dot b(t), (\nabla^{\C}\gamma_t)(b(t))\rangle\equiv 0$ which implies
\beq |\dot{\gamma}_t(b(t))|\leq \|\dot b(t)\|\cdot\| \nabla^{\C}\gamma_t(b(t))\|.\eeq
The last inequality implies \eqref{eq:ineqarg}, and this concludes the proof.
\end{proof}

\begin{theorem}\label{thm:rep}Let $\gamma:[0,1]\to P_{n,d}$ be a continuous semialgebraic curve such that $\gamma(0)\in \Delta_{n,d}$ and $\gamma(t)\notin \Delta_{n,d}$ for all $t\in (0, 1]$. Then there exists $m\in \mathbb{N}$ such that the curve $\alpha:[0,1]\to P_{n,d}$, defined by $\alpha(s):=\gamma(s^m)$, has finite energy, i.e. $\en(\alpha)<\infty$.
\end{theorem}
\begin{proof}Let $b:[0, r]\to \CP^n$ be the curve given by \cref{lemma:selection}. By \cite[Proposition 2.4.6]{BCR}, there exists a smooth, real algebraic embedding $\varphi:\CP^n\to \R^\ell$, for some $\ell\in \mathbb{N}$ sufficiently large. Denote by $f(t):=\varphi(b(t))$. By \cref{lemma:reparametrization} there exists $0<r_0\leq r\leq 1$ and $m\in \mathbb{N}$ such that the curve $\widetilde{f}(s):=f(s^m)$ is $\mathscr{C}^1$ on $[0, r_0]$. In particular, since $\varphi$ is a semialgebraic diffeomorphism, the curve
\beq \beta(s):=b(s^m)=\varphi^{-1}(\widetilde{f}(s))\eeq
is semialgebraic and of class $\mathscr{C}^1$ on $[0, r_0]$.

Consider now the curve $\alpha(s):=\gamma(s^m)$. By \cref{coro:energy} we have 
\begin{align} \en(\alpha)&=\int_{0}^1\frac{1}{\mathrm{vol}(Z(\alpha(s)))}\int_{Z(\alpha(s))}\frac{|\dot{\alpha}_s(b)|^2}{\|\nabla^{\C}\alpha_s(b)\|^2}\mathrm{vol}_{Z(\alpha(s))}(\dd [b])\dd s\\
&=\left(\int_{0}^{r_0}+\int_{r_0}^1\right)\left(\frac{1}{\mathrm{vol}(Z(\alpha(s)))}\int_{Z(\alpha(s))}\frac{|\dot{\alpha}_s(b)|^2}{\|\nabla^{\C}\alpha_s(b)\|^2}\mathrm{vol}_{Z(\alpha(s))}(\dd [b])\right)\dd s.
\end{align}
Since for $s\geq r_0$ the set $Z(\alpha(s))=Z(\gamma(s^{m}))$ is defined by a regular equation, the integrand is bounded on $[r_0, 1]$. 
Therefore it remains to estimate the integral in $[0, r_0]$. Using \eqref{eq:ineqarg}, we have:
\begin{align}
& \int_{0}^{r_0}\frac{1}{\mathrm{vol}(Z(\alpha(s)))}\int_{Z(\alpha(s))}\frac{|ms^{m-1}\dot{\gamma}_{s^m}(b)|^2}{\|\nabla^{\C}\gamma_{s^m}(b)\|^2}\mathrm{vol}_{Z(\alpha(s))}(\dd [b])\dd s\\
\leq & \int_{0}^{r_0}\frac{1}{\mathrm{vol}(Z(\alpha(s)))}\int_{Z(\alpha(s))}|ms^{m-1}\dot b(s^m)|^2\mathrm{vol}_{Z(\alpha(s))}(\dd [b])\dd s\\
=& \int_{0}^{r_0}\frac{1}{\mathrm{vol}(Z(\alpha(s)))}\int_{Z(\alpha(s))}|\dot \beta(s)|^2\mathrm{vol}_{Z(\alpha(s))}(\dd [b])\dd s\\
=&\int_{0}^{r_0}|\dot \beta(s)|^2\dd s<\infty\quad \textrm{(since $\beta$ is $\mathscr{C}^1$ on $[0, r_0]$)}.
\end{align}
This concludes the proof.
\end{proof}

\begin{corollary}\label{thm:finite}For any two $p_0, p_1\in P_{n,d}$ there exists a curve $\gamma:I\to P_{n,d}$ such that $\gamma(0)=p_0$, $\gamma(1)=p_1$ and $\en(\gamma)<\infty.$
\end{corollary}
\begin{proof}
If $p_0, p_1\notin \Delta_{n,d}$, then any smooth curve $\gamma:I\to P_{n,d}\setminus\Delta_{n,d}$ connecting $p_0$ with $p_1$ (which exists because the real codimension of $\Delta_{n,d}$ is \emph{two}) has finite energy by \cref{coro:energy}.

Given $p_0\in \Delta_{n,d}$ and $p_1\in P_{n,d}\setminus \Delta_{n,d}$ there exists a curve $\gamma$ joining them with finite energy, by \cref{thm:rep}. The result for general $p_0, p_1$ (possibly both on $\Delta_{n,d}$) follows by concatenating two such curves.
\end{proof}

We now prove that $W_2^{\textrm{in}}$ is a metric on $P_{n,d}$.

\begin{theorem}\label{thm:propertyw}
The function $W_2^{\textrm{in}}$ defines a distance on $P_{n,d}$ and the metric space 
$(P_{n,d}, W_2^{\textrm{in}})$ is  complete and geodesic, i.e. for any couple $p_1,p_2$, there exists an admissible curve $\gamma$ from $p_1$ to 
$p_2$ such that 
$W_2^{\textrm{in}}(\gamma(s),\gamma(t)) = |t-s| W_2^{\textrm{in}}(p_1,p_2)$.
\end{theorem}

\begin{proof}
By construction $W_2^{\textrm{in}}(p_0, p_1)=W_2^{\textrm{in}}(p_1,p_0)$. Moreover,  given $p_0, p_1, p_2\in P_{n,d}$, by \cref{thm:finite} for every $\epsilon>0$ there exist admissible curves $\gamma_1, \gamma_2:I\to P_{n,d}$ such that: $\gamma_1(0)=p_0, \gamma_1(1)=p_1, \gamma_2(0)=p_1, \gamma_2(1)=p_2$ and 
\beq 
\en(\gamma_1)\leq W_2^{\textrm{in}}(p_0, p_1)^2+\epsilon\quad\textrm{and}\quad \en(\gamma_2)\leq W_2^{\textrm{in}}(p_1, p_2)^2+\epsilon.
\eeq
We use now \cite[Lemma 1.1.4]{AGS:book} to get reparametrizations $\eta_1:[0, L_1]\to P_{n,d}$ and $\eta_2:[0, L_2]\to P_{n,d}$ of these curves, with $L_1^2\leq W_2^{\textrm{in}}(p_0, p_1)^2+\epsilon$ and $L_2^2\leq W_2^{\textrm{in}}(p_1, p_2)^2+\epsilon$, 
such that, denoting by $\zeta_i:=\mu\circ\eta_i$, we have
\beq 
|\dot{\zeta_i}|\equiv 1, \quad i=1, 2.
\eeq
Define now $\eta:[0, L_1+L_2]\to P_{n,d}$ by concatenating $\eta_1$ and $\eta_2$ and $\alpha:[0,1]\to P_{n,d}$ by:
\beq \alpha(t)=\eta((L_1+L_2)t).\eeq
Then, denoting by $\beta:=\mu\circ \alpha$, we have
\beq 
W_2^{\textrm{in}}(p_0,p_2)^2 \leq \en(\alpha)=\int_{0}^1
|\dot{\beta}|^2=(L_1+L_2)^2,
\eeq
which, letting $\epsilon$ go to zero, implies the triangle inequality. Finally, since if $\gamma$ is admissible, $\mu\circ \gamma\in AC^2(I, \mathscr{P}_2(\CP^n))$, it follows that $W_2^{\textrm{in}}(p_0, p_1)=0$ implies $p_0=p_1.$ Together with \cref{thm:finite}, this proves that $W_2^{\textrm{in}}$ defines a distance on $P_{n,d}$.

In order to prove that $(P_{n,d}, W_2^{\textrm{in}})$ is geodesic, we observe that the map $\mu:P_{n,d}\to \mathscr{P}_2(\CP^n)$ is injective and continuous by \cref{thm:ext}, where $P_{n,d}$ is taken with the manifold topology. In particular $E:=\mu(P_{n,d})$ is a compact subset of $ \mathscr{P}_2(\CP^n)$, and $E$ is homeomorphic to $P_{n,d}$ itself, through $\mu$. Admissible curves on $P_{n,d}$ are precisely curves $\gamma :I\to P_{n,d}$ such that there exists an admissible curve $\beta:I\to\mathscr{P}_2(\CP^n)$, entirely contained in $E$ and such that $\gamma=\mu^{-1}\circ \beta$. The fact that $(P_{n,d}, W_2^{\textrm{in}})$ is geodesic follows now from 
\cref{thm:finite}, \cite[4.3.2]{AmbrosioTilli} and \cite[Proposition 1.2.10]{Gigli_Pasqualetto_2020}.

The completeness is a direct consequence of $W_2^{\textrm{in}}$ being lower semicontinuous with respect to the Fubini--Study metric. 
The proof of the lower semicontinuity of 
$W_2^{\textrm{in}}$ is in the subsequent \cref{lemma:lsc}. 
Consider therefore a Cauchy sequence $\{p_k \}_{k\in \N} \subset (P_{n,d},W_2^{\textrm{in}})$. By compactness of $P_{n,d}$ in the Fubini--Study metric, 
it follows that $p_k$ has a converging subsequence (still denoted by $p_k$) in the Fubini--Study metric to a certain 
$p_\infty \in P_{n,d}$.
Hence for each $j \in \N$ the the lower semicontinuity of $W_2^{\textrm{in}}$ implies that 
$$
 \liminf_{k \to \infty} 
 W_2^{\textrm{in}}( p_k,p_j)
\geq 
W_2^{\textrm{in}} (p_\infty, p_j).
$$
Letting also $j \to \infty$ it follows that $W_2^{\textrm{in}} (p_\infty, p_j) \to 0$. Since for Cauchy sequences the convergence of a subsequence is enough to deduce convergence of the whole sequence, the completeness is proved. 
\end{proof}

\begin{lemma}\label{lemma:lsc}
For every fixed $p \in P_{n,d}$, the function $W_2^{\textrm{in}} (p,\cdot) : P_{n,d} \to [0,\infty)$ is lower semicontinuous 
with respect to the 
Fubini--Study metric.
\end{lemma}

\begin{proof}
Fix $p \in P_{n,d}$ and consider a sequence $\{p_k \}_{k\in N} \subset P_{n,d}$ with 
$p_k \to p_\infty \in P_{n,d}$ in the Fubini--Study metric. 
By \cref{thm:propertyw} there exists a sequence of admissible curves $\gamma_k$ such that 
$\mu(\gamma_k) \in AC^2([0,1]; \mathscr{P}_2(\CP^n))$ and
$\en(\gamma_k) = W_2^{\textrm{in}}(p_k,p)$.
From \cref{thm:finite} we know that $W_2^{\textrm{in}}(p_\infty,p)<\infty$; hence 
if $\liminf_{k\to \infty}W_2^{\textrm{in}}(p_k,p) = \infty$ the claim is trivially verified.  
We assume therefore that $\liminf_{k\to \infty}W_2^{\textrm{in}}(p_k,p) < \infty$
and, up to relabeling, that the liminf is
indeed a limit. 
In particular, we deduce that $\en(\gamma_k) \leq C$, for some positive constant $C>0$.
Therefore 
applying
H\"older inequality
for any $s < t \in (0,1)$ 
and $k \in \N$  we obtain:
$$
W_2(\mu(\gamma_k)_t, \mu(\gamma_k)_s) \leq \int_s^t |\dot{\mu(\gamma_k)}|(\tau)
\dd \tau \leq C \sqrt{t-s}.
$$
The previous inequality implies that the family of maps $[0,1]\ni t \mapsto  \mu(\gamma_k)_t \in \mathscr{P}_2(\CP^n)$ is equicontinuous. Since 
$\mathscr{P}_2(\CP^n)$ is compact we can invoke Ascoli--Arzel\'a
Theorem to deduce the existence of a subsequence 
$\mu(\gamma_{k_{j}})$ and of a limit curve
$\eta \in 
\mathscr{C}^{0}([0,1] ; \mathscr{P}_2(\CP^n))$ 
such that $\mu(\gamma_{k_{j}}) \to \eta$ in the $\mathscr{C}^{0}$ distance. 
We can invoke then the lower semicontinuity of the energy 
(see \cite[Proposition 1.2.7]{Gigli_Pasqualetto_2020}) 
to deduce 
$$
\mathcal{E}(\eta) \leq \liminf_{k\to \infty} \mathcal{E}(\mu(\gamma_{k_{j}})).
$$
Finally, since $\mu(P_{n,d})$ is a closed subset 
of $\mathscr{P}_2(\CP^n)$, we deduce 
the existence of a curve $\gamma_\infty$ 
such that $\eta = \mu (\gamma_\infty)$. 
Then the above inequality can be rewritten as 
$$
\mathcal{E}(\gamma_\infty) \leq \liminf_{k\to \infty} \mathcal{E}(\gamma_{k_{j}}) = 
\lim_{j\to \infty}W_2^{\textrm{in}}(p_{k_j},p)
= \liminf_{k\to \infty}W_2^{\textrm{in}}(p_k,p),
$$
implying in particular that $\gamma_\infty$ is an  admissible curve. Moreover is trivially verified (recall that $\mu$ is an homeomorphism)
that $\gamma_\infty$ is joining the fixed $p \in P_{n,d}$ with $p_\infty \in P_{n,d}$.
By taking the infimum on the first term, the claim follows. 
\end{proof}

\subsection{Comparing the inner Wasserstein distance with the Fubini--Study one}As one can easily verify, the continuity of 
$$\mathrm{id}:(P_{n,d}, \sfd_{\mathrm{FS}})\to (P_{n,d}, W_2^\mathrm{in})$$
is equivalent to the compactness of $(P_{n,d}, W_2^\mathrm{in})$; notice that $P_{n,d}$ is compact for the ambient $W_2$\,--topology, by \cref{propo:compactness}. 
Compactness for the inner distance topology of a compact space in the ambient topology is in general not true, see \cite[Figure 2.2]{Burago}.  However, building on the results from the previous section, here we will prove that the above map is continuous (in fact Lipschitz) ``away'' from the discriminant, where the structure is Riemannian. More precisely, we will prove the following result. For every $\epsilon>0$ denote by
$$P_{n,d}(\epsilon):=\{p\in P_{n,d}\,|\, \sfd_{\mathrm{FS}}(p,\Delta_{n,d})\geq \epsilon\}$$

\begin{theorem}\label{thm:Lipschitz}For every $\epsilon>0$ the identity map
$$\mathrm{id}:(P_{n,d}(\epsilon), \sfd_{\mathrm{FS}})\to (P_{n,d}(\epsilon), W_2^\mathrm{in})$$
is Lipschitz. In particular $P_{n,d}(\epsilon)$ is compact in the $W_2^\mathrm{in}$--topology.
\end{theorem}
\begin{remark}The above statement is phrased using the Fubini--Study distance on $P_{n,d}\simeq \CP^{N}.$ Recall that the Bombieri--Weyl structure (see \cref{sec:BW}) is a special case of the Fubini--Study one, i.e. the case when we identify $\C^{{n+d\choose d}}$ with $\C[z_0, \ldots z_n]_{(d)}$  through the map that sends the standard basis  to the basis \eqref{eq:BWb}; therefore the reader can also think of this as a statement involving the Bombieri--Weyl distance.
\end{remark}
In order to prove \cref{thm:Lipschitz} we need some preliminary lemmas.
\begin{lemma}\label{lemma:dist1}There exist two constants $c_1, c_2>0$ such that every $p\notin \Delta_{n,d}$ has a neighborhood $U_p$ with $\mathrm{clos}(U_p)\subset P_{n,d}\setminus\Delta_{n,d}$, which is geodesically convex in the Fubini--Study metric, and such that for every $q_1, q_2\in U_p$
$$W_2^{\mathrm{in}}(q_1, q_2)\leq\bigg(c_1\sfd_{\mathrm{FS}}(p, \Delta_{n,d})^{-c_2}\bigg)\sfd_{\mathrm{FS}}(q_1, q_2).$$
\end{lemma} 
\begin{proof}Let $U_p$ be a geodesically convex neighorhood of $p$ in the Fubini--Study metric such that $\mathrm{clos}(U_p)\cap \Delta_{n,d}=\emptyset$ and such that for every $q\in U_p$ we have 
\beq\label{eq:double}\sfd_{\mathrm{FS}}(q, \Delta_{n,d})\geq\frac{1}{2}\sfd_{\mathrm{FS}}(p, \Delta_{n,d}).\eeq
Consider also the semialgebraic functions $\alpha, \beta:P_{n,d}\to \R$ defined by
$$\alpha(p):=\sfd_{\mathrm{FS}}(p, \Delta_{n,d})\quad \textrm{and}\quad \beta(p):=\min_{b\in Z(p)}\frac{\|\nabla^\C p(b)\|^2}{\|p\|_{\mathrm{FS}}^2}.$$
(Notice that $\beta$ is well defined, since scaling $p$ by a nonzero multiple does not change its value.)
Notice that both $\alpha$ and $\beta$ vanish precisely on $\Delta_{n,d}$, therefore, since $P_{n,d}$ is compact, Lojasiewicz's inequality tells that there exist $c_3, c_4>0$ such that $\alpha^{c_3}\leq c_4 \beta,$ i.e.
\begin{equation}\label{eq:lipschitz}\sfd_{\mathrm{FS}}(p, \Delta_{n,d})^{c_3}\leq c_4 \min_{b\in Z(p)}\frac{\|\nabla^\C p(b)\|^2}{\|p\|_{\mathrm{FS}}^2}.\end{equation}
Let now $q_1, q_2\in U_p$ and consider the Fubini--Study geodesic $\gamma:I\to P_{n,d}$ joining them, parametrized with $\|\dot \gamma(t)\|_{\mathrm{FS}}\equiv \sfd_{\mathrm{FS}}(q_1, q_2)$. Since $U_p$ is geodesically convex, we have $\gamma_t\in U_p$ for every $t\in I$. Let us now compute the length of the curve $\gamma$ in $W_2^{\mathrm{in}}, $ using \cref{coro:energy}:
\begin{align}L_{W_2^\mathrm{in}}(\gamma)&=\int_{0}^1(\h)_{\gamma(t)}(\dot \gamma_t, \dot \gamma_t)^{\frac{1}{2}}\dd t\\
&=\frac{1}{d\cdot \mathrm{vol}(\CP^{n-1})}\int_{0}^1\left(\int_{Z(\gamma_t)}\frac{|\dot \gamma_t(b)|^2}{\|\nabla^\C \gamma_t(b)\|^2}\mathrm{vol}_{Z(\gamma_t)}(\dd b)\right)^\frac{1}{2}\dd t\\
&\leq \frac{1}{d\cdot \mathrm{vol}(\CP^{n-1})}\int_{0}^1\left(\int_{Z(\gamma_t)}\frac{c_4|\dot \gamma_t(b)|^2}{\|\gamma_t\|_{\mathrm{FS}}^2\cdot \sfd_{\mathrm{FS}}(\gamma_t, \Delta_{n,d})^{c_3}}\mathrm{vol}_{Z(\gamma_t)}(\dd b)\right)^\frac{1}{2}\dd t\\
&\leq \left(c_5 \sfd_{\mathrm{FS}}(p, \Delta_{n,d})^{-c_3/2}\right)\int_{0}^1\left(\int_{Z(\gamma_t)}\frac{|\dot \gamma_t(b)|^2}{\|\gamma_t\|_{\mathrm{FS}}^2}\mathrm{vol}_{Z(\gamma_t)}(\dd b)\right)^\frac{1}{2}\dd t=(*),
\end{align}
where we have used \eqref{eq:double} and \eqref{eq:lipschitz}. 

Observe now that, since $\|b\|^2=1$, we have $|\dot \gamma_t(b)|^2\leq c_6\|\dot \gamma_t\|_{\mathrm{FS}}^2$ and therefore:
\begin{align}(*)&\leq \left(c_5 \sfd_{\mathrm{FS}}(p, \Delta_{n,d})^{-c_3/2}\right)\int_{0}^1\left(\int_{Z(\gamma_t)}\frac{c_6\|\dot \gamma_t\|_{\mathrm{FS}}^2}{\|\gamma_t\|_{\mathrm{FS}}^2}\mathrm{vol}_{Z(p_t)}(\dd b)\right)^\frac{1}{2}\dd t\\
&= \left(c_7 \sfd_{\mathrm{FS}}(p, \Delta_{n,d})^{-c_3/2}\right)\int_{0}^1\frac{\|\dot \gamma_t\|_{\mathrm{FS}}}{\|\gamma_t\|_{\mathrm{FS}}}\left(\int_{Z(\gamma_t)}\mathrm{vol}_{Z(\gamma_t)}(\dd b)\right)^\frac{1}{2}\dd t\\
&\leq \left(c_1 \sfd_{\mathrm{FS}}(p, \Delta_{n,d})^{-c_2}\right)\int_{0}^1\frac{\|\dot \gamma_t\|_{\mathrm{FS}}}{\|\gamma_t\|_{\mathrm{FS}}}\dd t\\
&=\left(c_1 \sfd_{\mathrm{FS}}(p, \Delta_{n,d})^{-c_2}\right) L_{\mathrm{FS}}(\gamma)=\left(c_1 \sfd_{\mathrm{FS}}(p, \Delta_{n,d})^{-c_2}\right)\cdot \sfd_{\mathrm{FS}}(q_1, q_2).
\end{align}
Since $W_2^\mathrm{in}(q_1, q_2)\leq L_{W_2}(\gamma)$, this concludes the proof.
\end{proof}

\begin{lemma}\label{lemma:dist2}For every $\epsilon>0$ there exists $c(\epsilon)>0$ such that
$$W_2^{\mathrm{in}}|_{P_{n,d}(\epsilon)\times P_{n,d}(\epsilon)}\leq c(\epsilon).$$
\end{lemma}
\begin{proof}For every point $q\in P_{n,d}(\epsilon)$ let $U_q$ be the Fubini--Study convex neighborhood centered at $q$ given by \cref{lemma:dist1}. Possibly shrinking this neighborhood, we can assume that it is contained in $P_{n,d}(\epsilon/2)$. Since $P_{n,d}(\epsilon)$ is compact in $\sfd_{\mathrm{FS}}$, it can be covered by finitely many such neighborhoods $U_{q_1}, \ldots, U_{q_r}$. 

Given $p_0, p_1\in P_{n,d}(\epsilon)$, let $\gamma:I\to P_{n,d}$ be a piecewise smooth curve joining $p_0$ and $p_1$, entirely contained in $\bigcup U_{q_j}$ and such that for every $j=1, \ldots ,r$ the set $\gamma^{-1}(U_{q_j})$ is an interval. Then the length of $\gamma$ is bounded by 
$$L_{\mathrm{FS}}(\gamma)\leq r \max_{j}\mathrm{diam}_{\mathrm{FS}}(U_{q_j})\leq c_1(\epsilon).$$

Arguing now as in the proof of \cref{lemma:dist1}, we see that
$$W_2^{\mathrm{in}}(p_0, p_1)\leq L_{\mathrm{W_2^{\mathrm{in}}}}(\gamma)\leq \max_j\left\{c_1\sfd_{\mathrm{FS}}(q_j, \Delta_{n,d})^{-c_2}\right\}L_{\mathrm{FS}}(\gamma)\leq c(\epsilon).$$
\end{proof}

\begin{proof}[Proof of \cref{thm:Lipschitz}]
Assume by contradiction that $\mathrm{id}$ \emph{is not} Lipschitz. Then, for every $n\in \N$ there exist $a_m, b_m\in P_{n,d}(\epsilon)$ such that
\beq \label{eq:dist3}\frac{W_2^\mathrm{in}(a_m, b_m)}{\sfd_{\mathrm{FS}}(a_m, b_m)}\geq m.\eeq
Since $W_2^\mathrm{in}$ is bounded on $P_{n,d}(\epsilon)\times P_{n,d}(\epsilon)$ (by \cref{lemma:dist2}), it follows that we must have $\sfd_{\mathrm{FS}}(a_m, b_m)\to 0$. In particular, since $P_{n,d}(\epsilon)$ is compact for the Fubini--Study topology, up to subsequences, we must have $a_{m_k}, b_{m_k}\to a$ as $k\to \infty$ in the Fubini--Study topology, with $a\in P_{n,d}(\epsilon)$. 

Let now $U_a$ be the neighborhood of $a$ given by \cref{lemma:dist1}. The sequences $\{a_{m_k}\}$ and $\{b_{m_k}\}$ lie definitely in $U_p$, therefore by \cref{lemma:dist1}:
\beq \label{eq:dist4}W_2^{\mathrm{in}}(a_{m_k}, b_{m_k})\leq \bigg(c_1\sfd_{\mathrm{FS}}(p, \Delta_{n,d})^{-c_2}\bigg)\sfd_{\mathrm{FS}}(a_{m_k}, b_{m_k}).\eeq
However, as $k\to \infty$, \eqref{eq:dist4} contradicts \eqref{eq:dist3}.

The second part of the statement follows from the fact that, being Lipschitz, $\mathrm{id}:(P_{n,d}(\epsilon), \sfd_{\mathrm{FS}})\to (P_{n,d}(\epsilon), W_2^\mathrm{in})$ is in particular continuous (and $(P_{n,d}(\epsilon), \sfd_{\mathrm{FS}})$ is compact).
\end{proof}

\section{Absolute Continuity of curves of roots}\label{sec:abscont}
In this section we deduce a first application of the new definition of Wasserstein distance introduced in \cref{S:OTalgebraic}, generalizing a result by A. Parusinski and A. Rainer \cite{Sobolev} on the regularity of roots of univariate polynomials, to the case of hypersurfaces. As we will see, for the generalization the existence of rational curves on the zero set of our curve of polynomials will play a role.

Let us start by recalling \cite[Theorem 1]{Sobolev}.

\begin{theorem}[Parusinski--Rainer]
Let $(\alpha,\beta) \subset \R$ be a bounded open interval and $P_a$ a monic polynomial 
$$
P_a(t)(z) = P_{a(t)}(z) = z^d + 
\sum_{j = 1}^{d} a_j(t) z^{d-j},
$$
with coefficients $a_j \in \mathscr{C}^{d-1,1}([\alpha,\beta])$, with $j = 1,\dots,d$.
Let $\lambda \in \mathscr{C}((\alpha,\beta))$ 
be a continuous root of $P_a$ on $(\alpha,\beta)$. 
Then $\lambda$ is absolutely continuous on $(\alpha,\beta)$ and belongs to the Sobolev space 
$W^{1,p}((\alpha,\beta))$ for every $1 \leq  p < d/(d-1)$. The derivative $\lambda'$ satisfies 
$$
\| \lambda' \|_{L^p((\alpha,\beta))}
\leq  
C(d,p)\max\{1,(\beta -\alpha)^{1/p} \}
\max_{1\leq j\leq d} \|a_j \|_{\mathscr{C}^{d-1,1}([\alpha,\beta])}^{1/j}.
$$
\end{theorem}

{\cite[Theorem 1]{Sobolev}}
 is optimal: in general 
 the derivatives of the roots of a smooth curve of monic polynomials of degree $d$ are not locally 
 $d/(d - 1)$-integrable (take for instance $z^d + t$), and the roots may have locally unbounded variation if the coefficients are only of class
 $\mathscr{C}^{d-1,\alpha}$ 
 for $\alpha < 1$.

We now reformulate \cite[Theorem 1]{Sobolev} 
in our formalism. 
The natural generalisation to higher degree 
will be then discussed.

Given a monic polynomial $P_{a(t)}$ as in
\cite[Theorem 1]{Sobolev} (for simplicity defined over $[0,1]$), 
we consider the associated homogeneous polynomial 
$$
p_t(z_0,z_1) 
= z_1^d + \sum_{j = 1}^{d} a_j(t)z_1^{d-j}z_0^j \in \C[z_0,z_1]_{(d)}.
$$
The corresponding measure 
$\mu_t : = \mu (p_t)$ will then be a probability measure over $\CP^1$: 
$$
\mu_t =\frac{1}{d} 
\sum_{z \in Z(P_{a(t)})} m_{z}(P_{a(t)})\delta_{[1,z]}.
$$
We now compute the metric speed of the curve
$[0,1] \ni t \mapsto \mu_t 
\in \mathscr{P}_p(\CP^1)$, with $p\geq 1$.
Suppose additionally the existence of 
continuous, pairwise distinct roots
$\lambda_1,\dots, \lambda_r$ of $P_{a(t)}$ with 
$r \leq  d$.  
For $h$ sufficiently small
we then have
$$
W_p(\mu_t,\mu_{t+h})^p = 
\frac{1}{d}\sum_{j=1}^r 
m_{[1,\lambda_j(t)]}(p_t)
\sfd_{FS}([1,\lambda_j(t)], [1,\lambda_j(t+h)])^p.
$$
Hence taking the limit 
as $h \to 0$ we obtain
$$
\frac{C_1 }{d} \sum_{j=1}^r 
m_{[1,\lambda_j(t)]}(p_t) 
|\lambda_j(t)'|^p
\leq 
|\dot\mu_t|_p^p \leq 
\frac{C_2 }{d} \sum_{j=1}^r 
m_{[1,\lambda_j(t)]}(p_t) 
|\lambda_j(t)'|^p,
$$
where $C_1,C_2$ are two positive constants 
given by the equivalence of the metric speeds calculated in $\CP^1$ or in $\C$.
We see therefore that the absolute continuity of the roots and the integrability of their derivative  $|\lambda_j(t)'|^p$ is actually equivalent to 
the $p$-absolute continuity of the curve 
$t \mapsto \mu_t$ inside 
$\mathscr{P}_p(\CP^1)$.
In particular 
to $
(\mu_t)_{[0,1]} \in AC^p([0,1];\mathscr{P}_p(\CP^1))$.

In order to formulate and prove the first generalization of {\cite[Theorem 1]{Sobolev}}, we need one extra definition.
Let $\tau^*\to \mathbb{G}(1, n)$ be the dual of the tautological bundle and 
$$E:=\mathrm{sym}^{(d)}(\tau^*)\to \mathbb{G}(1,n)$$
be its $d$--th symmetric power. Recall that $E$ is a complex vector bundle on $\mathbb{G}(1,n)$; given $\ell\in \mathbb{G}(1,n)$, the fiber $E|_\ell$ is naturally isomorphic to the set of homogeneous polynomials of degree $d$ on $\ell$; in particular the complex rank of $E$ is $d+1$. Every homogeneous polynomial $p\in H_{n, d}$ gives rise to a section $\sigma_{p}$ of $E$, defined by
$$\sigma_p(\ell)=p|_\ell.$$
A line $\ell$ is contained in the zero set of $Z(p)$ if and only if if $\sigma_p(\ell)=0$. 
\begin{definition}[Fano discriminant]\label{definition:Fano}We denote by $F_{n,d}\subset P_{n,d}$ and call it the \emph{Fano discriminant} the set of polynomials $[p]\in P_{n,d}$ such that the section $\sigma_p$ \emph{is not} transversal to the zero section of $E$.   
\end{definition}

When $n=1$, the Fano discriminant is empty: in this case the only line is $\CP^1$ itself, and no nonzero polynomial $p\in\C[z_0, z_1]_{(d)}$ can vanish identically on this line. If $p\in P_{n,d}\setminus  {F_{n,d}}$, then the set of lines contained in $Z(p)$ is a smooth, complex subset of the Grassmannian, of complex codimension $d+1$ (the rank of $E$). In particular, if $d=2n-3$ and $p\notin  {F_{n,d}}$, then there are finitely many lines on $Z(p)$. If $d>2n-3$ and $p\notin  {F_{n,d}}$, then there is no line on $Z(p)$ (the ``generic'' section misses the zero section). This explains the assumptions in next theorem, which gives a regularity for a curve of measures of the form $\mu_t=\mu(p_t)$. Of course, the relevant fact is that $F_{n,d}$ is contained in a proper algebraic set \cite[Theorem 6.34]{3264}.

\begin{theorem}\label{T:Parusinski}
Let $d>2n-3$. For every curve $p_t \in \mathscr{C}^{d-1,1}(I, \mathbb{C}[z_{0},\dots,z_{n}]_{(d)})$, such that $p_t\notin {F_{n,d}}$ for every $t\in I$, 
and for every $1 \leq q < d/(d-1)$ the curve
$\mu_t:=\mu(p_t) \in AC^{q}(I,\mathscr{P}_q(\CP^n)).$
\end{theorem}

\begin{proof}
Fix any $q$ with $1 \leq q < d/(d-1)$. We will provide an upper bound
on the increment of $\mu_t$
via an explicit coupling 
between $\mu_{t+h}$ and $\mu_{t}$ 
constructed by means of 
integral geometry. 
For any $t \in [0,1]$,
the claim of Theorem \ref{thm:ext} can be rewritten as 
\beq 
\mu_{t}=\int_{\mathbb{G}(1,n)}
\mu (\gamma(t)|_{\ell}) \,\mathrm{vol}_{\mathbb{G}(1,n)}(\dd \ell),
\eeq
where, in particular, for any continuous function 
$f$ of $\CP^{n}$, the map 
$$
\mathbb{G}(1,n) \ni l \longmapsto  \int f \, \mu(\gamma(t)|_{l})  = \frac{1}{d} \sum_{z\in Z(\gamma(t))\cap \ell}m_z(\gamma(t)|_{\ell}) f(z), 
$$
is measurable and bounded. 
For ease of notation we will write $\mu_{t,l}$ 
in place of $\mu(\gamma(t)|_{l})$.

We consider then the optimal transport problem between 
$\mu_{t,l}$ and $\mu_{t+h,l}$ with cost function given by $\sfd_{std}^{q}$ 
for which, 
by compactness of $\Pi(\mu_{t,l},\mu_{t+h,l})$ and linearity of the functional,  
there exists a minimizer that we denote by
$\xi_{t,h,l} \in \Pi(\mu_{t,l},\mu_{t+h,l})$.
By a classical selection argument, 
the map 
$\mathbb{G}(1,n) \ni l 
\mapsto \xi_{t,h,l} \in 
\mathscr{P}(\CP^n \times \CP^n)$
can be assumed to be measurable; 
we can then define 
$$
\xi_{t,h} : = \int_{\mathbb{G}(1,n)}\xi_{t,h,l} \,\mathrm{vol}_{\mathbb{G}(1,n)}(\dd \ell) \in \Pi(\mu_{t},\mu_{t+h}),
$$
and compute: 
\begin{align}
W_{q}(\mu_{t+h},\mu_{t})^{q} 
&~ \leq 
\int_{\CP^{n} \times \CP^{n}} \sfd_{std}^{q}(z,w) \,\xi_{t,h}(\dd z \dd w) 
\\
&~ = 
\int_{\mathbb{G}(1,n)} \int_{\CP^{n} \times \CP^{n}} \sfd_{std}^{q}(z,w) \, \xi_{t,h,l} \,\mathrm{vol}_{\mathbb{G}(1,n)}(\dd \ell)  
\\
&~ = 
\int_{\mathbb{G}(1,n)} W_{q}(\mu_{t+h,l}, \mu_{t,l})^{q}
\,\mathrm{vol}_{\mathbb{G}(1,n)}(\dd \ell).
\end{align}
We will now estimate $W_{q}(\mu_{t+h,l}, \mu_{t,l})^{q}$.

To this end, fix a line $\ell_0$ and denote by $L_0\simeq \C^{2}\subset \C^{n+1}$ a complex vector space such that $\mathrm{P}(L_0)=\ell_0$. On a small enough neighborhood $U_1(\ell_0)$ of $L_0$ in $G(2, n+1)=\mathbb{G}(1,n)$, by \cref{lemma:nhood} (proved below), we can find smooth maps $e_0, e_1:U_1(\ell_0)\to \C^{n+1}$ such that: (i) for every $L\in U_1(\ell_0)$, where $\mathrm{P}(L)=\ell,$ the set $\{e_0(\ell), e_1(\ell)\}$ is a Hermitian orthonormal basis of $L$ and (ii) for every $t\in I$ the point $e_1(\ell_0)$ \emph{is not} a zero of  $p_t$.

For every $\ell \in U_1(\ell_0)$ 
define the polynomial:
\beq 
a_{t,\ell}(z_0, z_1):=p_t(z_0e_0(\ell) +z_1e_1(\ell))=\sum_{k=0}^da_{k}(t, \ell)z_0^{k}z_1^{d-k}\in \C[z_0, z_1]_{(d)}.
\eeq
Notice that, the functions $t\mapsto a_{k}(t, \ell)$ are $\mathscr{C}^{d-1, 1}.$ Moreover, since $a_{0}(t, l_0)=p_t(e_1(\ell_0))\neq 0$ for all $t\in I$, then $a_0(t, \ell_0)\neq 0$ for all $t\in I$. In particular (by continuity of $a_0$) there is a neighborhood $U_2(\ell_0)$ with $\clos(U_2(\ell_0))\subseteq U_1(\ell_0)$ such that $a_0(t, \ell)\neq 0$ for every $\ell\in \clos(U_{2}(\ell_0))$ and for every $t\in I$. For $(t, \ell)\in I\times U_2(\ell)$ define:
\beq b_k(t, \ell):=\frac{a_k(t,\ell)}{a_0(t, \ell)}.\eeq
For every $\ell\in U_2(\ell_0)$ and $t\in I$ we consider now the polynomial $b_{t, \ell}\in \C[z_0, z_1]_{(d)}$ defined by
\beq b_{t, \ell}(z_0, z_1):=\sum_{k=0}^db_k(t, \ell)z_0^k z_1^{d-k}=z_1^d+\sum_{k=1}^{d}b_k(t, \ell)z_0^k z_1^{d-k}.\eeq
Notice that $b_{t,\ell}$ has the same zeroes of $a_{t, \ell}$ and that the functions $t\mapsto b_k(t, \ell)$ are again $\mathscr{C}^{d-1, 1}$ (moreover, $b_0\equiv 1$). 

For every $\ell\in U_1(\ell_0)$ consider now the map $\varphi_\ell:\CP^1\to \ell$ given by 
\beq \varphi_\ell([z_0,z_1]):=[z_0e_0(\ell) +z_1e_1(\ell)].\eeq  
Such map is an isometry, since $\{e_0(\ell), e_1(\ell)\}$ is a Hermitian orthornomal basis.  In particular, if $\lambda_1(t, \ell), \ldots, \lambda_d(t, \ell)$ are  continuous selections of the zeroes of $b_{t, \ell}$, we have:
\beq 
W_q(\mu_{t+h,l}, \mu_{t,l})^q
\leq \frac{1}{d}\sum_{j=1}^d
\sfd_{std}^q(\lambda_{j}(t+h, \ell),
\lambda_{j}(t, \ell)).
\eeq 
Denote now by $\sigma:\CP^1\setminus\{[0,1]\}\to \C$ the stereographic projection, $\sigma([z_0, z_1])=\frac{z_1}{z_0}$, and by $\widetilde{\lambda}_k(t, \ell):=\sigma(\lambda_k(t, \ell)).$ Then $\widetilde\lambda_1(t, \ell), \ldots, \widetilde\lambda_d(t, \ell)$ are continuous selections of the zeroes of the polynomial $\widetilde{b}_{t, \ell}$ given by
\beq \widetilde{b}_{t, \ell}(z):=b_{t, \ell}(1, z)=z^d+\sum_{k=1}^{d}b_k(t, \ell)z_1^{d-k}.
\eeq 
Since there exists $C_1>0$ such that $\mathrm{dist}_{\CP_1}(x,y)\leq  C_1\cdot \mathrm{dist}_{\C}(\sigma(x), \sigma(y))$ for every $x, y\in \CP^1\setminus \{[0,1]\}$, then for some $C_2>0$
\beq 
\sfd_{std}(\lambda_{j}(t+h, \ell) , \lambda_{j}(t, \ell))\leq C_2\cdot |\widetilde\lambda_{j}(t+h, \ell)-
\widetilde\lambda_{j}(t, \ell)|_{(\C, g_\mathrm{std})}.\eeq
We use now 
\cite[Theorem 1]{Sobolev}, which 
guarantees that each $\widetilde{\lambda}_j(\cdot, \ell)\in AC^q(I, \C)$ and that
\beq 
\|\partial_t\widetilde{\lambda}_{j}(\cdot, \ell)\|_{L^q}=\left(\int_{I}|\partial_t\widetilde\lambda_{j}(t, \ell)|^q_{(\C, g_\mathrm{std})}\dd t\right)^{1/q}\leq C(d, q)\max_{1\leq j\leq d}\|b_j(\cdot, \ell)\|_{C^{d-1, 1}}^{1/j}.
\eeq
Since $a_0(t,\ell) \neq 0$ for all 
$t \in I$ and $\ell \in U_2(\ell_0)$, we deduce that  
$$
\max_{1\leq j\leq d}\|b_j(\cdot, \ell)\|_{C^{d-1, 1}}^{1/j} \leq  
C(\gamma,\ell_0,q) < \infty,
$$
for all $t \in I$ and $\ell \in U_2(\ell_0)$.
Hence for each $\ell \in U_2(\ell_0)$:
\begin{align}
\int_{[0,1-h]} \left(\frac{W_q(\mu_{t+h,\ell}, \mu_{t,\ell})}{h}\right)^q\,dt
\leq &~
\frac{1}{d}
\sum_{j=1}^d \int_{[0,1-h]}
\left(\frac{\sfd_{std}(\lambda_{j}(t+h, \ell),
\lambda_{j}(t, \ell))}{h}\right)^q\, dt \\
\leq &~
\frac{1}{d}
\sum_{j=1}^d \int_{[0,1-h]}
\left(\frac{|\widetilde\lambda_{j}(t+h, \ell) -
\widetilde\lambda_{j}(t, \ell)|}{h}\right)^q\, dt \\
\leq &~
\frac{1}{d}
\sum_{j=1}^d \int_{[0,1-h]}
\left(\frac{1}{h}\int_{[t,t+h]} 
|\partial_t \widetilde\lambda_{j}(s,\ell)|^q ds\right)\, dt  \\
=&~
\frac{1}{d}
\sum_{j=1}^d 
\int_{[h,1]} |\partial_t \widetilde\lambda_{j}(s,\ell)|^q \,ds \\
\leq  &~ C(\gamma,\ell_0,q) 
\end{align}
where the estimate is uniform for 
$\ell \in U_2(\ell_0)$.

By compactness, we can cover $\mathbb{G}(1, n)$ with finitely many open sets of the form $U(q,\gamma, \ell_1), \ldots, U(q,\gamma, \ell_r)$ and, setting $c(q, \gamma):=\max_k C(\gamma, \ell_k,q)$ we have:
\begin{align}
\int_{[0,1-h]}
\left(
\frac{W_{q}(\mu_{t+h},\mu_{t})}{h}
\right)^{q} \,dt
&~ 
= 
\int_{[0,1-h]}
\int_{\mathbb{G}(1,n)} \left(\frac{W_{q}(\mu_{t+h,l}, \mu_{t,l})}{h} \right)^{q}
\,\mathrm{vol}_{\mathbb{G}(1,n)}(\dd \ell) \, dt
 \\
&~ \leq  \int_{\mathbb{G}(1,n)} C(\gamma,\ell,q) \,
\,\mathrm{vol}_{\mathbb{G}(1,n)}(\dd \ell) < C, 
\end{align}
for any $h < 1$ 
 and 
$C$ independent of $h$ and $t$.
This estimate is known to be sufficient 
to prove the claim $\mu\circ \gamma\in AC^{q}(I,\mathscr{P}_q(\CP^n))$, 
see for instance 
\cite[Lemma 1]{Lisini06}.
\end{proof}

It remains to prove the Lemma used in the proof.
\begin{lemma}\label{lemma:nhood}Let $p_t:I\to \C[z_0, \ldots, z_n]_{(d)}$ be a curve satisfying the assumptions of \cref{T:Parusinski}. For every $\ell_0\in \mathbb{G}(1,n)$  denote by $L_0\simeq \C^{2}\subset \C^{n+1}$ the complex vector space such that $\mathrm{P}(L_0)=\ell_0$. There exists a neighborhood $U_1(\ell_0)$ of $L_0$ in $G(2, n+1)=\mathbb{G}(1,n)$ and smooth maps $e_0, e_1:U(\ell_0)\to \C^{n+1}$ such that: 
\begin{enumerate}
\item[(i)] for every $L\in U_1(\ell_0)$, where $\mathrm{P}(L)=\ell,$ the set $\{e_0(\ell), e_1(\ell)\}$ is a Hermitian orthonormal basis of $L$; 
\item[(ii)] for every $t\in I$ the point $e_1(\ell_0)$ \emph{is not} a zero of the polynomial $p_t$. 
\end{enumerate}
\end{lemma}
\begin{proof} Consider the set
$$U:=\{(z, t)\in \ell_0\times I\,|\, p_t(z)=0\}\subset L_0\times I,$$
together with the restrictions $\pi_1:U\to \ell_0$ and $\pi_2:U\to I$  of the projections on the two factors.
Since $p_t\notin {\Delta}_{n,d}^\mathrm{Fano}$, the hypothesis $d>2n-3$ imply that there is no line on $Z(p_t)$ and in particular for every $t\in I$ the polynomial $p_t|_{L_0}$ is not identically zero. For every $t\in I$, the set $\pi_2^{-1}(t)$ consists of finitely many points, therefore it has Hausdorff dimension zero; consequently $U$ has Hausdorff dimension at most $1$ in $\ell_0\times I\simeq \CP^1\times I$. Since $\pi_1$ is a smooth map, then also $\pi_1(U)$ has Hausdorff dimension at most $1$. In particular there exists a point $x\in \ell_0\setminus \pi_1(U).$ This point $x$ is the projectivization of some complex one--dimensional vector space $V\subset L_0$ and we can therefore pick a norm one vector $e_1(\ell_0)\in V.$ We set $e_0(\ell_0)\in L_0$ to be any norm--one vector in the Hermitian orthogonal complement of $e_1(\ell_0)$. There exists now a neighborhood $U(\ell_0)$ such that for every line $\ell=\mathrm{P}(L)\in U(\ell_0)$, the orthogonal projections $\{\mathrm{proj}_Le_0(\ell_0), \mathrm{proj}_Le_1(\ell_0)\}$ form a basis for $L$. Applying the Gram--Schmidt procedure to this basis yields the desired functions $e_0(\cdot), e_1(\cdot):U(\ell_0)\to \C^{n+1}$. \end{proof}
In the way of proving 
\cref{T:Parusinski} 
we have also proved the following
upper bound on the metric speed 
via integral geometry: 
using the notation of the above proof, for every curve $\gamma \in C^{d-1,1}(I, \mathbb{C}[z_{0},\dots,z_{n}]_{(d)})$, and for every $1 \leq q < d/(d-1)$, we have
\beq
|\dot{\mu_t}|_q^q\leq  \int_{\mathbb{G}(1,n)} 
|\dot{\mu_{t,\ell}}|_q^q
\,\mathrm{vol}_{\mathbb{G}(1,n)}(\dd \ell),
\eeq
for almost all $t \in [0,1]$; 
here $|\dot{\mu_t}|_q$ denotes the metric speed of $\mu_t$ in the $q$-Wasserstein space.

We now address the compatibility 
of \cref{T:Parusinski} with the sharp result obtained in \cite{Sobolev} 
by showing that, 
in the case $n = 1$,
\cref{T:Parusinski} recovers
\cite{Sobolev}.

\subsection{Rational maps}In this section we prove a further generalization of \cite[Theorem 1]{Sobolev}. The main assumption of \cref{T:Parusinski} was the condition that the curve $p_t$ \emph{does not} hit the Fano discriminant $F_{n,d}$; under this assumption, if $p_t\in \mathscr{C}^{d-1, 1}$, then $\mu_t\in AC^q(I, \mathscr{P}_q(\CP^n))$ for all $1\leq q<d/(d-1).$ We can remove the condition $p_t\notin  {F_{n,d}}$, asking for more regularity on the curve $p_t$ at the price of losing some regularity on the curve of measures $\mu_t$.

\begin{theorem}\label{thm:rational}For every $n\in \mathbb{N}$ there exists $e(n)>0$ such that for every curve $p_t \in \mathscr{C}^{k,1}(I, \mathbb{C}[z_{0},\dots,z_{n}]_{(d)})$,  with $k\geq e(n) d-1$,
and for every $1 \leq q < e(n)d/(e(n)d-1)$ the curve
$\mu_t:=\mu(p_t) \in AC^{q}(I,\mathscr{P}_q(\CP^n)).$
\end{theorem}

The main idea for the proof of \cref{thm:rational} is to replace the use of lines, in the integral geometry approach, with the use of rational maps.

\begin{definition}For every $n, e\in \mathbb{N}$ we denote by $R_e(\CP^n)\subset \CP^{(n+1)(e+1)-1}$ the set of $(n+1)$--tuples $[q_0, \ldots, q_n]$ where $q_0, \ldots, q_n\in H_{1,e}$ are polynomials with no common zeroes. (The set $R_e(\CP^n)$ is a Zariski open set.) We call elements of $R_e(\CP^n)$ \emph{rational maps}, since each of them  defines a map $\nu:\CP^1\to \CP^n$ by
$$
\nu(w):=[q_0(w), \ldots, q_n(w)].
$$
\end{definition}
If $Z\subset \CP^n$ is a hypersurface, we denote by $R_e(Z)\subset R_e(\CP^n)$ the set of rational maps such that $\nu(\CP^1)\subset Z$.
Notice that, given a polynomial $p\in H_{n,d}$ and a rational map $\nu:\CP^1\to \CP^n$, the composition $p\circ \nu$ defines an element of $H_{1, de}$ (up to multiples) and $\nu(\CP^1)\subset Z(p)$ if and only if $p\circ \nu\equiv 0.$

We can generalize the construction from \cref{sec:familiesgeneral} and recover the map $\mu:P_{n,d}\to \mathscr{P}(\CP^n)$ integrating over rational maps of high degree (instead of lines, which are images of rational maps of degree one).
\begin{proposition}
\label{P:rationalcoupling}
Let $\nu\in R_e(\CP^n)$ such that $\nu:\CP^1\to \CP^n$ is an embedding.  Then, for every $p\in P_{n,d}$ and $f\in \mathscr{C}^0(\CP^n)$ we have
\beq\label{eq:rationalslice}\int_{\CP^n}f \dd \mu(p)=\frac{1}{d\cdot e}\int_{U(n+1)}\sum_{p(g(\nu(w)))=0}m_{w}(p\circ g\circ \nu)f(g(\nu(w))) \dd g.\eeq
\end{proposition}
\begin{proof} Observe first that, if $Z\subset \CP^n$ is a smooth complex hypersurface and $U\subseteq Z$ is an open set, the integral geometry formula \eqref{eq:igfc} applied with the choice $A=U$ and $B=\nu(\CP^1)$ gives
$$\int_{U(n+1)}\#U\cap g\nu(\CP^1)\dd g=\frac{\mathrm{vol}(U)}{\mathrm{vol}(\CP^{n-1})}\frac{\mathrm{vol}(\nu(\CP^1))}{\mathrm{vol}(\CP^1)}.$$
Moreover, letting $U=Z=\CP^{n-1}$, we see that $\frac{\mathrm{vol}(\nu(\CP^1))}{\mathrm{vol}(\CP^1)}=e$.

Still letting $Z\subset \CP^n$ be smooth, arguing as in \cref{lemma:dis} and denoting by $\lambda:\mathscr{C}^0(Z)\to\R$ the linear functional
$$\lambda(f):=\int_{U(n+1)}\frac{1}{d\cdot e}\sum_{z\in Z\cap \nu(\CP^1)}f(z)\dd g,$$
we see that $\lambda\geq 0$ and we find a unique measure $\sigma$ on $Z$ such that $\lambda(f)=\int_Z f(z) \sigma (\dd z).$ Approximating the characteristic function of an open set $U$ by continuous functions, as in \cref{lemma:dis}, we see that $\sigma=\mathrm{vol}_Z$.

Observe now that, if $p\in P_{n,d}\setminus \Delta_{n,d}$, then for almost every $g\in U(n+1)$ the set $\nu(\CP^1)$ and $Z(p)$ intersect transversally and for every point of intersection $z=\nu(w)\in \nu(\CP^1)\cap Z$ we have $m_w(p\circ g\circ \nu)=1$. This proves the identity \eqref{eq:rationalslice} for $p\in P_{n,d}\setminus\Delta_{n,d}$. Finally, as in the proof of \cref{thm:ext}, we see that \eqref{eq:rationalslice} is continuous on $P_{n,d}$ and, since $P_{n,d}\setminus \Delta_{n,d}$ is dense in it, then \eqref{eq:rationalslice}
agrees with \eqref{eq:ext}.
\end{proof}

We will need the following recent result of B. Lehmann, E. Riedl and S. Tanimoto, privately communicated to the authors. Roughly speaking, the result says that, given positive integers $n$ and $T$, there exists $\xi(n, T)>0$ such that for any family $X\subset R_{e}(\CP^n)$ with $e \geq \xi$ and with $\dim_\C(X)\geq(e+1)(n+1)-T$, we have 
$$\bigcup_{\nu\in X}\nu(\CP^1)=\CP^n.$$ In particular such family of curves cannot lie on a single hypersurface. 
\begin{proposition}[Lehmann--Riedl--Tanimoto]\label{thm:riedl}For every $n, T\in \N$ there exists $\xi(n, T)>0$ such that for every $e\geq \xi(n, T)$ and for every $p\in P_{n,d}$ we have:
$$\dim_\C(R_e(Z(p))\leq (n+1)(e+1)-T.$$ 
\end{proposition}
In the special case $T(n):=\lfloor\frac{(n+1)^2}{2}\rfloor+1$, using \cref{thm:riedl}, we get the number
\beq \label{eq:riedl}e(n):=\xi( n, T(n)).\eeq
The next proposition ensures that, given a curve $p_t$ of polynomials, we can find a ``good'' rational map to perform integral geometry on the family $\{Z(p_t)\}_{t\in I}.$

\begin{proposition}\label{P:goodcurve}
Let $p_t:I\to P_{n,d}$ be a $\mathscr{C}^k$ curve, with $k\geq 1$ and $e(n)>0$ be given by \eqref{eq:riedl}. For every $e\geq e(n)$ there exists $\nu\in R_e(\CP^n)$ such that
\begin{enumerate}
\item for every $t\in I$ and for every $g\in U(n+1)$ the image of $g\circ \nu$ \emph{is not} contained in $Z(p_t);$
\item $\nu:\CP^1\to \CP^n$ is an embedding.
\end{enumerate}
\end{proposition}
\begin{proof}Denote by $Y\subset I\times R_e(\CP^n)$ the set
\begin{align}Y&=\left\{(t, \nu)\,\bigg|\, \exists g\in U(n+1),\, g\nu(\CP^1)\subset Z(p_t)\right\}\\
&=\left\{(t, \nu)\,\bigg|\, \nu\in \bigcup_{g\in U(n+1)} g R_e(Z(p_t))\right\}.
\end{align}
Let also $\pi_1:Y\to I$ and $\pi_2:Y\to R_e(\CP^n)$ denote the projections on the two factors. Notice that every element $\nu\notin \pi_2(Y)$ satisfies condition (1) of the statement. We will  first prove that there is plenty of such elements.

For every $t\in I$, the set $\pi_1^{-1}(t)$ is the continuous semialgebraic image of the semialgebraic set $U(n+1)\times R_e(Z(p_t))$ under the map $(g, \nu)\mapsto g\nu$. In particular
$$\dim(\pi_1^{-1}(t))\leq \dim(U(n+1))+2\dim_\C(R_e(Z(p_t)))\leq \dim_\R(R_e(\CP^n))-2.$$
This implies that the Hausdorff dimension of $Y$ is at most $\dim_\R(R_e(\CP^n))-1$ and, since $\pi_2$ is a smooth map, the Hausdorff dimension of $\pi_2(Y)$ is also at most $\dim_\R(R_e(\CP^n))-1$. This means that $\pi_2(Y)$ is a semialgebraic set of codimension at least one, in particular its complement $A_1$ contains an open, dense, semialgebraic set all of whose elements satisfy condition (1) of the statement.

The set $A_2$ of $\nu\in R_{e}(\CP^n)$ satisfying condition $(2)$ of the statement is open and dense. Therefore it is enough to pick $\nu\in A_1\cap A_2\neq \emptyset$ to conclude the proof.
\end{proof}
We are now ready to prove \cref{thm:rational}.

\begin{proof}[Proof of \cref{thm:rational}]

We are now in position to revisit the 
proof of \cref{T:Parusinski}.

Consider 
$p_t \in \mathscr{C}^{k,1}(I, \mathbb{C}[z_{0},\dots,z_{n}]_{(d)})$
for some natural number $k$. 
By \cref{thm:riedl}
and \cref{P:goodcurve}, letting $e(n)$ be defined by  \eqref{eq:riedl},
for every $k \geq  e(n)$
we have the existence of a 
$\nu\in R_e(\CP^n)$ such that
$\nu:\CP^1\to \CP^n$ is an embedding and 
 for every $t\in I$ and for every $g\in U(n+1)$ the image of $g\circ \nu$ is not contained in $Z(p_t)$.
In particular,  
$p_{t,g}: = p_t\circ g \circ \nu \in H_{1,de}$, up to multiplies.  
We can therefore define 
$$
\mu_{t,g} : =(g \circ \nu)_\sharp( \mu(p_{t,g})) \in \mathscr{P}(\CP^n).
$$
Notice indeed that 
$\mu(p_{t,g}) \in \mathscr{P}(\CP^1)$
and $(g \circ \nu) : 
\CP^1 \to \CP^n$.
Then  
\cref{P:rationalcoupling} 
can be restated by saying 
that 
for any measurable function $f$
\begin{align}
\int_{\CP^n}f(z)\,\mu_t (\dd z)
=&~  
\frac{1}{de}\int_{U(n+1)}
\int_{\CP^1}f(g(\nu(z)))
\left(\sum_{w \in Z(p_{t,g})}
m_w(p_{t,g}) \delta_w \right)(\dd z) \, \dd g \\
= &~
\int_{U(n+1)}
\int_{\CP^1}f(g(\nu(z)))
\mu(p_{t,g})(\dd z) \, \dd g \\ 
= &~
\int_{U(n+1)}
\int_{\CP^n}f(z)
(g\circ \nu )_{\sharp}\mu(p_{t,g})(\dd z) \, \dd g, 
\end{align}
giving the concise identity: 
\begin{equation}\label{E:disintegration}
\mu_t = \int_{U(n+1)}
\mu_{t,g} \, \dd g.
\end{equation}
Reasoning as in the proof of \cref{T:Parusinski} 
where we constructed a coupling between 
$\mu_t$ and $\mu_{t+h}$
by gluing the couplings 
line by line, 
one obtains that 
for any $q > 1$
$$
|\dot{\mu_t}|_q^q
\leq  
\int_{U(n+1)}
|\dot \mu_{t,g}|_q^q \,\dd g.
$$
Next we write 
$$
p_t(g(\nu))(z_0,z_1) = 
\sum_{k=0}^{de} a_k(t,g)
z_0^kz_1^{de - k} \in \C[z_0,z_1]_{(de)};
$$
the functions $t \mapsto 
a_k(t,g)$ are $\mathscr{C}^{k-1,1}(I;\C)$
and 
since the image of $g\circ \nu$ is not contained in $Z(p_t)$, 
$p_t(g(\nu))$ is not identically zero. 
Hence up to pre-composing 
$\nu$ with $h \in U(2)$, 
we can assume that 
$$
a_0(t,g) = 
p_t (g(\nu))(0,1) \neq 0.
$$
We can now proceed a in
the proof of \cref{T:Parusinski}
by dehomogenizing the polynomial 
$p_{t,g}$ 
and invoking 
\cite[Theorem 1]{Sobolev}
to deduce that,
for any 
$1 \leq q < e(n)d/(e(n)d-1)$,
the curve $\mu(p_{t,g})
\in AC^q(I;
\mathscr{P}_q(\CP^1)$
with the 
$L^q$-norm in time of the
$q$-metric speed 
bounded 
locally uniformly in $g \in U(n+1)$.

In order 
to deduce the same regularity for 
$\mu_{t,g}$ 
we reason as follows.
By the isometric injection of 
each (metric) space $X$ into 
the 
$q$-Wasserstein space over itself,
the push-forward of the smooth map
$g\circ \nu : \CP^1 \to \CP^n$
induces a Lipschitz map 
between $\mu(P_{1,ed})$
and $\mathscr{P}(\CP^n)$; 
in particular 
$$
|\dot \mu_{t,g}|_q
\leq  C(\nu)
|\dot \mu(p_{t,g})|_q, 
$$
where $C(\nu)$ 
only depends on the Lipschitz constant of $\nu$;
hence 
the proof is completed.
\end{proof}

This regularity 
is sufficient to deduce the compactness of the geodesic space 
$(P_{n,d},W_q^{\textrm{in}})$ 
for $q$ 
in the range of 
\cref{thm:rational}.

\begin{remark}
In the range of 
\cref{thm:rational} we will have $q \leq  2$. 
Hence by Jensen inequality for any couple of probability measures $W_q(\mu_0,\mu_1) \leq  W_2(\mu_0,\mu_1)$ and therefore 
also the metric speeds will have the same ordering: if 
$\mu_t \in AC^2(I;\mathscr{P}_2(\CP^n))$ 
then 
$\mu_t \in AC^q(I;\mathscr{P}_q(\CP^n))$ 
and $|\dot \mu_t|_q \leq  |\dot \mu_t|_2$.
Hence \cref{thm:finite} implies that 
$W_q^\textrm{in}$ is finite 
over $P_{n,d}$. 
Moreover all the conclusions of 
\cref{thm:propertyw} will remain valid for 
$(P_{n,d},W_q^{\textrm{in}})$.
\end{remark}

\begin{theorem}\label{thm:compactq}
Let $e(n)\in \N$ be given by \eqref{eq:riedl}. Then, for every $1 \leq q < e(n)d/(e(n)d-1)$, 
the complete and separable geodesic space $(P_{n,d}, W_q^\textrm{in})$ is compact.
\end{theorem}

\begin{proof}
Consider a sequence $p_k \in P_{n,d}$. 
Then there exists a subsequence, 
that we omit in the notation, 
and $p \in P_{n,d}$ 
such that $\sfd_{\textrm{FS}}(p_k,p) \to 0$
as $k \to \infty$.
Then, by \cite[Lemma 12.2]{KrieglMichor}, there exists a smooth curve 
$\gamma : I \to P_{n,d}$
and $t_k \in I$ such that 
for each $k\in \N$
$\gamma(t_k) = p_k$ and $t_k \to 0$ as $k\to \infty$.
Then by \cref{thm:rational} 
the curve $\mu_t : = \mu(\gamma)$
is in 
$AC^q(I;\mathscr{P}_q(\CP^n))$
for 
$1 \leq q < e(n)d/(e(n)d-1)$.
Hence 
$$
W_q^\textrm{in}(p_0,p_k) = 
W_q^\textrm{in}(\mu_0,\mu_{t_k})
\leq  \left(\int_0^{t_k}|\dot \mu_s|_q^q \,\dd s\right)^{1/q},
$$
proving that $W_q^\textrm{in}(p_0,p_k)\to 0$ 
as $k\to \infty$ and therefore the claim.
\end{proof}

\section{A Wasserstein approach to  condition numbers}\label{sec:condition}
In this section we adopt the approach of Wasserstein metric geometry for the study of condition numbers of polynomial systems.
Let us recall the standard framework from \cite{burgisser}.

Given a list of degrees $\underline{d}=(d_1, \ldots, d_n)\in \N^n$ we define $P_{n, \underline{d}}:=P_{n,d_1}\times \cdots \times P_{n, d_n}$ and consider the \emph{solution variety}
$$V:=\{((p_1, \ldots, p_n), [z])\in P_{n,\underline{d}}\times\CP^n\,|\, p_1(z)=\cdots =p_n(z)=0\},$$
together with the two projections $\pi_1:V\to P_{n, \underline{d}}$ and $\pi_2:V\to \CP^n$. We denote by $\Sigma'\subset V$ the set of critical points of $\pi_1.$ 

The \emph{discriminant} $\Delta_{n, \underline{d}}\subset P_{n, \underline{d}}$ is defined by
$$\Delta_{n,\underline{d}}:=\{(p_1, \ldots, p_n) \in P_{n, \underline{d}}\,|\,\textrm{$p_1(z)=\cdots=p_n(z)=0$ is not regular on $\C^{n+1}\setminus \{0\}$}\}. $$
Given $[w]\in \CP^n$ we also consider $\Delta_{n,\underline{d}, [w]}\subset \Delta_{n, \underline{d}}$ defined by
$$\Delta_{n,\underline{d}, [w]}:=\{(p_1, \ldots, p_n) \in P_{n, \underline{d}}\,|\,\textrm{$p_1(z)=\cdots=p_n(z)=0$ is not regular at $w$}\}.$$

The \emph{normalized condition number} $\nu_{\mathrm{norm}}:V\setminus \Sigma'\to [1, \infty)$ is defined by
$$
\nu_{\mathrm{norm}}(p, [w]):=\frac{1}{\sfd_{\mathrm{BW}}(p, \Delta_{n, \underline{d}, [w]})},$$
where $\sfd_{\mathrm{BW}}$ denotes the metric induced by the Bombieri--Weyl structure, see \cref{sec:BW}.
The normalized condition number can be used to turn $V\setminus \Sigma'$ into a \emph{Lipschitz} Riemannian manifold as follows. First, we denote by $g_V$ the Riemannian structure on $V\subset \mathrm{P}(H_{n, \underline{d}})\times \CP^n$ induced by the product Riemannian structure $g_{\mathrm{BW}}\times g_{\mathrm{FS}}.$ Then, for every absolutely continuous curve $\gamma_t=(p_t, z_t):I\to V\setminus \Sigma'$ we define its \emph{condition length} by
$$L_{\mathrm{cond}}(\gamma_t):=\int_{I}\nu_{\mathrm{norm}}(p_t, z_t)\cdot \|\dot\gamma_t\|_V\dd t.$$
Taking the infimum of the length over all absolutely continuous curves joining two points endows $V\setminus \Sigma'$ with a metric structure, called the \emph{condition metric}. We call a curve $\gamma_t$ in $V\setminus \Sigma'$ a condition geodesic if it minimizes the condition length between any two of its points.
Denoting by $D:=\max\{d_1, \ldots, d_n\}$, in \cite{CB6} it is shown that $O ( d^{ 3 / 2} L_\mathrm{cond}(\gamma))$ Newton steps are sufficient to continue the zero $z_0$ from $p_0$ to $p_1$ along $\gamma_t$. 

In \cite{BDM1, BDM2} it is asked whether the function $t\mapsto \log(\nu_{\mathrm{norm}}(\gamma_t))$ is convex along a condition geodesic. An affirmative answer to this question would imply that, for any condition geodesic $\gamma_t$, one has  $L_\mathrm{cond}(\gamma_t)\leq L_V(\gamma_t)\max\{\nu_{\mathrm{norm}}(p_0, z_0), \nu_{\mathrm{norm}}(p_1 , z_1)\},$ where $L_V(\gamma_t)$ is the length of $\gamma_t$ in the Riemannian metric of $V$ defined above. This problem is also formulated as an open problem in \cite[Problem 14]{burgisser}.

Here we propose a solution to this problem in the case $n=1$, replacing $\gamma_t$ with a lift of a Wasserstein geodesic. We expect a similar result to be true in the case $n>1$, but this will be the subject of future work.

\begin{theorem}\label{thm:P14}For every $d\in \N$ there exist $\beta_1, \beta_2, \beta_3, \beta_4>0$ such that, if $p_t:I\to P_{1,d}$ is a $W_2$--geodesic between $p_0, p_1\in P_{1,d}\setminus \Delta_{1,d}$ then the following is true:
\begin{enumerate}
\item for every $t\in I$ 
$$
\mathrm{dist}(p_t, \Delta_{1,d})^2\geq \beta_1\min\{\mathrm{dist}(p_0, \Delta_{1,d}), \mathrm{dist}(p_1, \Delta_{1,d})\}^{\beta_2}.
$$
In particular, $P_{1,d}\setminus \Delta_{1,d}$ is $W_2$--geodesically convex;
\item if $\gamma:I\to V\setminus \Sigma'$ is a lift of $p_t$, i.e.  $\gamma(t)=(p_t, z_t)$, then:
$$
L_\mathrm{cond}(\gamma)\leq \beta_3 \cdot L_{V}(\gamma)\cdot \max\{\nu_{\mathrm{norm}}(p_0, z_0), \nu_{\mathrm{norm}}( p_1, z_1)\}^{\beta_4}.
$$

\end{enumerate}
\end{theorem}
Before proving the theorem, we need to introduce a couple more ingredients. For the ease of notation, we simply denote by $p\in P_{1,d}$ the class of a polynomial and by $w\in \CP^1$ the class of a point (i.e. we omit the square brackets in the notation). Consider the following two functions $\alpha_1, \alpha_2:P_{1,d}\times \CP^1\to [0, \infty)$
$$
\alpha_1(p, w):=\nu_{\mathrm{norm}}(p, w)^{-1}\quad \textrm{and}\quad \alpha_2(p,w):=\min_{p(z_i)=p(z_j)=0, z_i\neq z_j}\sfd_{\mathrm{FS}}(w, z_i)^2+\sfd_{\mathrm{FS}}(w, z_j)^2.$$
Using the semialgebraic homeomorphism $P_{1,d}\simeq \mathrm{SP}^{d}(\CP^1)$, it is easy to see that $\alpha_2$ (which a priori is defined only $P_{1, d}\setminus \Delta_{1,d}$) is actually continuous on $P_{1,d}$. 

It will be also convenient to consider the
following  functions 
$\alpha_3, \alpha_4:P_{1,d} \to [0, \infty)$
$$
\alpha_3(p):=\mathrm{dist}(p_t, \Delta_{1,d})^2
\quad \textrm{and}\quad \alpha_4(p):=\min_{p(z_i)=p(z_j)=0, z_i\neq z_j}\sfd_{\mathrm{FS}}(z_i, z_j)^2.
$$
Functions 
$\alpha_3$ and $\alpha_4$ are the analogous of 
$\alpha_1$ and $\alpha_2$ where we do not fix a specified point $w \in \CP^1$.
As above, 
$\alpha_4$ (which a priori is defined only $P_{1, d}\setminus \Delta_{1,d}$) is actually continuous on $P_{1,d}$.

The key tools for the proof of \cref{thm:P14} are \cref{lemma:compare} and \cref{propo:quasiconc} below. 
The first one relates the function $\alpha_1$ to $\alpha_2$ and, in the same spirit, 
the function $\alpha_3$ to $\alpha_4$.
The second one studies concavity properties of $\alpha_2$ and $\alpha_4$ -- this is where optimal transport plays the key role giving the main estimate.

\begin{lemma}\label{lemma:compare}
There exists $c_1, c_2, c_3, c_4>0$ such that
$$
\alpha_2^{c_2}\leq c_1 \alpha_1\quad \textrm{and}\quad \alpha_1^{c_3}\leq c_4\alpha_2.
$$
Analogously, 
there exists $c_5, c_6, c_7, c_8>0$ such that
$$
\alpha_4^{c_6}\leq c_5 \alpha_3\quad \textrm{and}\quad \alpha_3^{c_7}\leq c_8\alpha_4.
$$
\end{lemma}
\begin{proof}All the functions 
$\alpha_1, \alpha_2, \alpha_3, \alpha_4$ are continuous and semialgebraic. 
Moreover, $\alpha_1$ and $\alpha_2$ have the same zero sets, and are defined on the compact set $P_{1,d}\times \CP^1.$ Therefore the first result follows from Lojasiewicz's inequality.
Analogously, 
$\alpha_3$ and $\alpha_4$ have the same zero sets, and are defined on the compact set $P_{1,d}$ hence 
Lojasiewicz's inequality proves the second claim.
\end{proof}

\begin{proposition}\label{propo:quasiconc}The functions $\alpha_2$ and $\alpha_4$ are ``quasi--concave along $W_2$--geodesics'', i.e. given  a curve $\gamma_t=(p_t, z_t):I\to V$ with $p_t:I\to P_{1,d}$ a $W_2$--geodesic, then for every $t\in [0,1]$ 
$$
\alpha_2(p_t, z_t)\geq (1-t)^2\alpha_2(p_0, z_0)+t^2\alpha_2(p_1, z_1),
$$
and 
$$
\alpha_4(p_t)\geq (1-t)^2\alpha_4(p_0)+t^2\alpha_4(p_1).
$$
\end{proposition}

\begin{proof}For every $t\in I$, let $z_i(t)$ be a zero of $p_t$. 
Since $z_t$ is also a zero of $p_t$
and $p_t$ is a $W_2$-geodesic, the optimality 
of the coupling gives the following estimate
 (see \cite[Eq. (8.45)]{villani:oldandnew} and \cite[Remark 2.1]{figallijuillet})
$$
\sfd_{\mathrm{FS}}(z_t, z_i(t))^2\geq (1-t)^2\sfd_{\mathrm{FS}}(z_0, z_i(0))^2+t^2\sfd_{\mathrm{FS}}(z_1,z_i(1))^2.
$$
For every $t\in I$, let now $z_i(t), z_j(t)$ be two zeros of $p_t$ such that
$$\alpha_2(p_t, z_t)=\sfd_{\mathrm{FS}}(z_t, z_i(t))^2+\sfd_{\mathrm{FS}}(z_t, z_j(t))^2.$$
Then
\begin{align}\alpha_2(p_t, z_t)&=\sfd_{\mathrm{FS}}(z_t, z_i(t))^2+\sfd_{\mathrm{FS}}(z_t, z_j(t))^2\\
&\geq (1-t)^2\left(\sfd_{\mathrm{FS}}(z_0, z_i(0))^2+\sfd_{\mathrm{FS}}(z_0, z_j(0))^2\right)+t^2\left(\sfd_{\mathrm{FS}}(z_1,z_i(1))^2+\sfd_{\mathrm{FS}}(z_0, z_j(0))^2\right)\\
&\geq (1-t)^2\min_{i\neq j}\left(\sfd_{\mathrm{FS}}(z_0, z_i(0))^2+\sfd_{\mathrm{FS}}(z_0, z_j(0))^2\right)\\
&\quad+t^2\min_{i\neq j}\left(\sfd_{\mathrm{FS}}(z_1,z_i(1))^2+\sfd_{\mathrm{FS}}(z_0, z_j(0))^2\right)\\
&=(1-t)^2\alpha_1(p_0, z_0)+t^2\alpha_2(p_1, z_1).
\end{align}
The same argument works also for $\alpha_4$.
\end{proof}

\begin{proof}[Proof of \cref{thm:P14}]
For the first part of the statement follows by direct application of 
\cref{lemma:compare} and
\cref{propo:quasiconc}:
\begin{align}
(\alpha_3(p_t) c_5)^{1/c_6}
&\geq  \alpha_4(p_t) \\
& \geq (1-t)^2\alpha_4(p_0)+t^2\alpha_4(p_1) \\
&\geq   \frac{1}{c_8}\left((1-t)^2\alpha_3(p_0)^{c_7}+t^2\alpha_3(p_1)^{c_7}\right).
\end{align}

For the second part of the statement, we observe that, by \cref{lemma:compare}, 
$$
\nu_{\mathrm{norm}}(p_t, z_t)=\alpha_1(p_t, z_t)^{-1}\leq c_1\alpha_2(p_t, z_t)^{-c_2}.
$$
Using \cref{propo:quasiconc}, and the convexity of $x\mapsto 1/x$, we get
$$\alpha_2(p_t, z_t)^{-1}\leq \left(\alpha_2(p_0, z_0)^{-1}+\alpha_2(p_1, z_1)^{-1}\right),$$
which in particular gives
$$\nu_{\mathrm{norm}}(p_t, z_t)\leq c_1\left(\alpha_2(p_0, z_0)^{-1}+\alpha_2(p_1, z_1)^{-1}\right)^{c_2}.$$
Using again \cref{lemma:compare}, we get
\begin{align} c_1\left(\alpha_2(p_0, z_0)^{-1}+\alpha_2(p_1, z_1)^{-1}\right)^{c_2}&\leq c_1\left(c_4\alpha_1(p_0, z_0)^{-c_3}+c_4\alpha_1(p_1, z_1)^{-c_3}\right)^{c_2}\\&\leq (c_1c_4^{c_2})\max\{\alpha_1(p_0, z_0)^{-1}, \alpha_1(p_1, z_1)^{-1}\}^{c_3c_2}\\
&=\beta_1\max\{\nu_{\mathrm{norm}}(p_0, z_0), \nu_{\mathrm{norm}}(p_1, z_1)\}^{\beta_2}.\end{align}
This proves that along a $W_2$--geodesic:
$$ \nu_{\mathrm{norm}}(p_t, z_t)\leq \beta_1\max\{\nu_{\mathrm{norm}}(p_0, z_0), \nu_{\mathrm{norm}}(p_1, z_1)\}^{\beta_2}$$
and therefore,
$$L_{\mathrm{cond}}(\gamma_t)=\int_{I}\nu_{\mathrm{norm}}(p_t, z_t)\cdot \|\dot\gamma_t\|_V\dd t\leq \beta_1\max\{\nu_{\mathrm{norm}}(p_0, z_0), \nu_{\mathrm{norm}}(p_1, z_1)\}^{\beta_2}\int_{I} \|\dot\gamma_t\|_V\dd t.$$
\end{proof}

\begin{remark}
Two comments are in order. 
The first is that 
from the proof of the first part claim of 
\cref{thm:P14} it follows a slightly stronger 
statement 
than the convexity of $P_{1,d}\setminus \Delta_{1,d}$.
Indeed the quasi-concavity of $\alpha_4$ yields
that if $p_t$ is a $W_2$-geodesic with 
$p_0$ or $p_1$ not in $\Delta_{1,d}$ then 
$p_t \notin \Delta_{1,d}$ for all $t \in (0,1)$.

The second one is that quasi-concavity implies 
the following inequality that we might call
quasi-log-concavity: from
$$
\alpha(p_t)
\geq (1-t)^2\alpha(p_0)+t^2\alpha(p_1),
$$
one deduces from the concavity of the log
\begin{align}
\log(\alpha(p_t)) \geq  (1-t) \log(\alpha(p_0))   
+ t \log(\alpha(p_1)) + (1-t)\log(1-t) + t\log(t).
\end{align}
Applying this inequality to $\alpha_2$ and invoking 
\cref{lemma:compare} one obtains that there exists 
universal constant $\widetilde c_1, \widetilde c_2, \widetilde c_3$, such that 
\begin{align}
\widetilde c_1 \log(\nu_{norm}(\gamma_t)) 
\leq  ~& 
\widetilde c_2 (1-t)  \log(\nu_{norm}(\gamma_0))   
+ \widetilde c_3 t  \log(\nu_{norm}(\gamma_1))  \\
&+ (1-t)\log\left(\frac{1}{1-t}\right) + t\log\left(\frac{1}{t}\right),
\end{align}
for any 
$\gamma_t=(p_t, z_t):I\to V\setminus \Sigma'$ 
lift of a $W_2$-geodesic $p_t : I \to P_{1,d}$.
\end{remark}

\addtocontents{toc}{\protect\setcounter{tocdepth}{1}}
\renewcommand{\thesection}{A}
\section{Appendix - Projective geometry} \label{Appendix A} 

\setcounter{teorema}{0}
\setcounter{equation}{0}  
\setcounter{subsection}{0}
\subsection{The Fubini--Study structure}
In this section we recall the Fubini--Study construction, which is a procedure to endow the projective space $\CP^n=\mathrm{P}(\C^{n+1})$ with a Riemannian Hermitian structure, once a Hermitian product is chosen on $\C^{n+1}$. We refer the reader to \cite{GH} for more details.

The construction works for any Hermitian product $h$  on $\C^{n+1}$, but here we consider simply the standard one $h_{\mathrm{std}}(v,w)=\langle v, w\rangle:=\sum v_j\overline w_j$. Denote by $S^{2n+1}\hookrightarrow \C^{n+1}$ the unit sphere with respect to the induced scalar product and by $\pi:S^{2n+1}\to \CP^n$ the quotient map. Given $b\in S^{2n+1}$ we can write:
\beq T_bS^{2n+1}=b^\perp\oplus\mathrm{span}_{\R}\{i b\},\eeq
where this decomposition is orthogonal with respect to the scalar product induced by the Hermitian structure. 

Recall also that given $[b]\in \CP^n$ we have 
\beq T_{[b]}\CP^n\simeq \mathrm{Hom}(\mathrm{span}_\C\{b\}, \C^{n+1}/\mathrm{span}_{\C}\{b\})\simeq b^\perp.\eeq
In particular, picking the representative $b\in S^{2n+1}$, for two tangent vectors $v_1, v_2\in T_{[b]}\CP^n\simeq b^\perp$ we can define
\beq h_{\mathrm{FS}, [b]}(v_1, v_2):=\langle v_1, v_2\rangle. \eeq 
This defines a Hermitian structure on each tangent space to $\CP^n$, which we call the \emph{Fubini--Study} structure (hence the subscript ``FS'' in the notation $h_{\mathrm{FS}}$). The real part of $h_{\mathrm{FS}}$ is a Riemannian metric $g_{\mathrm{FS}}:=\mathrm{Re}(h_{\mathrm{FS}})$ that we call the \emph{Fubini--Study metric} and its imaginary part $\omega_{\mathrm{FS}}:=\mathrm{Im}(h_{\mathrm{FS}})$ is a K\"ahler form that we call the \emph{Fubini--Study K\"ahler form}. With this definition the map $\pi:S^{2n+1}\to \CP^n$ becomes a Riemannian submersion. Moreover, for every $R\in \mathrm{U}(n+1)$ the map $[v]\mapsto [Rv]$ is an isometry of $(\CP^n, h_{\mathrm{FS}})$.

Let now $[b]\in \CP^n$ with $b\in S^{2n+1}$. We define a parametrization $\phi_b:\C^{n}\to U_b$ (i.e. the inverse of a chart) of a neighborhood $U_b\subset \CP^n$ of $[b]$. First, let $U_b:=\{[z]\in \CP^n\,|\, \langle z, b\rangle\neq 0\}$ (i.e. the complement of the hyperplane $b^{\perp}$). Let also $R\in \mathrm{U}(n+1)$ be such that $Re_0=b$ and define 
\beq\label{eq:parab} \phi_b(v):=\left[R\left(\begin{matrix}1\\v\end{matrix}\right)\right].\eeq
Notice that $\phi_b(0)=[b]$. A useful property of this map is that $D_0\psi_b$ gives an isometry 
\beq\label{eq:isob} D_0\phi_b:(\C^n, h_{\mathrm{std}})\stackrel{\simeq}{\longrightarrow}(T_{[b]}\CP^n, h_{\mathrm{FS}, [b]}).\eeq
\subsection{Vector bundles}In this section we explain why homogeneous polynomials are sections of an appropriate vector bundle (and not functions) on the projective space. Let us recall some basics notions first.

Recall that if $\pi:L\to M$  is a rank--$k$ vector bundle on $M$, then there exists an open cover $\{U_{i}\}_{i\in I}$ for $M$ and, for every $i\in I$, a smooth map $\psi_i:L|_{U_i}=\pi^{-1}(U_i)\to U_{i}\times \R^k$ such that
\begin{enumerate}
\item $\pi=p_1\circ \psi_i$, where $p_1:U_i\times \R^k$ denotes the projection on the first factor;
\item for every $i,j\in I$ such that $U_i\cap U_j\neq \emptyset$, and for every $(x, v)\in (U_i\cap U_j)\times \R^k$ we have 
\beq \psi_i\circ\psi_j^{-1}(x,v)=(x, g_{ij}(x)v),\eeq
where $g_{ij}(x)\in \mathrm{GL}(\R^k).$
\end{enumerate}
The family $\{(U_i, \psi_i)\}_{i\in I}$ is called a \emph{vector bundle atlas} and the family $\{g_{ij}\}_{i,j\in I}$ is called the \emph{cocycle} of the vector bundle atlas. The cocycle satisfies $g_{ij}g_{jk}=g_{ik}$ (on the common domains of definition) and $g_{ii}\equiv \mathbf{1}.$ It turns out that in order to construct a vector bundle it is enough to assign its cocycle. A \emph{complex} rank--$k$ vector bundle is obtained by requiring, in the above construction, that $g_{ij}(x)\in \mathrm{GL}(\C^k)$.

We denote by $\Gamma(L, U)$ the space of sections of $L$ over $U\subseteq M$. Let now $\sigma\in \Gamma(L, M)$ be a global section of $L$ (a real or complex vector bundle) and $\{(U_i, \psi_i)\}_{i\in I}$ be a vector bundle atlas (real or complex). Then, for every $i\in I$, the map $\sigma_i:=\psi_i\circ\sigma|_{U_i}\in \Gamma(L, U_i)$, being a section, has the form
\beq \sigma_i(x)=(x, \widetilde{\sigma}_i(x)),\eeq
for some map $\widetilde{\sigma}_i:U_i\to \R^k$. By definition of the cocyle, the family of maps $\{\widetilde{\sigma}_i\}_{i\in I}$ satisfies
\beq\label{transitionsection} \widetilde{\sigma}_i(x)=g_{ij}(x)\widetilde{\sigma}_{j}(x)\quad \forall x\in U_i\cap U_j.\eeq

Viceversa, given a family of smooth functions $\{\widetilde{\sigma}_i\}_{i\in I}$ satisfying \cref{transitionsection}, there exists a section $\sigma:M\to L$ such that $\psi_i(\sigma(x))=(x, \widetilde{\sigma}_{i}(x))$ for all $i\in I$ and $x\in U_i$.

\subsection{Homogeneous polynomials as sections of a vector bundle}\label{sec:homosec}Let us move now to homogeneous polynomials. Consider the open cover $\{U_i\}_{i=0, \ldots, n}$ of $\CP^n$ given by $U_i:=\{[z_0, \ldots, z_n]\,|\, z_i\neq 0\}$. Given a homogeneous polynomial $p\in \C[z_0, \ldots, z_n]_{(d)}$, we can define the zero set 
\beq 
Z(p):=\{[z_0, \ldots, z_n]\,|\, p(z)=0\} \subset \CP^n.
\eeq 
This makes sense because $p$ vanishes at $z\in \C^{n+1}\setminus \{0\}$ if and only if it vanishes at $\lambda z$, for every $\lambda\neq 0$, while the value of $p$ at a point $[z]\in \CP^n$ is not well defined. On the other hand, for every $i=0, \ldots, n$ we can consider the function $\widetilde{\sigma}_{p,i}:U_i\to \C$ given by
\beq \label{eq:transC}\widetilde{\sigma}_{p,i}([z]):=p(z_0/z_i, \ldots, z_n/z_i).\eeq
Notice that $Z(p)\cap U_i=\widetilde{\sigma}_{p,i}^{-1}(0)$. Moreover, for every $i\neq j$ and $[z]\in U_i\cap U_j$ we have
\beq \widetilde{\sigma}_{p,i}([z])=(z_j/z_i)^d\widetilde{\sigma}_{p,j}([z]).\eeq
Now, observe that the family of functions $\{g_{ij}([z]):=(z_j/z_i)^d\}_{ij},$ which are defined on $U_i\cap U_j$ and take values in $\mathrm{GL}(\C)=\C^*$, satisfy the cocyle condition. Therefore, by what we explained above, they define a complex line bundle that algebraic geometers denote by $\mathcal{O}_{\CP^n}(d)$; we will use the notation $L_{n,d}$ for this bundle. Because of \eqref{eq:transC}, the family of functions $\{\widetilde{\sigma}_{p,i}\}_{i=0, \ldots, n}$ glue together to a section of $L_{n,d}$ which we denote by $\sigma_p$. 

In this way we have constructed a linear map $\sigma_{(\cdot)}:\C[z_0, \ldots, z_n]_{(d)}\to \Gamma(L_{n,d}, \CP^n)$ that associates to a polynomial $p$ the section $\sigma_p$ which in coordinates is written as in \eqref{eq:transC}. One can show that this map is injective and that its image coincides with the set of global holomorphic sections of $L_{n,d}$. Moreover, denoting by $A\hookrightarrow L$ the zero section, we have $\sigma_p^{-1}(A)=Z(p)$. 

In particular, for every $i=0, \ldots, n$ the polynomial $z_i^d\in \C[z_0, \ldots, z_n]_{(d)}$ corresponds to a section $\sigma_{z_i^d}$ that, by construction \eqref{eq:transC} satisfies:
\beq\label{eq:trivib} \widetilde{\sigma}_{z_i^d, i}([z_0, \ldots, z_n])=1.\eeq
Notice that, using the standard Hermitian structure, we can write $z_i^d=\langle z, e_i\rangle^d$, where $e_i=(0, \ldots,0,1,0 \ldots, 0)\in \C^{n+1}$ (the ``$1$`` is in position $i$). In other words, $\sigma_{z_i^d}$ is a section of $L_{n,d}$ that does not vanish on a neighborhood $U_i=\{z_i\neq 0\}$ of $[e_i]$. We can build similar sections for every other point $[b]\in \CP^n$ as follows. Let $R\in \mathrm{GL}(n+1, \C)$ be such that $Re_0=b$ and consider the polynomial $p_{b}(z):=\langle z, b\rangle^d$. Since $p_b(b)=\|b\|^2\neq 0$, the section 
\beq\label{eq:secb}\sigma_{b}:=\sigma_{p_b}\eeq does not vanish at $[b]$ and, consequently, it does not vanish on the neighborhood $U_{b}:=\{\langle z, b\rangle\neq 0\}$ of $[b]$. We can use this section to produce a trivialization of $L_{n,d}$ on $U_{b}$, defining $\psi_{b}: L_{n,d}|_{U_b}\to U_b\times \C$ by
\beq \psi_b(w):=\left(\pi(w), \frac{w}{\sigma_{b}(\pi(w))}\right).\eeq
Observe that, given any other section $\sigma=\sigma_p\in \Gamma(L_{n,d}, U_b)$, with $p\in \C[z_0, \ldots, z_n]_{(d)}$, we have:
\beq\label{eq:sectionincoord} \psi_b(\sigma([z]))=\left([z], \frac{\sigma_{p}([z])}{\sigma_{b}([z])}\right)=\left([z], \frac{p(z)}{\langle z, b\rangle^d}\right).\eeq
\subsection{The Bombieri--Weyl hermitian structure}\label{sec:BW}
Consider now the \emph{Bombieri--Weyl} hermitian product on $\C[z_0, \ldots, z_n]_{(d)}$, defined by:
\beq \langle p_1, p_2\rangle_{\mathrm{BW}}:=\frac{1}{\pi^{n+1}}\int_{\C^{n+1}}p_1(z)\overline{p_2(z)}e^{-\|z\|^2}\mathrm{d} z,\eeq
where $\mathrm{d}z:=(i/2)^{n+1}\mathrm{d}z_0\mathrm{d}\overline{z_0}\cdots \mathrm{d}z_n\mathrm{d}\overline{z_n}$ is the Lebesgue measure.
It is not difficult to show that the Bombieri--Weyl hermitian structure is invariant under the action of $U(n+1)$ by change of variables. Moreover, since this action is irreducible (over $\C$), Schur's Lemma implies that this is the unique hermitian structure with this property (up to multiples). A hermitian orhtonormal basis for the Bombieri--Weyl product is given by
\beq \label{eq:BWb}\left\{\left(\frac{d!}{\alpha_0!\cdots \alpha_n!}\right)^{\frac{1}{2}}z_0^{\alpha_0}\cdots z_n^{\alpha_n}\right\}_{\{|\alpha|=d\}}.\eeq
If we apply now the Fubini--Study construction above to the space $\mathrm{P}(\C[z_0, \ldots, z_n]_{(d)})$ we obtain a hermitian structure on it, which we still call the Bombieri--Weyl structure. The Unitary group $U(n+1)$ acts on the projective space $\mathrm{P}(\C[z_0, \ldots, z_n]_{(d)})$ by change of variables and the action is by isometries.

\addtocontents{toc}{\protect\setcounter{tocdepth}{2}}

\bibliographystyle{acm}
\bibliography{literature.bib}

\end{document}